\documentclass[12pt]{ociamthesis}
\usepackage{graphicx}				
\usepackage{mathrsfs}									
\usepackage{amssymb}
\usepackage{amsmath}
\usepackage{amsthm}
\usepackage{mathtools}
\usepackage{setspace}
\usepackage{hyperref}
\usepackage{tikz-cd}
\usepackage{url}

\usepackage[OT2,T1]{fontenc}
\DeclareSymbolFont{cyrletters}{OT2}{wncyr}{m}{n}
\DeclareMathSymbol{\Sha}{\mathalpha}{cyrletters}{"58}

\newtheorem{thm}{Theorem}
\newtheorem*{thm*}{Theorem}
\numberwithin{thm}{section}
\newtheorem*{exmp*}{Example}
\newtheorem{cor}[thm]{Corollary}
\newtheorem{prop}[thm]{Proposition}
\newtheorem{lem}[thm]{Lemma}

\newtheorem{quest}[thm]{Question}

\newtheorem{assumption}[thm]{Assumption}
\newtheorem{defn}[thm]{Definition}

\newtheorem{exmp}[thm]{Example}

\newtheorem{fact}[thm]{Fact}
\newtheorem{defprop}[thm]{Defining Property}

\theoremstyle{remark}
\newtheorem{rem}{Remark}
\newtheorem*{pfsketch}{Proof Sketch}

\usepackage{mathtools}
\newcommand{\defeq}{\vcentcolon=}

\DeclareMathOperator{\Hom}{Hom}

\DeclareMathOperator{\Sym}{Sym}
\DeclareMathOperator{\Spec}{Spec}
\DeclareMathOperator{\Spf}{Spf}
\DeclareMathOperator{\Spm}{Spm}
\DeclareMathOperator{\Aut}{Aut}
\DeclareMathOperator{\im}{im}

\newcommand{\ts}{\textsuperscript}

\title{The Chabauty--Kim Method for Relative Completions}
\author{Noam Kantor}
\college{Oriel College}
\degree{DPhil in Mathematics}
\degreedate{Michaelmas 2019}

\begin{document}

\maketitle

\begin{dedication}
This thesis is dedicated to\\
those who struggled to find food\\
while I struggled to find references\\
\end{dedication}

\begin{acknowledgements}
When I was a beginning student, I saw number theory as a fractured subject. There seemed to be no unified motivation for the problems and no unified method for their solutions. In fact, the saying that number theory is ``the queen of mathematics'' seemed like less of a compliment and more of an assertion that number theory did not know how to stand on its own.

Well, I have grown a lot since those days: I now see number theory -- or at least arithmetic geometry -- as a beautifully coherent collection of concepts and methods. That coherence has only grown in the last few decades, thanks to number theorists, geometers, logicians, combinatorialists, topologists, and analysts whose work has either purposefully or accidentally contributed to arithmetic geometry. I am, then, incredibly thankful to those who put in sweat and tears before me so that I had a foundation upon which to work. Special thanks go to Jon Pridham, Kobi Kremnitzer, and Damian R\"{o}ssler, who read earlier versions of this work in confirmation and transfer and provided much-needed feedback. I am particularly grateful for those who have seen mathematics as part of a larger humane project, and to those who seek to make math a more open and accepting field which sees the world's problems as its own problems.

My growth into the mathematician I am today was initiated by Professor David Zureick-Brown, my undergraduate advisor, who opened my eyes to the world of number theory and especially the method of Chabauty.

I would not have finished this process without the many mathematical friends and colleagues who were willing to sit down and discuss math or anything-but-math: Alexander Gietelnik-Oldenziel; Alex Saad; Jay Swar; Nadav Gropper; Alyosha Latyntsev; Ma Luo (for teaching me so much about Hodge theory of relative completions); Netan Dogra (who read through an earlier draft of this, was there more than once when I thought things were crashing down, and has taught me a lot of good math); Jan Vonk (who answered many questions about arithmetic geometry); Carl Wang--Erickson (for email exchanges explaining his work that went far above and beyond the call of duty, and for his absolutely clear writing.)

And I am thankful for the non-mathematical friends who have inspired me: my flatmates (for more reasons than I can count); many sensitive and compassionate minds of the Marshall scholarship; the ever-knowledgable companions at Oxford JSoc; singers in Brasenose Chapel Choir; Ben St.\ Clair and Victoria Mousley (for always providing a home away from home); David Elitzer and Rhea Stark (for listening to my troubles with open ears and open hearts); Josh Lappen (who is a true mensch); and others who supported me more than I will ever understand.

My wonderful family has been there in the ups and downs of the last few months, and I will always be grateful for that.

\thispagestyle{empty}

Above all, it is impossible to express the fullness of the gratitude I owe to my advisor, Professor Minhyong Kim, for pushing me to write this DPhil when I was just a few months from submitting my MSc thesis. His patience, mathematical wisdom, and broad set of interests that seem to span every subject known to humanity have inspired his students and grand-students, and I am privileged to have had his time and guidance. I share the sentiment that he describes in the paper that announced non-abelian Chabauty: ``I seem to have accumulated quite a debt of gratitude even in writing this simple paper...''
\end{acknowledgements}

\begin{abstract}
In this thesis we develop a Chabauty--Kim theory for the relative completion of motivic fundamental groups, including Selmer stacks and moduli spaces of admissible torsors for the relative completion of the de Rham fundamental group. On one hand, this work generalizes results of Kim (and therefore Chabauty) in the unipotent case by adding a reductive quotient of the fundamental group. From this perspective, the addition of a reductive part allows one to apply Chabauty-type methods to fundamental groups with trivial unipotent completion, such as $SL_2(\mathbb{Z})$.

On the other hand, the unipotent part provides a natural extension of the recent work of Lawrence and Venkatesh. We show that their concern with the centralizer of Frobenius goes away as one moves up the unipotent tower and away from the reductive world of flag varieties and the Gauss--Manin connection. One is tempted to hope that the relative completion will provide a unified proof of Mordell's conjecture that takes advantage of the two methods. Toward this end, we apply our work to the Legendre family on the projective line minus three points, a particular example where the method of Lawrence and Venkatesh fails.
\end{abstract}

\baselineskip=18pt plus1pt

\begin{romanpages}
\tableofcontents
\end{romanpages}

\chapter{Introduction}
This thesis centers around the Chabauty--Kim method, a method for proving Diophantine finiteness using $p$-adic analysis. The Chabauty-Kim method, which belongs in fact to many mathematicians, is a fundamental example of the unity of theory and method that have permeated arithmetic geometry in the last half-century or so. The unified theory lies in the concept of a motive and its realizations. The unified method, on the other hand, can broadly be termed the ``box principle.'' This principle refers to the fact that finiteness results in arithmetic geometry arise when one is able to place the desired object in a more structured box. The Chabauty-Kim method uses this principle to show finiteness of rational points, using Selmer varieties as a box.

In this thesis, we advance Chabauty--Kim theory by constructing a Chabauty--Kim diagram, including Selmer stacks and de Rham moduli spaces, for relative Malcev completions of fundamental groups as studied by Hain, Olsson, Pridham, and others. This advance is an essential next step in the Chabauty-Kim theory because it injects some non-unipotency into a theory which has thus far only worked with the unipotent fundamental group.

More broadly, this work can be seen as a first inroads in the program to relate the Chabauty--Kim method to the recent work of Lawrence and Venkatesh \cite{lawrence2018diophantine} reproving Faltings' Theorem. On the technical side, there are many theorems left to prove, such as broader representability and algebraicity of Selmer stacks; and there are restrictions that can likely be removed, such as the simple-connectedness of algebraic monodromy. Finally, there are certainly many more connections to the work of Lawrence and Venkatesh that need exploring, for example the question of applying our method directly to Kodaira--Parshin families. Despite this to-do list, we see many exciting developments in what we have done here and hope this work inspires others to leave their unipotent or reductive silos. Experts in Chabauty--Kim may want to skip to our main example (\ref{projline}) after Section \ref{mainres} to see the theory in action.

\section{The method of Chabauty, Coleman, and Kim}
To give context to our main results, we now explain Kim's approach to Diophantine finiteness -- itself a vast generalization of the ingenious method of Chabauty that Coleman used in the 1980's to give bounds on rational points on hyperelliptic curves of low Mordell-Weil rank \cite{coleman1985effective}. The canonical exposition of the work of Chabauty and Coleman is the essay by McCallum and Poonen \cite{mccallum2007method}. Even more recently, the method of Chabauty and Coleman has been combined with tropical geometry to give uniform bounds on rational points on some curves of low Mordell-Weil rank. For more information on this train of ideas, see the work \cite{katz2016uniform} of Katz, Rabinoff, and Zureick-Brown.

Finally, we highlight the recent progress on quadratic Chabauty by Balakrishnan, Dogra, M{\"u}ller, Besser, and others. Quadratic Chabauty is an explicit realization of Kim's method in the first case that classical Chabauty cannot deal with \cite{balakrishnan2018quadratic}.

For all of the materials that follow, we recommend above all the original papers of Kim and the work of Deligne that inspired it \cite{deligne1989groupe}.

Consider a curve $X$ over a number field $F$, such that $X$ has a smooth model $\mathcal{X}$ over the ring $\mathcal{O}_{F,\Sigma}$ of $\Sigma$-integers of $F$ for some finite set of places $\Sigma$. Fix a basepoint $b \in \mathcal{X}(\mathcal{O}_{F, \Sigma})$ if one is available. (The Chabauty--Kim method says nothing about whether a basepoint exists, though the nonemptiness of the set of rational points is closely related to fundamental groups via Grothendieck's section conjecture; cf. for example \cite{harari2012descent}.)

Finally, let $v$ be a place not in $\Sigma$ (i.e. at which $X$ has good reduction) with $p$ the rational prime under $v$. Let $T$ be the set of places of $F$ that includes all places in $\Sigma$ and all primes dividing $p$. Kim's generalization of Chabauty's method centers on the commutative diagram (now called ``Kim's Cutter'')
\[
\begin{tikzcd}
\mathcal{X}(\mathcal{O}_{F,\Sigma}) \arrow{r} \arrow{d}{x \mapsto P^{et}_{b,x}} & \mathcal{X}(\mathcal{O}_{F_v}) \arrow{d}{x \mapsto P^{et}_{b,x}} \arrow{dr}{\int \colon P^{dR}_{b,x}} \\
H^1_f(G_{F,T}, \pi^{\text{\'et}}_{1,n}(X))  \arrow{r}{res} & H^1_f(G_{F_v}, \pi^{\text{\'et}}_{1,n}(X)) \arrow{r}{D_{dR}} & Res^{F_v}_{\mathbb{Q}_p}\left( \pi^{\text{dR}}_{1,n} (X)/F^0 \pi^{\text{dR}}_{1,n} (X) \right)
\end{tikzcd}
\]

There is one such diagram for every natural number $n$, and one hopes that analyzing the diagram for high enough $n$ might yield some sort of finiteness for the set $\mathcal{X}(\mathcal{O}_{F, \Sigma})$. Experience actually shows that with enough skill this diagram can yield explicit equations for $\mathcal{X}(\mathcal{O}_{F, \Sigma})$ inside of $\mathcal{X}(\mathcal{O}_{F_v})$. We will explain the maps in this diagram throughout this work. Suffice it to say that the two cohomology groups are called Selmer varieties, and the quotient on the right hand side might be called the de Rham period domain; the map $D_{dR}$ is the non-abelian version of Bloch and Kato's logarithm (it is a souped-up version of Faltings' $D_{dR}$ functor), and the two vertical maps send a point to the unipotent \'etale path torsor associated to that point.

An important point is that the maps $res$ and $D_{dR}$ are algebraic, while iterated integration is $p$-adic analytic.

The analysis should work as follows if one thinks that $\mathcal{X}(\mathcal{O}_{F, \Sigma})$ ought to be finite:
\begin{description}
\item[1]
Show that the diagonal integration map is Zariski dense in $Res^{F_v}_{\mathbb{Q}_p}\left( \pi^{\text{dR}}_{1,n} (X)/F^0 \pi^{\text{dR}}_{1,n} (X) \right).$
\item[2]
Show that the dimension of $H^1_f(G_{F,T}, \pi^{\text{\'et}}_{1,n}(\bar X))$ is strictly less than that of $$Res^{F_v}_{\mathbb{Q}_p}\left( \pi^{\text{dR}}_{1,n} (X)/F^0 \pi^{\text{dR}}_{1,n} (X) \right),$$ so that there is a nonzero algebraic function $\mathfrak{f}$ vanishing on the image of this global Selmer variety in the de Rham period space.
\item[3]
By the commutativity of the diagram, $\mathfrak{f}$ vanishes on the image of elements of $\mathcal{X}(\mathcal{O}_{F, \Sigma})$ under the integration map, which we may denote $\int$. Then $\mathfrak{f} \circ \int$ is a function on $\mathcal{X}(\mathcal{O}_{F_v})$ that vanishes on $\mathcal{X}(\mathcal{O}_{F, \Sigma})$. Furthermore, $\mathfrak{f} \circ \int$ is nonzero on $X(F_v)$ because $\int$ has Zariski-dense image.
\item[4]
The composition $\mathfrak{f} \circ \int$ is a nonzero $p$-adic analytic function on $\mathcal{X}(\mathcal{O}_{F_v})$, and any such function has finitely many zeros on a $p$-adic disc. Since $X(F)$ is covered by a finite number of discs (namely residue discs), there are only finitely many rational points on $X$. Here we use the fact that $X$ is a curve. There has been a slew of recents results using Ax--Schanuel-type methods to show that $\mathfrak{f} \circ \int$ can cut out an algebraic subset of $\mathcal{X}(\mathcal{O}_{F_v})$, which allows one to apply the method to higher-dimensional varieties \cite{hast2019functional, dogra2019unlikely}. 
\end{description}
By far the most difficult step of this process is the second; it invariably uses deep results on the dimensions of global Galois cohomology groups.

If one is interested in explicit equations for the rational points, one then needs equations for the image of $H^1_f(G_{F,\Sigma}, \pi^{\text{\'et}}_{1,n}(\bar X))$ in $Res^{F_v}_{\mathbb{Q}_p}\left( \pi^{\text{dR}}_{1,n} (X)/F^0 \pi^{\text{dR}}_{1,n} (X) \right)$. In the classical method of Chabauty these equations are found residue-disc-by-residue-disc using linear algebra, but the method becomes significantly more difficult as $n$ increases (i.e. as the situation becomes less linear.) One approach is through Massey products, which provide relations in Galois cohomology. This approach, implemented for example in \cite{balakrishnan2011appendix}, achieves the ultimate desiteratum of an explicit analytic equation for integral points that holds uniformly across all residue discs.

We pause to observe that the use of the unipotent fundamental group comes into almost every stage of the Chabauty-Kim process: iterated integration is the same as solving a unipotent system of differential equations, Selmer varieties have up until now only been defined for a unipotent fundamental group, and the very process of proving bounds on Galois cohomology always involves reducing to the case of linear Galois representations using unipotence. Incorporating non-unipotent information will, therefore, take us seemingly far afield of the classical Chabauty formalism. Yet, as we will see, the non-unipotent version of Chabauty will involve different linear invariants (such as adjoint Galois cohomology) which deepen the theory.

\section{Main results} \label{mainres}
In this thesis we prove the following results. For the entirety of this thesis, $F$ will denote a number field, and $\mathcal{O}_{F, \Sigma}$ the ring of $\Sigma$-integers of $F$ for $\Sigma$ a set of places of $F$. If $X$ is a scheme over a field, $\bar X$ denotes the base change of $X$ to the algebraic closure of the field.

\subsection{\'Etale realizations}
We give a detailed introduction to the relative completion in the main text, but let us say for now that the \'etale realization of the relative completion of the fundamental group of $X$ -- denoted $\pi_1^{rel, \acute{e}t}(\bar X)$ -- is a quotient of the pro-algebraic completion of the geometric \'etale fundamental group of a scheme. It is defined by fixing an \'etale $\mathbb{Q}_p$-local system $\mathcal{L}^{\acute{e}t}$ on $\bar X$, and considering the Tannakian category containing $\mathcal{L}^{\acute{e}t}$ and closed under extension, tensor product, dual, quotient, and direct sum. The relative \'etale fundamental group $\pi_1^{rel,\acute{e}t}(\bar X)$ is the Tannakian fundamental group of this category; when $\mathcal{L}^{\acute{e}t}$ is defined over the base field and suitably ``unramified,'' this fundamental group has an action by the Galois group $G_{F,T}$ of the maximal extension of $F$ unramified outside of some set $T$ which includes the primes of bad reduction for $X$ and $\mathcal{L}^{\acute{e}t}$ and all primes over $p$.

The best way to understand $\pi_1^{\acute{e}t}$ at a first approximation is that it fits into an exact sequence \[1 \to \mathcal{U}^{rel, \acute{e}t} \to \pi_1^{rel,\acute{e}t} \to S^{\acute{e}t} \to 1\] where $\mathcal{U}^{\acute{e}t}$ is unipotent and $S^{\acute{e}t}$ is the algebraic closure of the monodromy of $\mathcal{L}^{\acute{e}t}$. In this sense, it combines the monodromy-theoretic methods of classical period mappings with a unipotent part that is reminiscent of Kim's theory.

Our first result is the construction of cohomology and Selmer stacks for $\pi_1^{\acute{e}t}$, as well as maps from integral points into them assuming certain conditions on $S^{\acute{e}t}$. Here is a precise statement.
\begin{thm*}{(\ref{points})}

Let $\pi \colon \mathcal{Y} \to \mathcal{X}$ be a smooth and proper morphism of schemes over $\mathcal{O}_{F, \Sigma}$, with their generic fibers over $F$ being denoted $X$ and $Y$. Let $\mathcal{L}^{\acute{e}t}$ denote the $q$\ts{th} relative geometric \'etale cohomology of $\bar Y \to \bar X$ with coefficients in $\mathbb{Q}_p$, and $\pi_1^{rel,\acute{e}t}$ the relative unipotent fundamental group of $\bar X$ with respect to $\mathcal{L}^{\acute{e}t}$ and a fixed basepoint $b \in \mathcal{X}(\mathcal{O}_{F, \Sigma})$. Furthermore, let $v$ be a place of $F$ not contained in $\Sigma$, and $p$ the rational prime below $v$. Finally, let $T$ be a set of primes containing $\Sigma$ and all primes of $F$ above $p$.

There exist stacks on the rigid-analytic site, denoted $$H^{1}_{f,\bold{r}}(G_{F_v}, \pi_1^{rel,\acute{e}t})$$ and $$H^{1}_{f,\bold{r}}(G_{F,T}, \pi_1^{rel,\acute{e}t})$$ that parameterize Galois-compatible torsors for $\pi_1^{rel, \acute{e}t}$ which are crystalline at $p$ and have type $\bold{r}$; if the monodromy group $S^{rel, \acute{e}t}$ is semisimple and simply connected then there is a commutative diagram
\[
\begin{tikzcd}
\mathcal{X}(\mathcal{O}_{F, \Sigma}) \arrow{r} \arrow{d}{P_{b,x}^{\acute{e}t}} & \mathcal{X}(\mathcal{O}_{F_v}) \arrow{d}{P_{b,x}^{\acute{e}t}} \\
H^{1}_{f,\bold{r}}(G_{F,T}, \pi_1^{rel,\acute{e}t})(\mathbb{Q}_p) \arrow{r}{res} & H^{1}_{f,\bold{r}}(G_{F_v}, \pi_1^{rel,\acute{e}t})(\mathbb{Q}_p)
\end{tikzcd}
\]

Here $P^{\acute{e}t}$ refers to taking a certain relatively unipotent torsor of paths from $b$ to $x$; $res$ is the restriction of Galois representations; $\bold{r}$ is a triple that indicates the substack corresponding to torsors with $p$-adic Hodge type, residual representation, and in the global case, weight equal to that of the stalks of $\mathcal{L}^{\acute{e}t}$ (see Definition \ref{type}.)

\end{thm*}
\begin{rem}
The monodromy $S^{\acute{e}t}$ will always be semisimple in geometric situations. Its simple-connectedness does not seem universal, but many cases (such as the full symplectic group) satisfy this hypothesis.
\end{rem}

We call the stacks in the theorem Selmer stacks. For those used to the unipotent Selmer variety formalism, we should say here that stackiness enters the picture because our Selmer stacks are something like a combination of Kim's Selmer varieties and moduli stacks of Galois representations.

One hopes to show, as in the work of Kim \cite[Section 1]{kim2005motivic}, that these two stacks and their simpler avatars without Selmer conditions are representable by some geometric object. (They would naturally be representable as rigid-analytic stacks and preferably be algebraizable to algebraic stacks over $\mathbb{Q}_p$.) As a first step toward representability, we prove the following statement, which already uses difficult work of Pottharst on continuous cohomology of families of Galois representations \cite{pottharst2013analytic}.

To state our result, we need the notion of strat-representability: see Definition \ref{stratrep}. The reader should imagine that a stack in groupoids $\mathcal{X}$ is strat-representable if it has a stratification such that each piece is representable as a rigid-analytic stack.
\begin{thm*}{(Theorem \ref{representability})}

Preserve the notations of the previous theorem. The continuous cohomology stack $H^1(G,\pi_1^{rel, \acute{e}t})$ is strat-representable in the category of rigid-analytic stacks, where $G$ is $G_{F_v}$ or $G_{F,T}$.
\end{thm*}
We also explain the challenges that keep us from proving (strat-)representability of Selmer stacks (Remark \ref{challenges} and the discussion preceding it.)

\subsection{de Rham realizations}

One can play the same game with the relative de Rham cohomology $\mathcal{L}^{dR}$ of a smooth and proper family, and from it one obtains in exactly the same way the de Rham relative completion, $\pi_1^{rel, dR}(X \otimes_F F_v)$. Our other strain of results regards the creation of a de Rham period space which classifies torsors for the de Rham relative completion that are ``doubly admissible'' in a certain sense. The word ``doubly'' refers to the fact that they combine Kim's admissibility of torsors with weak admissibility in the sense of $p$-adic Hodge theory. The reader should think of this moduli space as combining the flag-variety approach of Lawrence and Venkatesh with the unipotent Albanese approach of Hain and Kim.

The de Rham moduli space is an open rigid-analytic substack $$\mathcal{M}^{dR} \subseteq [ (\pi_1^{rel, dR})^{\phi=1} \backslash \pi_1^{rel, dR}/F^0\pi_1^{rel, dR} ].$$ The Frobenius-invariant subgroup is intimately related to the headaches that Lawrence and Venkatesh overcome in bounding the centralizer of Frobenius, and our methods show how one might eliminate these headaches. Let us explain.

The subgroup $(\pi_1^{rel, dR})^{\phi=1}$ is a priori huge, since it is a subgroup of a pro-algebraic group. However, a deep result of Olsson on the weights of Frobenius acting on the unipotent radical of $\pi_1^{rel, dR}$ imply that $(\pi_1^{rel, dR})^{\phi=1}$ is no bigger than the centralizer of Frobenius $Z(\phi)$ in the monodromy of the local system $\mathcal{L}^{dR}$. In other words, the $\phi$-invariant part has no effect on the dimensions of the de Rham period space as we move up the ``unipotent tower'' of $\pi_1^{rel, dR}$. This fact has the potential to eliminate the centralizer of Frobenius headache -- at the expense, of course, of much more challenging motivic and stack-theoretic arguments. For the full discussion, see our discussion of the relative unipotent version of Besser's Tannakian Coleman integration in Theorem \ref{integration}.

\subsection{The Chabauty--Kim diagram for relative completions}
Last but not least, the various motivic realizations of the relative completion have known comparisons due to Olsson, Pridham, Hain, and others. Assume that both $S^{dR}$ and $Z(\phi)$ are semisimple and simply connected. Taken together with a simple generalization of the Bloch--Kato logarithm, these comparisons give us a complete diagram of stacks (Section \ref{diagram})
\[
\begin{tikzcd}
\mathcal{X}(\mathcal{O}_{F, \Sigma}) \arrow{r} \arrow{d}{x \mapsto P^{\acute{e}t}_{b,x}} & \mathcal{X}(\mathcal{O}_{F_v}) \arrow{d}{x \mapsto P^{\acute{e}t}_{b,x}} \arrow{dr}{x \mapsto P^{dR}_{b,x}} \\
H^{1}_{f,\bold{r}}(G_{F,T}, \pi_1^{rel,\acute{e}t}) \arrow{r}{res} & H^{1}_{f,\bold{r}}(G_{F_v}, \pi_1^{rel,\acute{e}t}) \arrow{r}{D_{dR}} & Res^{F_v}_{\mathbb{Q}_p} \mathcal{M}^{dR}
\end{tikzcd}
\]
that strictly generalizes the Chabauty--Kim diagram when $Y \to X$ is the identity, i.e. when $\pi_1^{rel}$ is just the unipotent completion of the fundamental group. We also prove a number of other results, such as analyticity of the Albanese map (Section \ref{analyticity}) and strat-representability of the Selmer stack for the reductive quotient of the \'etale relative completion (Section \ref{reductiveselmer}.) Last but not least, we assume ``all conjectures'' and show how one could prove Siegel's theorem for the affine modular curves $Y_1(N)$ ($N \geq 4$) using Eisenstein classes (Section \ref{projline}.) The case $N=4$ is that of the projective line minus three points.

\begin{rem}
The semisimplicity and simple-connectedness of $Z(\phi)$ also has to be checked on a case-by-case basis. This condition, and the one on the monodromy above, arises because we must use a theorem of Kneser on vanishing of Galois cohomology for semisimple and simply-connected algebraic groups to produce points on certain path torsors. These conditions could likely be removed in future work by using non-archimedean stacks in the \'etale topology, but we have avoided that here because descent in the \'etale topology is much more complicated for rigid spaces than for schemes.
\end{rem}

\subsection{Related results}

The reader might see some superficial similarities between our cohomology stacks and the non-abelian cohomology schemes of Hain's paper \cite{hain2010remarks}, which proves some representability results for non-abelian cohomology sets related to relative completions. However, his results regard unipotent fundamental groups, so they are really quite different than those of the current paper.

Much of the work on motivic relative completions after Hain comes from the work of Pridham. The reader is directed to the invaluable works \cite{pridham2012l} and \cite{pridham2009weight} for a high-level (and highly homotopical) analysis of relative Malcev completions of fundamental groups.

There has been a flurry of recent work applying relative completions to modular forms and periods. Both Hain \cite{hain2016hodge} and Brown \cite{brown2014multiple} and many others have fruitfully applied relative completions to understanding these arithmetic problems. As we see in the section on the projective line minus three points, these works are quite related to our problem.

We note that our use of the phrase ``relatively unipotent" clashes with that of a body of work that uses this phrase for a kind of relative fundamental group $\pi_1(Y/X)$, for example in \cite{wildeshaus2006realizations}. These two meanings are sometimes related by pushforward and pullback of sheaves on $X$ and $Y$. Aside from the connection of our work to elliptic polylogarithms, we do not delve further into this relationship.

\section{Future work}

The most urgent result that one needs in order to apply Chabauty--Kim for relative completions is (strat-)representability and algebraicity of Selmer stacks for the relative completion.

Once that happens, it would be wonderful to have a theory of Coleman integration for the relative completion to make the work more explicit; $p$-adic integrals of modular forms were studied by Coleman, but computational tools would be welcome to form conjectures on vanishing of integrals of modular forms on rational points.

Most of our theory could be applied to the entire pro-algebraic completion of the fundamental group if the relevant $p$-adic Hodge theory were more developed in that case. Pridham's work on pro-algebraic homotopy types gets us most of the way.

Even though we apply Chabauty for the relative completion to the projective line minus three points in this work, we need to assume the Fontaine--Mazur conjecture for three-dimensional Galois representations (in addition to algebraicity of the relevant Selmer stacks.) One would like to give an unconditional proof of at least Siegel's theorem. One approach might involve showing that the cocyles $c$ for which the dimension of $H^1_f(G_T, {}_c\mathcal{G}^{et}_n)$ jumps are cut out by $p$-adic equations in $$H^1_f(G_p, S^{et}).$$

Finally, though the goal of this work was to bring the method of Chabauty--Kim closer to the method of Lawrence--Venkatesh, we have not yet applied our work to Kodaira--Parshin families. The first step would be to understand the cohomology of the local systems coming from these families. Alternatively, one could approach finiteness (via the original formulation of the Shafarevich Conjecture) using relative completions for the moduli space of curves, which have a rich arithmetic and geometric theory (see \cite{hain1997infinitesimal}, \cite{hain2009relative}.)

\section{Notation}
We have tried throughout to keep uniform algebraic notation: $F$ is a general number field with places $v$, $K$ is a $p$-adic field with residue field $\kappa$, $K_0$ is the maximal unramified subextension of $K$ and $R$ is the valuation ring of $K_0$. If $V$ is a scheme over $\Spec A$ and $B$ is an $A$-algebra, we often write $V \otimes_A B$ for $V \times_{\Spec A} \Spec B$.

We usually denote a Tannakian category by $\mathcal{T}$ and its fundamental group by $\pi(\mathcal{T})$; if the Tannakian category is a thick subcategory of another category $\mathcal{T}'$ generated by an implicit object $\mathcal{L} \in \mathcal{T}$, we usually supress $\mathcal{L}$ and write just $\mathcal{T}^{rel}$ for the thick subcategory and $\pi^{rel}$ or $\pi_1^{rel}$ for its fundamental group. All torsors are right torsors.

\chapter{The Relative Completion}

\section{Relative completions: algebraic theory}
A reference for the following material is \cite[Part 1, Section 3]{hain2016hodge}. Let $(\mathcal{T}, \omega)$ be a Tannakian category with fundamental group $\Pi$, and $\mathcal{L}$ a distinguished element of $\mathcal{T}$ that we will generally suppress in our notation. We denote by $S$ the Tannakian fundamental group of $\langle \mathcal{L} \rangle _{\otimes}$ the Tannakian subcategory generated by $\mathcal{L}$ inside $\mathcal{T}$. (By this we mean it contains $\mathcal{L}$ and is closed under tensor product, dual, quotient, and direct sum.) A general note about our Tannakian categories notation: If $\omega, \omega'$ are two fiber functors on a Tannakian category $\mathcal{T}$, we write $P_{\omega, \omega'}$ for the scheme $Isom(\omega, \omega')$ -- if these two fiber functors come from points $x,x'$ in a geometric situation, we will simply write $P_{x,x'}$.

We make the following assumption throughout this work: $S$ is reductive, or equivalently, every element of the tensor category generated by $\mathcal{L}$ is semisimple.

\begin{defn}
Let $\mathcal{T}^{rel}$ be the subcategory of $\mathcal{T}$ defined by the property that its elements have a filtration by subobjects $$\cdots \subseteq V_2 \subseteq V_1 = V $$ such that each representation $\Pi \to \Aut \left( \omega(gr_i(V)) \right)$ factors through the map $\Pi \to S$; in other words, the associated graded of $V$ lies in $\langle \mathcal{L} \rangle _{\otimes}$. Equivalently, $\mathcal{T}^{rel}$ is the thick tensor subcategory generated by $\mathcal{L}$. We call this the category of objects of $\mathcal{T}$ that are unipotent relative to $\mathcal{L}$, or if $\mathcal{L}$ is implicit just the relatively unipotent objects.

Finally, we will refer to $\pi(\mathcal{T}^{rel})$ as either the relative unipotent fundamental group of $\mathcal{T}$ or as just the relative completion, with $\mathcal{T}$ and $\mathcal{L}$ implicit.
\end{defn}
(Note that the terminology relative completion is slightly misleading because there is no completion going on yet.) When $\mathcal{L}$ is the tensor unit, the category $\mathcal{T}^{rel}$ is the traditional category of unipotent objects. For brevity, denote by $\mathcal{G}$ the fundamental group of $\mathcal{T}^{rel}$. There is a map $\mu \colon \mathcal{G} \to S$ given by including the Tannakian category generated by $\mathcal{L}$ into the thick Tannakian category generated by $\mathcal{L}$. Also, the inclusion of the full subcategory of truly unipotent objects induces a surjection from the fundamental group of objects unipotent relative to $\mathcal{L}$ to the fundamental group of unipotent objects.
\begin{prop}
The homomorphism $\mu \colon \mathcal{G} \to S$ induces an exact sequence 
\begin{align} \label{extn}
1 \to \mathcal{U} \to \mathcal{G} \to S \to 1,
\end{align}
 where $\mathcal{U}$ is pro-unipotent.
\end{prop}
\begin{proof}
The surjectivity of $\mu$ follows from the fact that $\langle \mathcal{L} \rangle _{\otimes}$ is a full subcategory of $\langle \mathcal{L} \rangle _{thick}$, and that the former is closed under subobjects in the latter \cite{milne2007quotients}.

Now let $\mathcal{U}$ denote the subgroup of $\mathcal{G}$ which fixes the associated graded of any element of $\mathcal{T}$ for any filtrations as above, i.e. the $\mathcal{U}$-representations $\omega\left( gr^i(V) \right)$ are trivial for all objects $V$ and all filtrations whose associated gradeds descend to $S$-representations. From this description it is clear that $\mathcal{U}$ is precisely the kernel of the map $\mathcal{G} \to S$.

Let us show that $\mathcal{U}$ is unipotent by induction. There is an increasing sequence of subcategories $\mathcal{T}^{rel}_n$ that bounds by $n$ the number of terms in the filtration associated to objects of $\mathcal{T}^{rel}$; denote by $\mathcal{G}_n$ its fundamental group, and $\mathcal{U}_n$ the subgroup which fixes the associated graded of $S$-filtrations, as before. The inclusion $\mathcal{T}^{rel}_n \to \mathcal{T}^{rel}_{n+1}$ gives a map $\mathcal{G}_{n+1} \to \mathcal{G}_n$ whose kernel $K$ consists of elements which fix (the fiber of) all objects of $\mathcal{T}^{rel}_n$. Let $V \in \mathcal{T}^{rel}_{n+1}$; then the action of $K$ on $V$ factors through $gr^1(V)$ because $V_2 \in \mathcal{T}^{rel}_n$. But $\mathcal{U}_n$ acts trivially on the associated graded by definition, and so $\mathcal{U}_{n+1}$ acts unipotently on $\omega(V)$ for all $V \in \mathcal{T}^{rel}_{n+1}$.

Now $\mathcal{U}_n$ is unipotent because it acts unipotently on representations of $\mathcal{G}_n$. One sees this by embedding $\mathcal{G}_n$ into some $GL_m$ and observing that $\mathcal{U}_n$ acts unipotently on the canonical representation of $GL_m$. This argument moreover shows that $\mathcal{U}_n$ is the unipotent radical of $\mathcal{G}_n$, since any presentation of $\mathcal{G}_n$ as an extension of a reductive group by a unipotent group is unique up to conjugation. Thus, since $\mathcal{U}$ is the inverse limit of the $\mathcal{U}_n$, it is pro-unipotent.
\end{proof}

We use the following notation for a pro-algebraic group that is an extension of a reductive group by a pro-unipotent group. Let $\mathcal{U}^n$ denote the $n$-th level of the lower central series with $\mathcal{U}^1 = [\mathcal{U}, \mathcal{U}]$, and so on and so forth, setting also $\mathcal{U}_n = \mathcal{U}/\mathcal{U}^n$ and $Z_n = \mathcal{U}^n/ \mathcal{U}^{n+1}$. The notation $\mathcal{G}_n$ will refer to the extension obtained by pushing out \ref{extn} by the map $\mathcal{U} \to \mathcal{U}_n$.

We can use this machinery to construct a generalized version of the unipotent or Malcev completion of an abstract group.
\begin{defn}
Let $\rho \colon \Gamma \to S(F)$ be a continuous and Zariski-dense homomorphism of a topological group $\Gamma$ into an reductive group $S$ over a topological field $F$. Consider the Tannakian category $Rep_F(\Gamma)$ of all finite dimensional continuous representations of $\Gamma$ defined over $F$ and its Tannakian subcategory $Rep_F(\Gamma)^{rel}$ of representations unipotent relative to $\rho$, as defined above.

The relative completion of $\Gamma$ with respect to $\rho$, and over $F$, is defined as the Tannakian fundamental group of $Rep(\Gamma)^{rel}$.
\end{defn}

This relative completion has an important universal property which mirrors that of the Malcev completion. Preserve the notation of the last definition, writing $\mathcal{G}$ for the relative completion and $\mathcal{U}$ for its unipotent radical.
\begin{prop}
Let $\mathcal{G}'$ be a proalgebraic group which is an extension $$1 \to \mathcal{U}' \to \mathcal{G}' \to S \overset{\phi} \to 1,$$ where $\mathcal{U}'$ is unipotent, and suppose there exists a map $\Gamma \to \mathcal{G}'(F)$ whose composition with $\phi$ is $\rho$. Then there is a map $\mathcal{G} \to \mathcal{G}'$ that makes the following diagram commute, where the horizontal arrow is $\rho$:
\[
\begin{tikzcd}
& \mathcal{G}(F) \arrow{dr} \arrow{dd} &\\
\Gamma \arrow{ur} \arrow{dr} \arrow{rr} &  & S(F)\\
& \mathcal{G}'(F) \arrow{ur} &
\end{tikzcd}
\]
\end{prop}
The proof of this fact proceeds in exactly the same way as the result that the Tannakian fundamental group of all representations is precisely the algebraic envelope over $F$, cf. \ref{envelopes}.

The operation of taking relative completions is functorial in the sense that if one is given two representations $\rho \colon \Gamma \to S(F)$ and $\rho' \colon \Gamma \to S'(F)$ and a map $f \colon S \to S'$ such that $f \circ \rho = \rho'$, there is a morphism $\mathcal{G} \to \mathcal{G}'$ ``making all of the relevant diagrams commute.''

\begin{exmp}
We may apply the technology of the relative completion to \'etale fundamental groups as follows.

Suppose $X$ is a scheme defined over any field $K$, that $\mathcal{L}^{\acute{e}t}$ is a $\mathbb{Q}_p$-local system on $X$ such that the pullback of $\mathcal{L}^{\acute{e}t}$ to $\bar X$ has reductive monodromy group $S$. Write $\pi_1^{rel,\acute{e}t}(\bar X,b)$ for the Tannakian relative completion of the category of $\mathbb{Q}_p$-local systems with respect to $\mathcal{L}^{\acute{e}t}$ (equivalently, with respect to the morphism $\pi_1^{\acute{e}t, profinite}(\bar X,b) \to S(\mathbb{Q}_p)$ defined by $\mathcal{L}^{\acute{e}t}$.)

Now the Galois action can be defined functorially: $G_K$ acts on $\bar X = X \otimes \bar K$ and thus on the category of $\mathbb{Q}_p$-relatively unipotent local systems on $\bar X$. The functoriality of Tannakian reconstruction implies that an action by equivalences of this category induces an action by automorphisms on its fundamental group.

An equivalent and more conventional way to define the Galois action is using a conjugation action. We start with the so-called Fundamental Exact Sequence $$0 \to \pi_1^{\acute{e}t, prof}(\bar X,b) \to \pi_1^{\acute{e}t, prof}(X,b) \to G_K \to 0.$$ Here $prof$ refers to the profinite \'etale fundamental group. The basepoint $b$ provides a section of this map, using the functoriality of Galois categories, and then $G_K$ acts on $\pi_1^{\acute{e}t, prof}(\bar X,b)$ by conjugation. Secondly, $G_K$ also acts on $S(\mathbb{Q}_p)$ by conjugation using the Galois action on the stalk of $\mathcal{L}^{\acute{e}t}$ at $b$, and the map $$\pi_1^{\acute{e}t, prof}(\bar X,b) \to S(\mathbb{Q}_p)$$ is equivariant with respect to these two actions.

Given this equivariance, the functoriality of the relative completion furnishes an action of $G_K$ on $\pi_1^{rel,\acute{e}t}(\bar X,b)(\mathbb{Q}_p)$.

\end{exmp}

Last but not least, we cover an essential bit of homological algebra that is important to computations and comparisons with the relative completion. For any extension $\mathcal{G}$ of a reductive $S$ by a unipotent $\mathcal{U}$ as above and any representation $V$ of $\mathcal{G}$, we have a Lyndon--Hochschild--Serre spectral sequence $$H^p(S, H^q(\mathcal{U}, V)) \Rightarrow H^{p+q}(\mathcal{G},V).$$ But $S$ is reductive and thus has vanishing higher cohomology; so the terms of the spectral sequence are zero if $p>0$. We thus obtain an equality $$H^q(\mathcal{G},V) \simeq  H^q(\mathcal{U}, V)^S.$$

Now suppose that we have two such extensions of $S$, with a map between them. With some algebraic work
one shows that the morphism $\mathcal{G}_1 \to \mathcal{G}_2$ is an isomorphism if the induced map $$H^1(\mathcal{U}_1, V)^S \to H^1(\mathcal{U}_2, V)^S$$ is an isomorphism for all irreducible representations $V$ of $\mathcal{G}$. The point is that both sides will have an interpretation as certain (de Rham, \'etale, etc.) cohomology groups on a space $X$, which we can compare.

\section{Overview: relative completions for the motivically inclined}
In the next sections we will define and discuss various ``relative completions'' of motivic fundamental groups. Again, we use quotation marks here because the resulting fundamental group will not always be the relative completion of a group in the category of sets with respect to a representation. However, they will always come from a thick Tannakian subcategory of a larger category, as we explained in the last section.

In any case, we take this section to prepare the reader for the motivic story that she is about to read in various realizations. The protagonist of our story is a smooth quasi-projective scheme $X$ over $\mathcal{O}_{F, \Sigma}$, and in its arsenal lay a vector bundle with integrable connection on $X_{F_v}$ ($v \not \in \Sigma$) or $X_{\mathbb{C}}$, or an isocrystal on the special fiber of $X$ at $v$, or a $\mathbb{Q}_p$-local system on $\bar X$ that is defined over $\mathcal{O}_{F, \Sigma}$. Betti local systems also play an important role in our story.

Being motivic, these act like they are realizations of one true entity, $\mathcal{L}^{mot}$. The best way to obtain such an object will be to start with a smooth and proper morphism $f \colon Y \to X$ and push forward a ``trivial motive" to obtain the relative de Rham cohomology, or fiberwise geometric \'etale cohomology, etc. We call this the ``Gauss--Manin case'' because the de Rham realization of this process is the Gauss--Manin connection.

The Tannakian fundamental group of (the various realizations of) $\langle \mathcal{L}\rangle_{\otimes} $ in the category of local systems over $X$, or the category of vector bundles, or... will be denoted $S$, as in the last section. Then Deligne's machinery of algebraic geometry in a Tannakian category says that the structure sheaf $\mathcal{O}(S)$ always has the structure of an ind-motive over $X$, i.e. an ind-$\mathbb{Q}_p$-local system, or an ind-vector bundle with integrable connection, or...

We then take the cohomology $\mathbb{R}\Gamma(X,\mathcal{O}(S))$. For smooth and proper $X$, we have formality theorems of the following kind, given by Olsson, Pridham, and Hain: 
\begin{center}
The Lie algebra $\mathfrak{u}$ of the unipotent radical $\mathcal{U}$ of $\mathcal{G}$ in each category is a quotient of the free Lie algebra on $H^1(X,\mathcal{O}(S))^*$ with relations given by the cup products of these elements into $H^2(X,\mathcal{O}(S))^*$.
\end{center}
These types of theorems are generalizations of Chen's original $\pi_1$-de Rham theorem and the amazing formality result of Deligne, Griffiths, Morgan, and Sullivan \cite{deligne1975real}. In the classical setting of unipotent Tannakian fundamental groups, this formality result says precisely that $\mathfrak{u}$ is the quotient of the free Lie algebra on $H^1_{dR}(X)$ by cup product relations.

When $X$ is not projective, one can still use the Leray and Adams spectral sequence and weight arguments to give a presentation for $\mathfrak{u}$. For our purposes, the important statement is as follows \cite[proof of Proposition 3.4]{pridham2012l}. We write $X^o = X - D$ for $X$ projective and $D$ some normal crossings divisor such that $\mathcal{L}$ has tame monodromy around $D$, and $j$ for the inclusion of $X^o$ into $X$.

\begin{center}
The Lie algebra $\mathfrak{u}$ of the unipotent radical $\mathcal{U}$ of $\mathcal{G}$ in each category is a quotient of the free Lie algebra on $H^1(X^o,\mathcal{O}(S))^* \oplus H^0(X, R^1j_*\mathcal{O}(S))^*$ with relations given by cup product and spectral sequence maps between these groups and wedge and tensor products of the groups $H^2(X,j_*\mathcal{O}(S))$, $H^1(X,R^1j_*\mathcal{O}(S))$, and $H^0(X,R^2j_*\mathcal{O}(S))$.
\end{center}
The most difficult weight argument appears in the crystalline case, where one needs the recent work of Abe and Caro \cite[Main Theorem]{abe2018theory} for the above statement.

Now that we have some motivation for the motivic story, we begin to build moduli spaces for the \'etale relative completion.

\chapter{Stacky Non-Abelian Galois Cohomology}
As the conjectures of Bloch-Kato and others have shown, the theory of Galois representations is a central organizing principle in number theory. And as soon as one begins to study any group representations, one immediately becomes interested in the cohomology of the given group with values in a representation -- these groups are the derived functors of the invariants.

For us, the important fact is that the \'etale fundamental group of a scheme defined over a number field comes with an action of the Galois group of the base field. The \'etale fundamental group can therefore rightly be called a \textit{non-abelian} Galois representation.  Let us define stacky Galois cohomology of these non-abelian Galois representations and prove that it satisfies the single most important property of group cohomology: as much as possible, it converts short exact sequences of representations to long exact sequences in cohomology. First we must define what we mean by an exact sequence of groupoids.

\begin{defn}
\begin{description}
\item[(1)] A homomorphism of groupoids is a functor between the underlying categories. 
\item[(2)] We call a groupoid $\Gamma$ pointed if it has a distinguished isomorphism class, called the point of $\Gamma$ and usually denoted $pt$.
\item[(3)] A homomorphism of groupoids $g \colon \Gamma \to \Gamma'$ is called pointed if it sends the point of $\Gamma$ to the point of $\Gamma'$.
\item[(4)] The kernel of a homomorphism of pointed groupoids is the subcategory sent to the point, and the image of such a homomorphism is the essential image of the functor. We say that a sequence of homomorphisms of groupoids is exact if the kernel of the second map is equal the image of the first.
\end{description}
\end{defn}
We now have the technology to discuss non-abelian cohomology groupoids. The basic material on short exact sequences below is from the canonical reference, Serre's book on Galois cohomology \cite[Chapter 1, Section 5]{ion2001galois} -- though Serre does not use the language of groupoids.
\begin{defn}
Consider $G$, a topological group, acting continuously on a topological group $B$. The action of $g$ on $b$ will be denoted ${}^g b$.
\begin{description}
\item[\textbf{(0th Cohomology Group)}]
We define $$H^0(G, B) = B^G.$$
\item[\textbf{(1st Non-abelian Cohomology Groupoid -- non-stacky)}]
The set of continuous 1-cocycles of $G$ with values in $B$ is $$Z^1(G, B) \defeq \{c \colon G \to B \text{ s.t. $c$ is continuous and } c(gh) = c(g) ^gc(h)\}.$$ There is an action of $B$ on $Z^1(G, B)$ given by $$c(g) \mapsto b^{-1}c(g)^gb,$$ and the quotient groupoid by this action is the first cohomology groupoid $H^1(G,B)$. It has a canonical point given by the isomorphism class of the trivial 1-cocycle.
\end{description}
\end{defn}

The following construction, called Serre twisting, will come up throughout our exposition. If $c \in Z^1(G, C)$ is a cocycle with values in a group $C$ and $C$ acts on another group $A$, then we denote by $_c A$ the new $G$-set on which $G$ acts via $a \mapsto c(g) \cdot {}^g a$. Usually it will either be the case that $C=A$ acting on itself on the left or we will have an extension of $C$ by $A$ and $C$ will act by conjugation.

Finally, for our purposes a short exact sequence of topological $G$-groups is a sequence of topological groups with continuous $G$-action \[1 \to A \to B \to C \to 1 \] which is exact as a sequence of groups and such that each of the morphisms in the exact sequence is continuous. Note that in many places in the literature one requires the second map to be a topological quotient, but we do not need that here.
\begin{prop}
In the setup of the last definition, let $A$ be a subgroup of $B$. Consider the following conditions on $A$, of increasing stringency:
\begin{description}
\item[0] $A$ is a $G$-stable subgroup of $B$ such that $B/A$ has some given topology such that the map $B \to B/A$ is continuous, and there is a continuous section of the map $B \to B/A$.
\item[1] $A$ is normal and $G$-stable in $B$, yielding a short exact sequence of topological $G$-groups $$1 \to A \to B \to C \to 1,$$ and there is a continuous set-theoretic section of the map from $B$ to $C$.
\item[2] $A$ is abelian, normal, and $G$-stable in $B$, and the aforementioned splitting exists.
\item[3] $A$ is central and $G$-stable in $B$ and the aforementioned splitting exists. 
\end{description}
Then we have the following non-abelian analogues of short exact sequences for cohomology groupoids:
\begin{description}
\item[0] Under condition $\textbf{0}$, there is an exact sequence $$1 \to H^0(G,A) \to H^0(G,B) \to H^0(G,B/A) \to H^1(G, A) \to H^1(G,B)$$ of pointed groupoids, which is an exact sequence of groups at the first two terms. Furthermore, the following ``twisted exactness property'' holds: Let $a \in Z^1(G,A)$, and $b \in H^1(G,B)$ its image. Then the groupoid-theoretic fiber above $b$ in the above exact sequence is given by the groupoid $$H^0(G,{}_aB/ {}_aA)/H^0(G,{}_a B).$$ (In particular, it is independent of which $a$ in the fiber is chosen.)

This twisted exactness property does not require the existence of a continuous splitting of the quotient map.

\item[1] Under condition $\textbf{1}$, there is an exact sequence $$1 \to H^0(G,A) \to H^0(G,B) \to H^0(G,C) \to H^1(G, A) \to H^1(G,B) \to H^1(G,C)$$ of pointed groupoids, which is an exact sequence of groups at the first 3 terms. Furthermore, the following ``twisted exactness property'' holds, in addition to the property mentioned in $\textbf{0}$: Let $b \in Z^1(G,B)$, and denote by $c$ the image of $b$ in $H^1(G,C)$. Then the groupoid-theoretic fiber above $c \in H^1(G,C)$ in the exact sequence is given by $H^1(G,{}_b A)/H^0(G,{}_b C)$. Here the twist $_b A$ is by conjugation.

\item[2] Under condition $\textbf{2}$, the exact sequence given above, $$1 \to H^0(G,A) \to H^0(G,B) \to H^0(G,C) \to H^1(G, A) \to H^1(G,B) \to H^1(G, C),$$ satisfies a further exactness property. For a class $c \in H^1(G,C)$, there is an element $\Delta(c) \in H^2(G, {}_c A)$ such that $c$ is in the image of the map from $H^1(G,B)$ if and only if $\Delta(c)$ is cohomologous to the trivial cocycle. (By ${}_c A$ we mean that $C$ acts on $A$ by conjugation since $A$ is abelian, and as described before the $G$-action on ${}_c  A$ is $a \mapsto c(g) \cdot {}^g a$.) Furthermore, $\Delta(c)$ is defined only up to a choice of cocycle representative for $c$, but its triviality does not depend on this choice.
\item[3] Under condition $\textbf{3}$, there is an exact sequence 
\begin{align*}
1 \to H^0(G,A) \to H^0(G,B) \to H^0(G,C) \to \\ H^1(G, A) \to H^1(G,B) \to H^1(G, C) \to H^2(G, A)
\end{align*}
of groupoids which furthermore satisfies all of the previous twisted exactness properties.
\end{description}

\end{prop}

\begin{rem}
The connecting homomorphism $H^0(G,C) \to H^1(G, A)$ is given by continuously lifting an element $\tilde x \in H^0(G,C)$ to $x \in B$, and sending it to the cocycle $(g \mapsto x^{-1} {}^g x).$ The cocycle $\Delta(c)$ is defined by lifting a cocycle $c \in H^1(G, C)$ to a continuous function $\tilde c \colon G \to A$. One sends this function to the continuous 2-cocycle $\Delta(c)(g,h) = \tilde c(g){}^g \tilde c(h) \tilde c(gh)^{-1}$; in general, this is an element of $H^2(G,{}_c A)$, but when $A$ is central it is easy to see that ${}_c A$ is canonically isomorphic to $A$, so we have an honest connecting homomorphism.

The existence of $\Delta$ can be made more theoretically satisfying. In fact, suppose that we are in case $\textbf{1}$, where $A$ is just normal. Then if one considers the short exact sequence $0 \to A \to B \to C \to 0$ as an exact sequence of crossed modules, there is always an honest long exact cohomology sequence of non-abelian cohomology crossed modules up to $H^2(G,C)$! The theory of crossed modules also makes Serre twisting much more natural. However, defining non-abelian $H^2$ would take us too far afield of our exposition here. For a complete reference, see \cite{inassaridze2002higher}.
\end{rem}

\begin{rem}
There is no modification needed to the proofs in \cite{ion2001galois} besides the obvious fact that all maps there are indeed maps of groupoids; we claim no original thought in this regard but this observation is essential to our later work. We do not review the proofs here, and it would be hard to find better proofs than in Serre's exposition.
\end{rem}

If we imagine the set $H^1(G,C)$ as a geometric space, we can think of the groups $H^2(G,{}_c A)$ as a family of groups that varies continuously over it. In later chapters we will do much more than imagine such a space. The following proposition is at least some cheap version of this fantasy: it says that $\Delta$ acts like a section of that imagined bundle.

First, a definition that we need to state the proposition.
\begin{defn} \label{continuity}
A topological group $G$ acts continuously on a topological space $A$ if the map $G \times A \to A$ is continuous. A topological group $G$ acts continuously on a functor $F \colon \mathcal{C} \to Top$ if its action on $F(c)$ is continuous for each $c \in \mathcal{C}$.
\end{defn}
\begin{exmp} \label{galactdefn}
Let $C$ be an algebra over a complete Hausdorff topological field of characteristic zero. We give $C$ the inductive limit topology, the finest topology for which the inclusion of finite-dimensional subspaces is continuous. (Finite-dimensional subspaces are given the product topology from the base field.)

If $X$ is an affine $C$-scheme of finite type, then $X(C)$ has the topology induced by choosing some {\bfseries closed} embedding $X \hookrightarrow \mathbb{A}^m_C$. Then $X(C) \subseteq C^m$, and we give $X(C)$ the subset topology using the product topology on the right. This topology is independent of the embedding, and maps of affine schemes induce continuous maps on $C$-points \cite[Lemmas 1-3]{kim2005motivic}. We will only use this topologization of points for affine schemes; if one wants to glue this definition for general schemes or related nice properties such as open embeddings of schemes inducing open embeddings of topological spaces, it becomes necessary to assume that the inversion map on the units of $C$ is continuous (cf. the exposition in \cite[Section 4]{christensen2020topology}.)

Thus, if $S$ is an algebraic group over $\mathbb{Q}_p$ then the functors \[F_1 \colon R \mapsto S(R \widehat{\otimes} B_{crys})\] and \[F_2 \colon R \mapsto S(R)\] for $R$ running over affinoid $\mathbb{Q}_p$-algebras land in the category of topological spaces, again using closed embeddings into affine space and the topologies on both $R$ and $B_{crys}.$ It therefore makes sense to discuss the continuity of an action of a profinite group $G$ on the two functors above.

For us, the important example of such a continuous action arises as follows \cite[Lemma 5]{kim2005motivic}: Suppose that the coordinate ring of $S$ comes with a continuous action $\alpha$ of $G_K$, then the action of $g \in G$ on $F_1$ given by \[\left(p \colon \mathcal{O}(S) \to R \widehat{\otimes} B_{crys}\right) \mapsto \alpha(g)p\alpha(g)^{-1}\] is continuous in the above sense, and the same is true for $F_2$.
\end{exmp}
We will need long exact sequences in these contexts, so we must provide functorial and continuous splittings for short exact sequences.
\begin{prop}
Let \[1 \to U \to H_1 \overset{r} \to H_2 \to 1 \] be a short exact sequence of linear algebraic groups over a field $k$ in which $U$ is unipotent. There exists a scheme-theoretic section $\sigma \colon H_2 \to H_1$ of $r$.
\end{prop}
\begin{proof}
It is an exercise in group theory to show that $H_1$ is a $U$-torsor over $H_2$ for the \'etale topology. Then the proof follows Mostow's Theorem in the non-algebraically closed case \cite[Prop. 5.4.1, until ``We will modify...'']{conrad2014reductive}, as follows.

By pushing out to the lower central series, we may assume $U$ is abelian; then $H_1 \to H_2$ is a vector bundle. We would like to show that this torsor is trivial; it forms a class in sheaf cohomology on the \'etale site: $H^1 \left( (H_2)_{\acute{e}t}, U \otimes_k \mathcal{O}_{H_2} \right)$. But this cohomology vanishes because $H_2$ is affine.
\end{proof}

We need one more result establishing functoriality for the connecting homomorphism.
\begin{prop} \label{cohomologous}
Let $A$ and $B$ be topological-group functors from a category $\mathcal{C}$ to the category of topological groups, such that $G$ acts continuously on them. Suppose that for all $T \in \mathcal{C}$, $A(T)$ and $B(T)$ satisfy condition $\textbf{2}$ above, i.e. that $A(T)$ is an abelian, normal, and $G$-stable subgroup of $B(T)$ with quotient functor $C(T)$. Then the ad-hoc connecting homomorphism $\Delta$ is functorial in the following sense.

Let $c_T \in Z^1(G,C(T))$ be a cocycle, $f \colon T \to T'$ a morphism in $\mathcal{C}$. On one hand, $\Delta(c_T)$ is an element of $H^2(G, _c A(T))$; post-composing with $f$ produces an element which we will denote $\Delta(c_T) \otimes T'$. On the other hand, $c_T$ defines a cocycle $c_{T'}$ by post-composing with $f$, and it has an associated cocycle $\Delta(c_{T'})$. The claim of functoriality means that $$\Delta(c_T) \otimes T' = \Delta(c_{T'}.)$$

\end{prop}
\begin{proof}
One has $$\left( \Delta(c_T)\otimes T' \right)(g,h) = f \left (  c_T(g){}^g  c_T(h)  c_T(gh)^{-1} \right ).$$ On the other hand, $$\Delta(c_{T'}) = f( c_T(g)){}^g f( c_T(h)) f( c_T(gh))^{-1}.$$
The proposition thus follows from the functoriality of the group operation on $A$ and the functoriality of the action of $G$ on $A$.
\end{proof}

If $A$ is a functor in groups on a Grothendieck site then we can use the machinery of non-abelian cohomology sets to develop a stacky version of non-abelian cohomology for $A$.
\begin{defn}{$H^1$ \textbf{of a non-abelian representation (stacky)}} 
Let $A$ be a functor in topological groups on a Grothendieck site $(\mathcal{C}, \tau)$, $G$ a topological group acting continuously and functorially on $A$, and $T \in \mathcal{C}$. We define the stack $$H^1(G,A)$$ to be the stackification of the pseudofunctor $$T \mapsto H^1(G,A(T)).$$ In other words, $H^1(G,A)$ is the stack quotient $$\left[Z^1(G,A)/A\right].$$ We define $H^0$ of a general functor and $H^2$ for an abelian group functor similarly, using sheafification. 
\end{defn}
\begin{exmp}
It is not hard to see that if $(\mathcal{C}, \tau)$ is the Zariski site of schemes and $A$ is a group scheme with an action of some group $G$, the functor $H^0(G, A)$ agrees with the usual definition of $A^G$: $A^G(T)$ is the set of all $x \in A(T)$ such that $gx_{T'} = x_{T'}$ for all $g \in G$, $h \colon T' \to T$.
\end{exmp}

\section{Interpreting cohomology stacks: torsors vs. cohomology classes}
Armed with this definition, we proceed to the equivalence which is the core of nonabelian group cohomology; the dictionary in the unipotent case says that there is a bijection between torsors and elements of a certain $H^1$.  In this section we extend this dictionary to the non-unipotent case.

Let $A$ be a topological-group-valued sheaf on a site $\mathcal{C}$, with a continuous (functorial) action of a profinite group $G$. An $A$-torsor with $G$-action is a sheaf $P$ with a continuous action by $G$ and an action $A \times P \overset{\alpha} \to P$ such that the torsorial map $A \times P \overset{\alpha \times id} \to P \times P$ is an isomorphism, and such that $\alpha$ is $G$-equivariant with respect to the diagonal action on the left hand side.

We denote by $A-Tors_G$ the category of $A$-torsors with compatible $G$-action, with isomorphisms as morphisms. For any $T \in \mathcal{C}$ denote by $A_T$ the object $A \times T$. For varying $T$ the groupoids $A_T-Tors_G$ form a stack on $\mathcal{C}$ -- this is a purely sheaf-theoretic argument \cite[\href{https://stacks.math.columbia.edu/tag/036Z}{Tag 036Z}]{stacks-project}.

Note that such torsors generally may not be representable in the category $\mathcal{C}$. However, since we will always have descent for the topologies we consider, such torsors will indeed be representable.

\begin{defn} \label{whatisatorsor}
Let $N$ be an algebraic group over a $p$-adic field $E$ on which $G$ acts continuously and $E$-linearly for the $p$-adic topology, and $X$ a rigid-analytic space over $E$. By an $N$-torsor $P$ over $X$ we mean an $N^{an}$-torsor over $X$ in the Tate topology: For every point $x \in X$, there exists an admissible neighborhood $U$ containing $x$ and a $G-$ and $N$-equivariant isomorphism \[U \times_E N \simeq P \times_X U.\] Such a torsor is represented by a rigid-analytic space over $X$, since rigid-analytic spaces satisfy descent for admissible coverings by definition, and we see that the $N$-torsors form a stack in groupoids over the category of rigid-analytic spaces.
\end{defn}

Let us finally state the desired equivalence. We return to the abstract situation where $A$ is a functor in topological groups on $(\mathcal{C}, \tau)$, and assume that $A-Tors_G$ is a stack on $\mathcal{C}$. Suppose, furthermore, that every $A_T$-torsor-with-$G$-action acquires a $T'$-point for some $T'$ which covers $T$ in $\tau$. Consider the association of pseudofunctors $$\mu(T) \colon H^1(G,A(T)) \to A_T-Tors_G$$ given by sending a cocycle $c_T$ to the $A$-torsor $_{c_T}A_T$, i.e. the Serre twist of $A_T$ by $c_T$. The universal property of stackification gives a map $$\tilde \mu \colon H^1(G,A) \to A-Tors_G,$$ where now the source means the stackified cohomology.
\begin{prop} \label{torsequiv}
The map $\tilde \mu$ is an isomorphism of stacks.
\end{prop}
\begin{proof}
We check the defining properties of the stackification as given in \cite[\href{https://stacks.math.columbia.edu/tag/02ZM}{Section 02ZM}]{stacks-project} (we have placed it in \ref{stackification} for ease of reference).
\begin{description}
\item[(1)]
By assumption, for any $P \in A_T-Tors_G$ there is an object $T'$ which covers $T$ in $\tau$ and a cocycle $c' \in Z^1(G,A(T'))$ such that $T(c') = P_{T'}$. Indeed, we may by our assumption let $T'$ be a $\tau$-covering of $T$ such that $P(T')$ has a point $\alpha$ and work as follows.

The point $\alpha$ defines a cocycle in the standard way: $$c'(g) = (g\alpha)\alpha^{-1},$$ i.e. the unique element $c'(g) \in A(T')$ such that $g\alpha = c'(g) \alpha$, and  standard argument shows that $\mu(c') = P_{T'}$ and thus $\tilde \mu (c') = P_{T'}$.

\item[(2)]
We must show that, for any two cocycles $c_1, c_2$ in $Z^1(G, A(T))$, the map of presheaves on $\mathcal{C}/T$ $$Isom(c_1,c_2) \to Isom(\mu(c_1), \mu(c_2))$$ identifies the latter as the sheaffification of the former. But this is obvious because the two are canonically isomorphic on all $T$-points: an element $a \in A(T')$ exhibiting $c_1|_{T'}$ and $c_2|_{T'}$ as cohomologous is sent to the isomorphism ``translation by $a$'' on the right hand side. This association is clearly injective, and in fact any isomorphism between $\mu(c_1)$ and $\mu(c_2)$ is given by translation by an element $a' \in A(T)$.

\end{description}
\end{proof}

Recall that by the universal property of rigid analytification (cf. \ref{univan}) if $A$ is an affine algebraic group over $E$ and $T$ is an affinoid algebra over $E$, we have a bijection $A(T) = A^{an}(T).$
\begin{exmp}
We apply Proposition \ref{torsequiv} when $\mathcal{C}$ is the category of rigid-analytic spaces over $E$, $\tau$ is the Tate topology on $\mathcal{C}$, and $N$ is an algebraic group over $E$. Then we see that the stackification of the functor which sends an affinoid $E$-algebra $T$ to $H^1(G, A(T)) = H^1(G, A^{an}(T))$ is precisely the stack of $A$-torsors defined in Definition \ref{whatisatorsor}.
\end{exmp}

\begin{rem}\label{nonempty}
On the \'etale side, we are not generally concerned with the question of whether a $B$-torsor with trivial $G$ action has a point defined over the base field, because all path torsors for the \'etale fundamental group will have a point over the base field. An extremely important point is that this point might have a highly non-trivial Galois orbit.

Here is a sketch of the fact that all path torsors have (likely not Galois-invariant!) rational points \cite[Proposition 5.5.1]{szamuely2009galois}. The universal cover of $X$ has a model defined over the base field. The fiber at a point $x$ is just the profinite limit of the fibers of all covers in the pro-cover. The inverse limit of nonempty finite sets is nonempty, and we conclude that there is a point in the fiber of the universal cover. Mapping this fiber to the Tannakian fundamental group gives the desired path in the path torsor.

A brief warning: There are, then, two senses in which Galois cohomology could potentially enter into the non-unipotent picture. On one hand, one has a Galois action on the entire pro-scheme $\pi_1^{\acute{e}t}$. On the other hand, one has a Galois action on the points $\pi_1^{\acute{e}t}(\bar K)$. The first is the one involved in the Chabauty-Kim diagram, and the second would be involved in classical questions about twists of algebraic groups. It is often confusing to newcomers that \textit{our Galois cohomology only detects the former.}

This is related to the fact that we have to use Kneser's theorem to assume our torsors on the de Rham side have trivializations, and it comes from the fact that our stacks will have effective descent for the Tate topology only. Such a move might be avoidable, but in order to stackify with respect to the \'etale topology in the rigid setting, we would have to know much more about effective descent for torsors in the rigid \'etale topology. Unfortunately, fpqc and \'etale descent of quasi-coherent sheaves is very difficult in the rigid setting \cite{conrad2006relative}, so we have purposefully avoided these topologies.
\end{rem}

We record for use the following fact that follows directly from the cohomology long exact sequence above. Let $\mathcal{C}$ be a Grothendieck site.
\begin{prop} \label{fibers}
Let $T \in \mathcal{C}$ and $G$ a topological group; suppose that $A$ and $B$ are sheaves of topological groups on the category $\mathcal{C}$, $A$ is a subsheaf of $B$, and that $G$ acts continuously on $B$ while preserving $A$. Now let $H^1(G,A) \to H^1(G,B)$ be the morphism of stacks coming from the inclusion of $A$ into $B$. Pick a cocycle $a_T \in Z^1(G,A(T))$, and let $b_T \in H^1(G,B_T)$ be its image in the stackified cohomology of $B_T$. Then the stack-theoretic fiber above $b_T$ in the stack $H^1(G, A)$ is $\tau$-locally the quotient stack $$\left[H^0(G,{}_{a_T}A_T/{}_{a_T}B_T)/ H^0(G,{}_{a_T} B_T)\right].$$

If $A$ is normal in $B$ with quotient $C$ and $b_T \in Z^1(G,B(T))$ with image $c_T \in Z^1(G,C_T)$, then the fiber of $H^1(G,B) \to H^1(G,C)$ above the class $c_T \in H^1(G,C_T)$ is given by the stack quotient $\left[H^1(G,{}_{b_T}A_T)/ H^0(G,{}_{b_T} C_T)\right]$.

\end{prop}
\begin{proof}
We just prove the second statement, and the first is identical. The stack-theoretic fiber $P$ above $c_T$ is given by the pullback 
\[
\begin{tikzcd}
P \arrow{d} \arrow{r} & H^1(G,B_T) \arrow{d} \\
T \arrow{r}{c_T} & H^1(G,C_T)
\end{tikzcd}
\]
Since sheafification commutes with pullback, we may compute this pullback on the level of cohomology pseudofunctors and then stackify. But Serre's long exact sequence shows that this is isomorphic to the unstackified $H^1(G,{}_{c_T}A_T)/ H^0(G,{}_{c_T} C_T)$, which finishes the proof. 

\end{proof}

We will actually need a slightly upgraded version of this proposition, in which we twist by a torsor.
\begin{defn}
Let $P \in A_T-Tors_G$, for an object $T \in \mathcal{C}$. Suppose, furthermore, that $A_T$ acts on the left on another group $H$ over $T$, compatible with the $G$-action and the group structure on $H$. Then the twist of $H$ by $P$ is the group $$_PH = P \times_{A_T} H.$$
\end{defn}
As a group, this twist is isomorphic to $H$, but it has the twisted $G$-action $g \cdot (p,h) = (gp, gh).$ One locally sees the isomorphism with $H$ because locally on $T$, $P$ has a point $p \in P(T')$, and one has the bijection $H \to {}_PH,  h \mapsto (p,h)$. The key point is:
\begin{prop}
If $P$ has a point $p$, let $c \colon G \to P(T)$ denote the associated cocycle. Then we have an isomorphism in $A_T-Tors_G$: \[{}_c H \simeq {}_P H\] given by $h \mapsto (p,h)$.
\end{prop}
\begin{proof}
This follows from the equalities \[(p,c(g)\cdot {}^g h) = (p, (p^{-1}){}^g p {}^gh) =  ({}^g p, {}^gh.)\]
\end{proof}

Now the following statement is obtained by locally following through the bijection between twisting by torsors and cocycles. Suppose, as before, that any $A_T$-torsor has a point locally on $T$ for whichever topology we are considering.
\begin{cor}
The statements of Proposition \ref{fibers} hold true when $b_T$ (resp. $c_T$) is replaced with any torsor in the stackified $H^1(G,B_T)$ (resp. $H^1(G,C_T)$), and any cocycle twist is replaced by a torsor twist.
\end{cor}

\section{A review of classical Selmer conditions}

From the ``box'' perspective on arithmetic geometry, Galois cohomology is a good first step: it provides a container for the rational points, and it often has extra internal structure (it may be a group or a scheme) or external structure (it might fit into various exact sequences.) However, experience shows that not all cocycles can arise as cohomology classes in arithmetic situations. The key observation is that cocycles coming from global points are constrained when we look at them as local points.

The original Selmer condition comes from the theory of abelian varieties: Let $A$ be an abelian variety over a number field $F$. We will denote its $l$-adic Tate module by $T_l$, and $V_l \defeq T_l \otimes_{\mathbb{Z}_l} \mathbb{Q}_l$. If $[m]$ denotes multiplication by $m$ on $A$, the exact sequence $$0 \to A[m](\bar F) \to A(\bar F) \overset{[m]} \rightarrow A(\bar F) \to 0$$ yields the long exact sequence $$0 \to A[m](F) \to A(F) \overset{[m]} \rightarrow A(F) \to H^1(G_F, A[m](\bar F)) \to H^1(G_F, A(\bar F)) \overset{[m]} \rightarrow H^1(G_F, A(\bar F)).$$
From this sequence, one learns the short exact sequence $$0 \to A(F)/mA(F) \to H^1(G_F, A[m](\bar F)) \to H^1(G_F, A(\bar F))[m] \to 0.$$ If $H^1(G_F, A[m](\bar F))$ is finite then we have proved the finiteness of $A(F)/mA(F)$, also known as the Weak Mordell Weil theorem. The following result dashes our hopes.
\begin{prop}
Let $M$ be a finite continuous Galois module for $G_F$, the Galois group of a number field $F$. Then $H^1(G_F, M)$ is infinite. In particular, when $A$ is a nonzero abelian variety over $F$, $H^1(G_F, A[m](\bar F))$ is infinite.
\end{prop}
\begin{proof}
We reproduce here a proof from Mikhail Borovoi, posted on MathOverflow \cite{239543}. We make a number of reductions. Let $L$ be a finite extension of $F$ with Galois group $H$ for which the action of $G_L$ on $M$ is trivial. The exact sequence $$0 \to M[p] \to M \to M/M[p] \to 0$$ gives a long exact sequence $$\cdots \to H^0(G_F, M/M[p]) \to H^1(G_F, M[p]) \to  H^1(G_F,M) \to H^1(G_F,M/M[p]) \to \cdots$$ Thus the kernel of the map $H^1(G_F, M[p]) \to  H^1(G_F,M)$ is finite, so we may assume $M = M[p]$, which is $(\mathbb{Z}/p\mathbb{Z})^k$ for some $k$. One now uses inflation restriction with the extensions $\bar F/L$ and $L/F$; since $H^1(H, M^{G_{L}})$ and $H^2(H,M^{G_L})$ are finite ($H$ is a finite group), it suffices to show the result when the Galois action on $M$ is trivial.

Let $S$ be a collection of primes $q$ which are congruent to $1 \mod p$ and split completely in $L$. (Note that by the Chinese remainder theorem, $(\mathcal{O}_L/q)^\times \otimes_{\mathbb{Z}} \mathbb{Z}/p\mathbb{Z} \simeq \mathbb{Z}/p\mathbb{Z}[H]$ as representations.) Now $G_{F,S}$, the maximal extension unramified outside of $S$, is a quotient of $G_F$, so it suffices to show that $H^1(G_{F,S},M) = \Hom(G_{F_S}, (\mathbb{Z}/p\mathbb{Z})^k)$ is unbounded as $S$ varies. If we let $\Gamma_S$ denote the ray class group of $L$ of conductor $S$, then class field theory furnishes us with an exact sequence
$$\mathcal{O}_L^\times \to \prod_{q \in S} (\mathcal{O}_L/q)^\times \to \Gamma_S \to Cl(\mathcal{O}_L) \to 0.$$ (We are again using the congruence condition on $q \in S$ here.) Now both the unit group and the class group are finitely generated, independent of $S$. Thus, if we tensor the above sequence by $\mathbb{Z}/p\mathbb{Z}$, we see using right exactness that $\Gamma_S/p\Gamma_S$ has size $\#S \times \#(\mathbb{Z}/p\mathbb{Z}[H]) - C$, where $C$ does not depend on $S$. 

But then $\Gamma_S/p\Gamma_S$ is $\mathbb{Z}/p\mathbb{Z}$-dual to $\Hom(G_{F_S}, \mathbb{Z}/p\mathbb{Z})$ by the main property of the ray class field. We now pick $S$ as large as we want by Chebotarev, and we have bounded $\Hom(G_{F_S}, \mathbb{Z}/p\mathbb{Z})$ from below, and we are done.
\end{proof}

Undeterred, we look for a subgroup of $H^1(G_F, A[m](\bar F))$ that is finite. There are local counterparts of our short exact sequences every completion of $F$, which lead to a commutative diagram with short exact top and bottom rows
\[
\begin{tikzcd}
 A(F)/mA(F) \arrow{d} \arrow{r} & H^1(G_F, A[m](\bar F)) \arrow{d} \arrow{r} \arrow{dr}{\rho} & H^1(G_F, A(\bar F))[m] \arrow{d} \\
\prod_v A(F_v)/mA(F_v) \arrow{r} & \prod_v H^1(G_{F_v}, A[m](\bar F_v)) \arrow{r} & \prod_v H^1(G_{F_v}, A(\bar F_v))[m]
\end{tikzcd}
\]
(Here $\rho$ is defined as the composition of the middle vertical restriction map and the left lower horizontal arrow.)

\begin{defn}
The $m$-Selmer group of $A/F$ for $F$ a number field is $Sel_m(A/F) \defeq \ker(\rho)$. As $m$ varies, the groups $Sel_m(A)$ form a directed system, and we set $Sel_\infty(A) \defeq \varinjlim Sel_m(A)$, with a similar definition for $Sel_{p^\infty}(A)$ along powers of $p$. Note that $Sel_{\infty}(A) \subseteq  H^1(G_F, A[\infty](\bar F))$. If $F_v$ is a completion at a place of $F$, then we define the local Selmer group by $Sel_m(A/F_v) \defeq \im(A(F_v)/mA(F_v) \to H^1(G_{F_v}, A[m](\bar F_v)))$. By definition, then, the global Selmer group consists of all global cohomology classes whose localization lands in the local Selmer group.
\end{defn}
In words, $Sel_m(A)$ is the set of all global cohomology classes whose restrictions come from local points. By the commutativity of the diagram, $A(F)/mA(F)$ lands in $Sel_m(A)$. Furthermore, if we define $$\Sha(A/F) \defeq \ker \left( H^1(G_F, A(\bar F)) \to \prod_v H^1(G_{F_v}, A(\bar F_v))\right),$$ then there is an exact sequence $$0 \to A(F)/mA(F) \to Sel_m(A/F) \to \Sha(A/F)[m] \to 0.$$ 

Taking the inverse limit of the $l^n$-Selmer groups and tensoring with $\mathbb{Q}_l$ we see that $$\mathbb{Q}_l \otimes \lim_n Sel_{l^n}(A/F) = \im(A(F)\otimes \mathbb{Q}_l \to H^1(G_{F_v}, V_l)),$$ so long as the $l$-part of $\Sha(A/F)$ does not contain a copy of $\mathbb{Q}_l/\mathbb{Z}_l$. This is of course implied by the famous conjecture that $\Sha(A/F)$ is in fact finite.

The proof of concept for $Sel_m(A)$ is the following proposition:
\begin{prop}
$Sel_m(A)$ is finite.
\end{prop}
\begin{pfsketch}
There are two main steps to the proof. First, one shows that the classes in $Sel_m(A)$ are unramified outside of the primes dividing $m$ and the primes of bad reduction for $A$; call the set of all such primes $S$. Then, one shows that the set of cohomology classes unramified outside of a finite set of primes is finite. For this, one reduces to trivial action using inflation-restriction. Using the first step, one can then replace the cohomology of $G_F$ with the cohomology of $Gal(L/F)$, where $L$ is the maximal abelian extension of $F$ with exponent $\# A[m]$ that is unramified outside of $S$. The extension $L/F$ is in fact finite, and it is here that one uses the classic and deep result on the finiteness of the class group. The set of homomorphisms between two finite groups is finite, and the proof concludes. \hfill \qedsymbol
\end{pfsketch}

\begin{rem}
As one can see from the above sketch, one major advantage of the cohomological proofs of diophantine finiteness is their clarity: One often uses the same exact sequences again and again to reduce to a question in algebraic number theory such as the calculation of the class groups of number fields or, in more advanced applications, control theorems in Iwasawa theory and $K$-theory.
\end{rem}

\section{$p$-adic Hodge Theory and Selmer conditions}

The genius of Bloch and Kato was to define a Selmer subgroup of the Galois cohomology of a representation that does not come from an abelian variety. It is indeed their construction from $p$-adic Hodge theory that we will apply to non-abelian Galois representations. Crucially, we need the formalism of crystalline representations in both the absolute case they considered \cite{bloch2007functions} and in the case of families of representations \cite{berger2008familles} We will take here an axiomatic view of $p$-adic Hodge theory, and specifically its period rings; in a whirlwind review such as this, digressions into cotangent complexes and perfectifications can be left to wonderful references like \cite{brinon2009cmi}.

For this section we will let $K$ be a $p$-adic field, $K_0$ the largest unramified subfield of $K$, and $V$ the valuation ring of $K$.

\subsection{Crystallinity over a point and in a family}
 
\begin{defprop}
The period ring $B_{crys} = B_{crys}(V)$ is a filtered flat $K_0$-algebra with a continuous $G_K$-action. Here are some essential facts:
\begin{description}
\item[1] The associated graded of $B_{crys} \otimes_{K_0} K$ is $\bigoplus_{i \in \mathbb{Z}} \mathbb{C}_K(i)$, where $\mathbb{C}_K(i)$ denotes the i-th Tate twist of the completion of the algebraic closure of $K$.
\item[2] There is a $K_0$-semilinear Frobenius map $\phi \colon B_{crys} \to B_{crys}$. It is injective, and the Frobenius-fixed part sits in a short exact sequence that relates it to the de Rham period ring \cite[Proposition 1.17]{bloch2007functions}. 
\item[3] $(B_{crys})^{G_K} = K_0$ 
\item[4] $B_{crys}$ is $(\mathbb{Q}_p,G_K)$-regular, meaning, first, that $$B_{crys}^{G_K} = frac(B_{crys})^{G_K}$$ where $frac$ is the field of fractions of a domain and second, that if the $\mathbb{Q}_p$-line generated by an element $b \in B_{crys}$ is fixed by $G_K$ then $b$ is a unit in $B_{crys}$.
\end{description}
\end{defprop}

The last property seems especially arbitrary, but it actually drives much of the theory. In fact, if $V$ is a $G_K$ representation over $\mathbb{Q}_p$ then we can define a functor $$D_{crys}(V) = (V \otimes_{\mathbb{Q}_p} B_{crys})^{G_K}$$ to the category of vector spaces over $K_0$; the regularity of $B_{crys}$ then has a number of beautiful consequences. The first is the inequality $$\dim_{K_0} D_{crys}(V) \leq \dim_{\mathbb{Q}_p} V.$$ We then say that $V$ is $B_{crys}$-admissible -- or crystalline, more succinctly -- if equality holds above. One proves that an equivalent condition for crystallinity is that the natural map $$D_{crys}(V) \otimes_{K_0} B_{crys} \to V \otimes_{\mathbb{Q}_p} B_{crys} $$ is an isomorphism. We will also use, without further comment, the period ring $B_{dR}$, which is a $K$ algebra with the property that $D_{dR}(V) \simeq D_{crys}(V) \otimes_{K_0} K$ when $V$ is crystalline.

In the case of representations over a field (i.e. the non-relative case), here is a definition of crystalline representations that suits the infinite-dimensional algebras with which we will be working.
\begin{defn}
Let $V/\mathbb{Q}_p$ be a possibly infinite-dimensional linear representation of $G_K$. We say that $V$ is crystalline if $V$ is the direct limit of finite-dimensional subrepresentations which are crystalline in the above sense.
\end{defn}
One would like a definition of infinite-dimensional crystalline representations that mimics the tensor product of the finite-dimensional definition. The following theorem of Olsson gives this intrinsic algebraic characterization.
\begin{prop}{(\cite{olsson2011towards}, Remark D.10)} \label{indreps}
Let $V$ be a possibly infinite dimensional representation of $G_K$. Then $V$ is the union of finite dimensional crystalline subrepresentations (i.e. crystalline in the above sense) if and only if the natural map $$D_{crys}(V) \otimes_{K_0} B_{crys} \to V \otimes_{\mathbb{Q}_p} B_{crys} $$ is an isomorphism.
\end{prop}

Now we turn back to representations of finite rank, but now in a family. These families will help us compare the work of Wang-Erickson on crystalline representations to our Selmer stack.
\begin{defn}
Let $S$ be an affinoid $\mathbb{Q}_p$-algebra.  A family of representations over $S$ of a profinite group $G$ is a locally free $S$-module $V$ with an $S$-linear and continuous (for the $p$-adic topology on $V$) action $G \times V \to V$. 
\end{defn}
The extension of this definition to non-affinoid rigid bases is immediate. Pottharst \cite[Theorem, p. 1]{pottharst2013analytic} brings the following theorem about such families.

\begin{thm} \label{pharst}
Assume $G$ has finite cohomology on all discrete $G$-modules of finite, $p$-power order, vanishing in degrees greater than $e$. Let $V$ be a family of continuous representations of $G$ over a rigid-analytic space $X$. Then the continuous cochains of $G$ with values in the sections $\Gamma(Y, V)$, where $Y \subseteq X$ runs over affinoid subdomains, form a perfect complex of coherent $\mathcal{O}_X$-modules, vanishing in degrees greater than $e$. 
\end{thm}
Denote by $C^\bullet \left(G, V \right)$ the complex in Pottharst's theorem. By taking cohomology of $C^\bullet \left(G, V \right)$, we have that the cohomology of $G$ with values in $\Gamma(Y,V)$ forms a coherent $\mathcal{O}_X$-module as $Y$ varies.

We can emulate the definition of admissibility for families, which we state now for any $G_K$-regular period ring $B$.
\begin{defn}{(\cite[Section 2.3]{berger2008familles})} \label{crystalfamdefn}
Let $V$ be a family of representations of $G_K$ over $S$. Set $D_B(V) = (  (S \hat \otimes B)  \otimes_S V  )^{G_K}$.

We say that $V$ is $B$-admissible if the $S \otimes B^{G_K}$-module $D_B(V)$ is projective of finite type and the map \[ (S \hat \otimes B ) \otimes_{S \otimes B^{G_K}} D_B(V) \to (S \widehat{\otimes} B) \otimes_S V \] is an isomorphism.
\end{defn}
An ind-admissible family of representations is defined in the obvious way, and one can generally set $D_B(W) = \varinjlim D_B(W_{\alpha})$ for an ind-representation $W$ with choice of presentation $W = \varinjlim W_{\alpha}$. Bellovin \cite[Corollary 5.1.16]{bellovin2015padic} shows that families of admissible representations are closed under subobjects, tensor products, and duals, and all three commute with $D_B$.

Setting $B=B_{crys}$, we get a notion of crystallinity for a family. These families form a stack in the usual way, which we denote $\mathcal{R}ep^{crys}$. One very important property of families of Galois representations is that they are crystalline if and only if they are crystalline at all closed points. Here is a precise statement.
\begin{thm}{(\cite[Corollaire 6.3.3]{berger2008familles})} \label{pointwise}

Suppose $S$ is a reduced affinoid algebra. Let $V$ be a family of rank $d$ $G_K$-representations over $S$, $X = \Spm(S)$ and $[a,b]$ an interval in $\mathbb{Z}$ such that for all $x \in X$, $V_x$ is crystalline with Hodge-Tate weights in $[a,b]$. Then $D_{crys}(V)$ is a locally free $S \otimes K_0$-module of rank $d$, and $S/\mathfrak{m}_x \otimes D_{crys}(V) \simeq D_{crys}(V_x)$. In particular, $V$ is crystalline if and only if $V_x$ is crystalline for all $x$.
\end{thm}
Bellovin \cite[Theorem 1.1.3]{bellovin2015padic} extended this theorem to the case where $S$ need not be reduced, and crystallinity can be checked on finite ring extensions of the base field.

\begin{rem} \label{strictmors}
If $V$ is any family of representations, $D_{dR}(V)$ has a natural filtration by the locally free $S \otimes K$-modules induced by the inclusion $D_{dR}(V) \hookrightarrow V \otimes (S \widehat{\otimes} B_{dR})$ and the canonical filtration on $B_{dR}$. If $V$ is de Rham, then this filtration makes the isomorphism \[ (S \widehat{\otimes} B_{dR} ) \otimes_{S \otimes K} D_{dR}(V) \to (S \hat \otimes B_{dR}) \otimes_S V \]  an isomorphism of filtered $(S \widehat{\otimes} B_{dR} )$-modules. (The filtrations of course being by locally free submodules.)

If $V \to W$ is a map of de Rham representations then the morphism $D_{dR}(V) \to D_{dR}(W)$ is strict for the filtrations. This is because we have a diagram
\[
\begin{tikzcd}
V \otimes (S \widehat \otimes B_{dR}) \arrow{r} \arrow{d}{\simeq} & W \otimes (S \widehat \otimes B_{dR}) \arrow{d}{\simeq} \\
D_{dR}(V) \otimes (S \widehat \otimes B_{dR}) \arrow{r} & D_{dR}(W) \otimes (S \widehat \otimes B_{dR})
\end{tikzcd}
\]
where the vertical isomorphisms preserve filtrations; the top map is obviously strict, so the bottom is too. This induces a strict map $D_{dR}(V) \to D_{dR}(W)$ because the subspace filtration is exactly the filtration we have just put on $D_{dR}$.
\end{rem}

Finally, we record the following definition.
\begin{defn}{ \cite[Section 2.6]{kisin2008potentially}} 
Let $E$ and $K$ be $p$-adic fields, and $D_E$ a finite-dimensional $E$-vector space equipped with a filtration $Fil^i$ on $D_{E,K} = D_E \otimes_{\mathbb{Q}_p} K.$ Consider the collection of vector spaces $$\bold{v} = \{ D_E, Fil^i D_{E,K} \}.$$ Now let $C$ be a finite $E$-algebra, and suppose $V_C$ is a $d$-dimensional $G_K$-representation over $C$.

Then we say that $V_C$ is of $p$-adic Hodge type $\bold{v}$ if for all $i$ we have an isomorphism of $C \otimes_{\mathbb{Q}_p} K$-modules $$gr^i\Hom_{B[G_K]}(V_C, B_{dR} \otimes_{\mathbb{Q}_p} C) \simeq gr^i D_{E,K} \otimes_E C.$$   
\end{defn}
In particular, all representations of a given $p$-adic Hodge type have uniformly bounded Hodge--Tate weights. For us, the important example of Hodge types will come from the family of Galois representations coming from the Gauss--Manin connection. In this case, all stalks have the same $p$-adic Hodge type since embedding $K$ into $\mathbb{C}$ allows us to use the complex statement: the Gauss--Manin connection is a variation of Hodge structure, and all stalks of a variation of Hodge structures have the same Hodge type.

One defines the concept of $p$-adic Hodge type for a family of representations using closed points. Packaging all of these conditions together, we have:
\begin{defn}
We denote $\mathcal{R}ep^{crys, \bold{v}}$ the stack of all crystalline families of $G_K$-representations of $p$-adic Hodge type $\bold{v}$.
\end{defn}

\subsection{Stacky Selmer conditions}

We can now apply this definition to our discussion of torsors and cohomology classes. The reader who is not familiar with formal or rigid geometry should refer to Chapter \ref{geometry}.

\begin{defn} \label{bk}
Let $K$ be a $p$-adic field and $N$ a group scheme over $\Spec \mathbb{Q}_p$ on which $G_K$ acts continuously for the $p$-adic topology. Consider the pseudofunctor consisting of 2-kernels $$\Lambda \mapsto \ker \left(H^1(G_K,N(\Lambda)) \to H^1(G_K,N(\Lambda \widehat{\otimes}_{\mathbb{Q}_p} B_{crys})) \right).$$ Here $\Lambda$ is a test object in the category of affinoid $\mathbb{Q}_p$-algebras. We define $H^1_f(G_K,N)$ to be the stackification of this pseudofunctor.

When $l\neq p$ and $N$ is a vector space, we define $$H^1_f(G_K,N) \defeq \ker \left(H^1(G_K,N) \to H^1(G_{K^{nr}},N)\right).$$ By this we mean the fibered product as stacks of the map inside of the kernel with the $\mathbb{Q}_l$-point of $H^1(G_{K^{nr}},N)$ given by the trivial torsor. (Here $K^{nr}$ is the maximal unramified extension of $K$.)

\end{defn}
Note that the Galois action on $H^1(G_K,N(\Lambda \widehat{\otimes}_{\mathbb{Q}_p} B_{crys}))$ incorporates the Galois action on $B_{crys}$ and the Galois action on $N$; it is defined as in Example \ref{galactdefn}.

We will only use the latter case when $N$ is a vector space. We can patch these local conditions together for global representations.
\begin{defn}
Let $F$ be a global field, $p$ a rational prime, $T$ a set of primes of $F$ including all primes lying over $p$. Then for $N$ an algebraic group over $\mathbb{Q}_p$ on which $G_{F, T}$ acts continuously, we define the \textit{global Selmer stack} $$H^1_f(G_{F,T},N) \defeq H^1(G_{F,T},N) \times_{\prod_{v \in T, v|p} H^1(G_{F_v},N)} \prod_{v \in T, v|p} H^1_f(G_{F_v},N).$$

\end{defn}

(The category of stacks has fibered products, and so we see that these objects are indeed stacks.)

\begin{rem} \label{selmerdescription}
One has the following (rather finnicky) interpretation of these rigid-analytic Selmer stacks, which follows from unwinding the stackification. Let $P/Z$ be an $N^{an}$-torsor, where $Z$ is a rigid-analytic space. The question of whether $P$ lies in the Selmer part is local for the Tate topology, so we may assume that $Z = \Spm(\Lambda)$ and $P(\Lambda)$ is nonempty. Thus $P$ is determined by a cocycle in $H^1(G_K, N^{an}(\Lambda))$. But since $N^{an}(\Lambda) = N(\Lambda)$,  $P$ is determined by an algebraic $N$-torsor with continuous $G$-action $P^{alg}$ over $\Spec(\Lambda)$. Now $P$ is in the Selmer part if and only if we have the usual Galois-equivariant isomorphism \[P^{alg} \otimes ( \Lambda \widehat{\otimes} B_{crys}) \simeq N \otimes( \Lambda \widehat \otimes B_{crys}).\]
\end{rem}

\begin{rem} \label{etaleremark}
One might also want to impose conditions at primes away from $p$, for example the condition that they are in the image of a map $$X(\mathbb{Z}_l) \to H^1(G_v, N).$$ We do not deal with such conditions in any generality here, except for a weight condition when proving the algebraicity of the Selmer stack.
\end{rem}

\begin{rem}
One would generally like to use the \'etale topology for this stackification, in which case no issues of trivializing the Galois cohomology of $S$ or $S^{\phi=1}$ would arise. Abstractly, this definition of cohomology and Selmer stacks would be just fine (although it would sacrifice the representability of torsors, since \'etale descent for rigid torsors seems hard.) However, the arguments of Theorem \ref{representability} do require one to glue objects in the topology in which one stackified, so the representability argument would become more difficult.
\end{rem}

\subsection{Agreement with the abelian case}
For the benefit of the reader, we briefly explain why the Bloch--Kato and classical Selmer groups agree in the case of an abelian variety $A$ over a number field $F$. We may reduce to the local case, because both Selmer groups are defined by local conditions: pick $p$ a rational prime such that $A$ has good reduction at all places dividing $p$. 
Let $T_p$ be the $p$-adic Tate module of $A$, and $V_p = T_p \otimes \mathbb{Q}_p$ (or $V$ when there is no risk of confusion.) We claim that $$\lim_n Sel_{p^n}(A/F_v) \otimes \mathbb{Q}_p = H^1_f(G_{F_v}, V_p),$$ both groups lying inside $H^1(G_{F_v}, V_p)$, for any place $v$ of $F$.

The easier case is when $v \not | p$. By definition $H_f^1(G_{F_v}, V) = \ker(res \colon H^1(G_{F_v},V) \to H^1(I_{F_v},V))$. We will show that this group is zero. By inflation restriction, the kernel of this map is $H^1(G_{F_v}/I_{F_v}, V^{I_{F_v}})$. When $A$ has good reduction, this group is equal to the Galois cohomology of the cohomology of the special fiber and the Weil conjectures say that this representation has negative weight, and thus $H^0(G_{F_v}/I_{F_v}, V^{I_{F_v}}) = 0$; by a result of Serre vanishing of zeroth cohomology implies vanishing of first cohomology.
Thus $H^1(G_{F_v}/I_{F_v}, V^{I_{F_v}})=0$ then as well. In the case of bad reduction, the fact that the inertia fixed part corresponds to the part of the special fiber that is an abelian variety \cite[Section 2]{grothendieck2006groupes} implies again by the Weil conjectures that $H^1(G_{F_v}/I_{F_v}, V^{I_{F_v}}) = 0$ using weights exactly as above.

We now show that $Sel_{l^\infty}(A/F) = 0$. This group is defined as the vector space spanned by images of rational points in the Galois cohomology. But a classical theorem of Mattuck \cite{mattuck1955abelian} says that $A(F_v)$ contains a finite index subgroup that is isomorphic to the valuation ring of $F_v$.
 In particular, $A(F_v)$ has a finite index subgroup which is a pro-$p$ group. But then tensoring it with $\mathbb{Q}_l$ kills this pro-$p$ part and any torsion. We thus have that the image of the local points under the Kummer map is trivial, and we are done when $v \not | p$.

When $v | p$ things are much more difficult -- here we follow \cite[Example 3.11]{bloch2007functions}. The key is a commutative square
\[
\begin{tikzcd}
tan(A) \arrow{r}{exp} \arrow{d}{comp} & A(F_v) \otimes \mathbb{Q} \arrow{d}{\kappa} \\
D_{dR}(V)/F^0 D_{dR}(V) \arrow{r}{BK-exp} & H^1_f(G_{F_v}, V) 
\end{tikzcd}
\]
where $tan(A)$ denotes the tangent space of $A$. The maps in the above square are as follows. The map $comp$ is an isomorphism that arises from Faltings' de Rham comparison theorem. Namely, $$D_{dR}(V) \simeq H_1^{dR}(X, F_v),$$ and removing the 0th filtered piece of the Hodge filtration leaves only $H^0(X,\Omega^1_{F_v})^*$ -- this is precisely the tangent space. The top horizontal arrow is also an isomorphism, since $p$-adic exponential and logarithm are both local diffeomorphisms (but a linear map that is a local diffeomorphism is an isomorphism.) Finally, one shows that the Bloch-Kato exponential is an isomorphism, with explicit inverse given by Fontaine's functor $D_{dR}$. 
Since those three maps are isomorphisms, we conclude that the Kummer map $\kappa$ is an isomorphism too.

We would be amiss if we did not mention that this whole crystalline story can be done integrally, i.e. by looking at each representation $A[p^n](\bar F_v)$ individually. For this one can use the theory of Fontaine--Laffaille \cite{fontaine1982construction} when the ramification of $K$ is not too large.

\subsection{Rational points land in Block--Kato stacks}
In the case of abelian varieties, rational points landed in the Selmer group because of a crystalline comparison theorem.
For general varieties, the situation is much more complicated. Using our definition, we find that a torsor $P \in H^1(G_K, N)$ is crystalline (i.e.\ lives in $H_f^1$) exactly when there is an ordinary $N$-torsor $P'$ and a Galois-equivariant isomorphism $$P' \otimes_{\mathbb{Q}_p} B_{crys} \simeq P \otimes_{\mathbb{Q}_p} B_{crys}.$$
Just like in the case of abelian varieties, the non-abelian statement that the crystalline path torsor $P^{\acute{e}t}_{b,x}$ lies in $H^1_f(G,\mathcal{G}^{et})$ uses the full force of $p$-adic Hodge theory. In fact, we'll need the crystalline path torsor comparison theorem of Olsson.
\begin{thm}{\cite[Theorem 1.11]{olsson2011towards}}
Let $X$ be a variety over $K$ which is a complement of a normal crossings divisor in a smooth, proper scheme. Fix $\mathcal{L}^{et}$ and $\mathcal{L}^{crys}$ on $\bar X$ and $X_0$ which are associated in the sense of $p$-adic Hodge theory. Suppose that the Tannakian monodromies of both $\mathcal{L}^{et}$ and $\mathcal{L}^{crys}$ are reductive and locally nilpotent. 
Then for all $x, y \in X(R)$, we have a Galois-equivariant isomorphism of path torsors $$P^{\acute{e}t}_{x,y} \otimes_{\mathbb{Q}_p} B_{crys} \simeq P^{crys}_{x,y} \otimes_{K_0} B_{crys}.$$ Equivalently, $\mathcal{O}(P^{\acute{e}t})$ is an (ind-)crystalline Galois representation and $D_{crys}(\mathcal{O}(P^{\acute{e}t})) \simeq \mathcal{O}(P^{crys})$.
\end{thm}

Here is our statement, then.
\begin{prop}
Let $X/\mathcal{O}_{F,\Sigma}$ be a smooth complement of a normal crossings divisor relative to $\mathcal{O}_{F, \Sigma}$. Take $p$ a prime outside of $\Sigma$ and $T = \Sigma \cup \{v ; v|p\}$. Let $\mathcal{L}^{et}$ be a $\mathbb{Q}_p$-local system on $X_{\mathcal{O}_{F,T}}$ that is crystalline on $X_{\bar{F}_v}$ for all places $v|p$ of $F$ and has locally unipotent monodromy. Let $\mathcal{G}^{et}$ denote the relative \'etale fundamental group of $X_{\bar F}$ with respect to $\mathcal{L}^{\acute{e}t}$. Let $\mathcal{L}^{crys}$ be the associated isocrystal, and $\mathcal{G}^{crys}$,$S^{crys}$ its relative completion and monodromy.

Suppose that $S^{crys}$ is semisimple and simply connected. Then the pushout to $\mathcal{G}_n^{\acute{e}t}$ of the restriction of the path torsor $P_{b,x}^{\acute{e}t}$ to $G_{F_v}$ lands inside $H^1_f(G_{F_v}, \mathcal{G}_n^{\acute{e}t})$ for $v|p$. Furthermore, $G_F$ acts on the relative \'etale $\mathbb{Q}_p$-fundamental group of $X_F$ through the quotient $G_{F,T}$. In particular, the pushouts to $\mathcal{G}_n^{\acute{e}t}$ of the global path torsors $P_{b,x}^{\acute{e}t}$ sit in the Selmer stack $H^1_f(G_{F,T}, \mathcal{G}_n^{\acute{e}t})$.
\end{prop}

\begin{rem}
By the statement $P_{b,x}^{\acute{e}t}$ lands in $H^1_f$, we of course mean that the rigid analytification of $P_{b,x}^{\acute{e}t}$ lands in $H^1_f$, which classifies rigid-analytic torsors. Since, as we note below, $(P_{b,x}^{\acute{e}t})^{an}$ is completely determined by $P_{b,x}^{\acute{e}t}$, there is minimal risk of confusion.

Note also that $S_{crys}$ is semisimple and simply connected if and only if $S^{dR}$ is, by \cite[Lemma 4.16, Section 4.12]{olsson2011towards}.
\end{rem}

\begin{proof}
We establish that $G_F$ acts through its quotient $G_{F,T}$; i.e. we need to show that the action of $G_{F_w}$ on $\mathcal{G}^{et}$ is in fact unramified when $w \not \in T$. We use the Adams spectral sequence and the corresponding structure theorem for $\mathcal{G}^{\acute{e}t}$, as follows.

The Galois action on $S^{\acute{e}t}$ is given by conjugation by the Galois action on the fiber $\mathcal{L}^{\acute{e}t}_b$, while the Galois action on the unipotent part is given by the Galois action on \[H^1_{\acute{e}t}(\overline{X},\mathcal{O}(S^{\acute{e}t}))^* \oplus H^0_{\acute{e}t}(\overline{X^c}, R^1j_*\mathcal{O}(S^{\acute{e}t}))^*\] (here $X^c$ is some smooth compactification of $X$ over $\mathcal{O}_{F,\Sigma}$.)

Now the unramifiedness of the Galois action on $\mathcal{L}^{\acute{e}t}_b$ follows by definition, since $\mathcal{L}^{\acute{e}t}$ is defined over the whole of $\mathcal{O}_{F_w}$; thus the Galois action of $G_{F_w}$ on it factors through $\pi_1^{\acute{e}t}(\mathcal{O}_{F_w}),$ i.e. it is unramified.

The corresponding statement for the unipotent part is a consequence of the smooth and proper base change theorems, in fact a standard consequence if $X$ is proper. In the non-proper case there is a ``Deligne Hodge II-style'' spectral sequence involving the components of $X^c \setminus X$ that abuts to the cohomology group \[H^1_{\acute{e}t}(\overline{X},\mathcal{O}(S^{\acute{e}t}))^*.\] Furthermore, standard Tannakian theory tells us that the sheaf $\mathcal{O}(S^{\acute{e}t})$ is in $ind \langle \mathcal{L}^{\acute{e}t} \rangle_{\otimes}$. In particular, it is defined over the whole of $\mathcal{O}_{F_w}$.

The $E_1$ terms in this spectral sequence involve the cohomology of smooth and proper schemes over $\mathcal{O}_{F_w}$, and a well-known argument with smooth and proper base-change implies that the Galois action on each of these cohomology groups is unramified. From this we see that the Galois action on $H^1_{\acute{e}t}(\overline{X},\mathcal{O}(S^{\acute{e}t}))^*$ is unramified. A completely similar argument gives unramifiedness for $H^0_{\acute{e}t}(\overline{X^c}, R^1j_*\mathcal{O}(S^{\acute{e}t}))^*$, since $X^c$ is already proper.

Now on to the Selmer part. Since there is no risk of confusion, we also write $P^{\acute{e}t}_{b,x}$ for the pushout of the path torsor to the quotient $\mathcal{G}_n^{\acute{e}t}$.  Recall that the torsors $P^{\acute{e}t}_{b,x}$ are trivial when we forget the Galois action, by \ref{nonempty}. By Remark \ref{selmerdescription}, then, it suffices to give a Galois-equivariant isomorphism \[P^{\acute{e}t}_{b,x} \otimes B_{crys} \simeq P^{\acute{e}t}_{b,b} \otimes B_{crys}.\]

Olsson's comparison theorem gives a Galois-equivariant isomorphism $$P^{\acute{e}t}_{b,x} \otimes_{\mathbb{Q}_p} B_{crys} \simeq P^{crys}_{b,x} \otimes_{K_0} B_{crys}.$$ We will see that it suffices to provide an isomorphism $$P^{crys}_{b,x} \simeq P^{crys}_{b,b},$$ i.e. a point of $P^{crys}_{b,x}$.

By Kneser's theorem \cite[Section 3.1]{ion2001galois} the Galois cohomology of $S^{crys}$ is trivial; because the Galois cohomology of unipotent groups is trivial, every torsor for $\mathcal{G}^{crys}_n$ has a point. We conclude that $P^{crys}_{b,x}(K_0)$ is nonempty. We use that point in the second isomorphism of the following string; every isomorphism that says $comp$ is induced by Olsson's comparison theorem.
\begin{align*}
B_{crys} \otimes  P_{b,b}^{et} \overset{comp} \simeq B_{crys} \otimes P_{b,b}^{crys} \\ \simeq  B_{crys} \otimes P_{b,x}^{crys} \\ \overset{comp}  \simeq  B_{crys} \otimes P_{b,x}^{et}
\end{align*}
This isomorphism finishes the proof of the statement for the local Selmer group, since we see that $P_{b,x}^{et}$ is in the Selmer stack. Now the global statement follows by its definition as fibered product.

\end{proof}

\begin{rem}
In the first part of this proof we are tempted to use the fact that $P^{\acute{e}t}_{b,x}(F_v)$ is non-empty, i.e. that \'etale paths always exist \cite[Proposition 5.5.1]{szamuely2009galois}. This path would induce a second path $$\gamma \in P^{crys}_{b,x} \otimes B_{crys}$$ by the comparison isomorphism, thus exhibiting an isomorphism $$P^{crys}_{b,x} \otimes B_{crys} \simeq P^{crys}_{b,b} \otimes B_{crys}.$$ Unfortunately this isomorphism is not Galois-equivariant, since $p$ likely has a nontrivial Galois action and equivariance of the previous isomorphism is equivalent to $\gamma$ being fixed by the Galois action.
\end{rem}

\chapter{Representability of Cohomology Stacks}

The \'etale realization of the relative completion has furnished us with an extension of $G$-groups, defined over $\mathbb{Q}_p$: $$0 \to \mathcal{U} \to \mathcal{G} \to S \to 0$$ where $\mathcal{U}$ is pro-unipotent, $S$ is reductive, and $G$ is profinite. The first natural question in the nonabelian arithmetic philosophy is: What is the structure of $H^1(G,\mathcal{G})$? In Theorem \ref{representability} we answer this question from an abstract point of view, assuming only mild conditions on the objects present. The strategy hinges on the fact that the stack $H^1(G,\mathcal{G})$ fits into a number of nice exact sequences whose terms are either treated by the work of Pottharst (via the space $H^1(G,\mathcal{U})$) or need to be manipulated to look like representation spaces of $G$ (these pieces come from $H^1(G,S)$.)

See \ref{stratrep} for the notion of pro-strat-representability.
\begin{thm}\label{representability}
Let $$0 \to \mathcal{U} \to \mathcal{G} \to S \to 0$$ be an extension of a reductive group over $\mathbb{Q}_p$ by a pro-unipotent group over $\mathbb{Q}_p$, and suppose that a profinite group $G$ acts continuously for the $p$-adic topology on the entire short exact sequence. 

Suppose that a substack $H^{1,type}(G,S) \subseteq H^{1}(G,S)$ is strat-representable by rigid-analytic stacks of finite type. Denote by $H^{1,type}(G, \mathcal{G})$ the fibered product  $$H^{1}(G, \mathcal{G}) \times_{H^1(G,S)} H^{1,type}(G,S).$$ Then the tower of cohomology functors $\{H^{1,type}(G, \mathcal{G}_n)\}$ is pro-strat-representable in the category of rigid-analytic stacks over $\mathbb{Q}_p$. 

\end{thm}
\begin{rem}
The stacks $H^{1,type}$ and $H_f^{1,type}$ will be a certain substack corresponding to Galois representations with a fixed Hodge type, Galois type, residual pseudorepresentation, and weights.
\end{rem}

We proceed by induction. By assumption, $H^{1,type}(G,S)$ is representable in the category of rigid analytic stacks over $\mathbb{Q}_p$. Suppose that the tower \[H^1(G, \mathcal{G}_{n-1}) \to \cdots \to H^1(G,S)\] is strat-representable. By definition of strat-representability there exist finitely many rigid-analytic substacks $\Sigma_{1}, \cdots, \Sigma_{l}$ such that \[| H^1(G, \mathcal{G}_{n-1}) | = \left| \coprod_i \Sigma_i \right|\] in the category of stacks in groupoids over $\mathbb{Q}_p$.

Consider the morphism of stacks in groupoids over the site $Rig_{\mathbb{Q}_p}$ given by \[q \colon H^{1,type}(G,\mathcal{G}_n) \to H^{1,type}(G,\mathcal{G}_{n-1})\] and its restrictions to \[q_i \colon H^{1,type}(G,\mathcal{G}_n)|_{\Sigma_i} \to \Sigma_i.\]
\begin{lem}
The image of $q_i$ is a locally closed rigid-analytic substack of $\Sigma_i$.
\end{lem}
\begin{proof}
For any $\Lambda$-point $c \in \Sigma_i$, we have that ${}_cZ_n$ (the twist is by conjugation) is a family of Galois representations over $\Lambda$. By Theorem \ref{pharst}, the $\Lambda$-module $H^2(G, {}_cZ_n)$ is coherent. Choose lifts $\tilde c_i$ of $c$ to $Z^1(G, \mathcal{G}_{n-1})$ locally for the Tate topology. Now $\Delta(\tilde c_i)$ is a section of a coherent sheaf on a rigid affinoid, so we may take its zero locus. This zero locus does not depend on the choice of the lifts $\tilde c_i$, by Serre's work. 

Finally, by functoriality of $\Delta(c)$, these zero loci glue to a locally closed substack of $\Sigma_i$.
\end{proof}

We can now demonstrate Theorem \ref{representability}.
\begin{proof}
Since $image(q_i)$ is locally closed, we may replace $\Sigma_i$ with $image(q_i)$ and it does not affect strat-representability. Thus we assume that none of the fibers of $q_i$ are empty.

We will be done if we can give a stratification $\Sigma_{i,j}$ of  $\Sigma_i$ such that \[H^{1,type}(G,\mathcal{G}_n) \times_{H^{1,type}(G,\mathcal{G}_{n-1})} \Sigma_{i,j} \] is representable by a rigid stack. Now, for $\Lambda$ an affinoid $\mathbb{Q}_p$-algebra, the fiber above any $T = \Spm (\Lambda)$-torsor $c_T \in \Sigma_i \subseteq H^{1,type}(G,\mathcal{G}_{n-1})(T)$ is locally in the Tate topology given by the stack in groupoids $\left[H^1(G, {}_{b_T}Z_{n, T})/ H^0(G, {}_{c_T} \mathcal{G}_{n-1,T})\right]$ from Proposition \ref{fibers}. Here $b_T$ is a local lift of $c_T$.

Note that ${}_{c_T}Z_{n, T}$ is a family of abelian representations of $G$ over $\Spm T$. Now the work of Pottharst (Theorem \ref{pharst}) implies that $H^1(G, {}_{c_T}Z_{n, T})$ is a coherent sheaf over $\Spm T$. By gluing in the Tate topology, the association \[(T \to \Sigma_i) \mapsto H^1(G, {}_{c_T}Z_{n, T})\] is a coherent sheaf on $\Sigma_i$. We denote it by $H^1(G, {}_{c}Z_{n})$ to remind us of the twist that is present.

Let $\Sigma'_{i,j}$ be a flattening stratification of $H^1(G, {}_{c}Z_{n})$, i.e. a stratification such that the restriction of the sheaf to each stratum is locally free. (The existence of such a stratification can be argued using formal models and \cite[Corollary 3.7]{bosch1993formal}). By the finite-type assumption on the base, this stratification can be taken to be finite. Now $\mathscr{L} = H^1(G, {}_{c}Z_{n}) \otimes_{H^{1,type}(G,\mathcal{G}_{n-1})} \Sigma'_{i,j}$, being locally free, is thus relatively representable over $\Sigma'_{i,j}$ (\ref{locallyfree}.)

Now we deal with the whole quotient. For any $c_T$ as above, the functor $H^0(G, {}_{c_T} \mathcal{G}_{n-1,T})$ is representable by a scheme. Indeed, for any $g \in G$, the fixed points of $g$ form a closed subscheme of ${}_{c_T} \mathcal{G}_{n-1,T}$ because $G$ acts by algebraic automorphisms; the entire $H^0$ is then the intersection of all of these closed subschemes, which is again a closed subscheme. This description as the intersection of fixed points (which are pullbacks) makes it clear that $H^0$ commutes with arbitrary pullback along morphisms $T' \to T$, and in particular it commutes with restriction to admissible affinoid subsets. Therefore, gluing (the analytifications of) these together in the Tate topology gives us a rigid group $H^0(G, {}_{c} \mathcal{G}_{n-1})$ over $\Sigma_i$. Now Noetherianness (see, for example, \cite[Theorem 8.3]{rydh2016approximation}) implies that there is a finite stratification $\Sigma''_{i,j}$ of $\Sigma_i$ such that $H^0(G, {}_{c} \mathcal{G}_{n-1}) \times \Sigma''_{i,j} \to \Sigma''_{i,j}$ is flat for all $j$.

Finally, let $\Sigma_{i,j}$ be a common refinement of $\Sigma'_{i,j}$ and $\Sigma''_{i,j}$, and $c_T \in \Sigma_{i,j}(T)$. We see that for each $i,j$, the stack \[\left[H^1(G, {}_{c_T}Z_{n})/ H^0(G, {}_{c_T} \mathcal{G}_{n-1})\right] \times_{\Sigma_i} \Sigma_{i,j}\] is a quotient of a rigid analytic space by a flat rigid group. Each such quotient is representable in the category of rigid-analytic stacks by \ref{flatquotient}, and so the total space \[\left[H^1(G, {}_{c}Z_{n})/ H^0(G, {}_{c} \mathcal{G}_{n-1})\right] \times_{\Sigma_i} \Sigma_{i,j}\] is representable by Proposition \ref{transfer}.

\end{proof}

Let us see what it would take to prove the Selmer stack analogue of Theorem \ref{representability}. Write $B$ for $B_{crys}$.

By definition of the Selmer part, we have a morphism of stacks on the rigid \'etale site, $$H^{1, type}_f(G,\mathcal{G}_n) \to H^{1,type}_f(G,\mathcal{G}_{n-1}).$$ By assumption the latter stack is strat-representable.

We proceed to fix some $n \geq 1$ and analyze the diagram of pseudofunctors
\[
\begin{tikzcd}[column sep=small]
H^0(G,\mathcal{G}_{n-1}(T)) \arrow{d} \arrow {r} & H^1(G, Z_n(T)) \arrow{d} \arrow {r} & H^1(G, \mathcal{G}_n(T)) \arrow{d} \arrow {r} & H^1(G, \mathcal{G}_{n-1}(T)) \arrow{d} \\
H^0(G,\mathcal{G}_{n-1}(T \widehat \otimes B))  \arrow {r} & H^1(G, Z_n(T \widehat \otimes B))  \arrow {r} & H^1(G, \mathcal{G}_n(T \widehat \otimes B))  \arrow {r} & H^1(G, \mathcal{G}_{n-1}(T \widehat \otimes B))
\end{tikzcd}
\]
Consider a torsor $c \in H^{1,type}_f(G, \mathcal{G}_{n-1})(T)$. What is the fiber $\mathscr{F}$ above $c$ inside of $H^{1,type}_f(G, \mathcal{G}_n)$?
By using the long exact sequences in cohomology above and Proposition \ref{fibers} this fiber is $(f_{1})^{-1} \circ im(f_2)$ in the following diagram.
\[
\begin{tikzcd}
 & H^1(G,{}_b Z_n(T))/H^0({}_c \mathcal{G}_{n-1}(T)) \arrow{d}{f_1} \\
H^0(G,{}_c \mathcal{G}_{n-1}(T \otimes B)) \arrow{r}{f_2} &  H^1(G, {}_b Z_n(T \otimes B)) 
\end{tikzcd}
\]
where $b$ is an element of the fiber, which locally exists if the fiber is nonempty.

Applying the exact sequence again, the image of $f_2$ is isomorphic to \[\mathcal{H}_1 = [H^0(G,{}_c \mathcal{G}_{n-1}(T \widehat \otimes B))/H^0(G,{}_b \mathcal{G}_{n}(T \widehat \otimes B))],\] which we will see in the next section is representable if $b$ is crystalline. Now denote by $\mathcal{H}_2$ the image of $H^1(G, {}_bZ_n)$ in $H^1(G, {}_bZ_n \otimes (T \widehat \otimes B))$. We see that $\mathscr{F}$ is a bundle with fiber $[H^1_f(G, {}_b Z_N)/H^0({}_c \mathcal{G}_{n-1})]$ over the ``intersection'' \[\mathcal{H}_1 \cap \mathcal{H}_2.\]

As we see, the representability of $\mathscr{F}$ is nearly the question of whether $H^1_f$ forms a coherent sheaf for analytic families of Galois representations; but we have the intervening intersection, which we would need to prove representable by a stack.

\begin{rem} \label{challenges}
The reader should note here that the pro-strat-representability of $H^1_f(G, \mathcal{G})$ when $G = G_K$ for $K$ a local field is thus closely related to the question of whether Bloch--Kato Selmer groups of analytic families of representations form coherent sheaves. Such a result is not known in the crystalline case, although the de Rham and Hodge--Tate cases were proven by Shah \cite{shah2018interpolating}. Pottharst also proves a coherence result for cohomology of families of $(\phi, \Gamma)$-modules. It might be interesting to impose ordinarity conditions on the reductive level of the relative completion and to try to use Pottharst's result in this case.
\end{rem}

\chapter{Cohomology Stacks of Groups with Conjugation Action}
\section{Representability of stacks without Selmer conditions}

We would now like to prove that suitable substacks of $H^1(G,S)$ and $H^1_f(G,S)$ are strat-representable as rigid-analytic stacks. The easiest situation is when $S$ is abelian, i.e. $S=\mathbb{G}_m$; then the cohomology set is representable by a vector space. In any other case, things are more complicated. Our study is guided by the following theorem of Chenevier. Let $Rep^{ss}_n(G)$ be the functor which associates to a rigid-analytic space $X/\Spm(\mathbb{Q}_p)$ the set of locally free $\mathcal{O}_X$-modules of rank $n$ which are acted on $\mathcal{O}_X$-linearly, continuously, and semisimply by $G$, up to isomorphism in the usual sense of families of representations.

We remind the reader that a profinite group $G$ satisfies Mazur's $\Phi_p$ condition for a prime $p$ if for every open subgroup $G_0$ of finite index in $G$, $\Hom(G_0, \mathbb{F}_p)$ is finite. There are two important examples of fields which satisfy this finiteness condition: The first is the absolute Galois group of a $p$-adic field. The second example is given by the groups $G_{K,\Sigma}$ -- here $\Sigma$ is any set of primes of a number field $K$ and $G_{K,\Sigma}$ denotes, as usual, the Galois group of the maximal extension ramified only above $\Sigma$.
\begin{thm}{\cite[Theorems A-G]{chenevier2014p}}
Suppose that $G$ is a profinite group that satisfies Mazur's $\Phi_p$-condition. Then the functor on rigid analytic spaces over $\mathbb{Q}_p$ given by $$Rep^{ss}_n(G) \colon Rig \to Set$$ is representable by a rigid analytic space. Furthermore, that rigid analytic space is a disjoint union of rigid spaces indexed by the isomorphism class of the residual representation.
\end{thm}
For our purposes, the best way to think of $Hom(G_K, \text{GL}_n)$ is as $Z^1(G_K, \text{GL}_n)$ when the action of $G_K$ is trivial. We will end up using closely related results, but non-semisimple representations abound in the contexts we care about -- for example, even assuming Faltings' Theorem on semisimplicity of global Galois representations attached to abelian varieties, the local restrictions of these global representations will no longer be semisimple in general. For this purpose we will require a strengthening of Chenevier's result by Wang-Erickson. In order to introduce his result we will need to introduce the concept of a pseudorepresentation, which plays a key role in the work of both Chenevier and Wang-Erickson.

\begin{defn}
Let $R$ be an $A$-algebra. A $d$-dimensional pseudorepresentation $D$ of $R$ over $A$ is a homogeneous multiplicative polynomial law $$D \colon R \to A,$$ i.e. for each $A$-algebra $\Lambda$ a \textit{function} $$D_\Lambda \colon R \otimes_A \Lambda \to \Lambda$$ satisfying
\begin{description}
\item[1] $D_\Lambda$ sends 1 to 1, and is multiplicative with respect to the multiplication on $R \otimes_A \Lambda$.
\item[2] $D_\Lambda$ is homogeneous of degree $d$ for the $\Lambda$-module structure on $R \otimes_A \Lambda$.
\item[3] $D_\Lambda$ is functorial in $\Lambda$.
\end{description}
\end{defn}
\begin{rem}
We follow Wang-Erickson's terminology; Chenevier calls these objects determinants.
\end{rem}

If $\rho \colon k[G] \to End_k(V)$ is a representation then $\det(\rho)$ is a pseudorepresentation of $k[G]$ over $k$. If $k$ is algebraically closed, there is a precise sense in which the pseudorepresentation of a representation determines the semisimplification of the representation \cite[Theorem A]{chenevier2014p}.

Let $\bar D$ be a residual pseodorepresentation of $G$, i.e. a pseudorepresentation of $\mathbb{Z}_p[G]$ over $\mathbb{F}_p$. Chenevier proves that formal deformations of that pseudorepresentation are parameterized by a universal deformation ring $R_{\bar D}$ \cite[Proposition E]{chenevier2014p}. The more general statement, proved by Wang-Erickson \cite[Theorem A]{wang2017algebraic}, is as follows.

In the following statement, we let $\Lambda$ denote any quotient of a ring of the form $\mathbb{Z}_p[[t_1, \ldots t_n]] \langle z_1, \ldots , z_m \rangle$, where $n$ and $m$ are arbitrary and finite.
\begin{thm}
Denote by $\mathcal Rep^f_{\bar D}$ the groupoid which attaches to any $\Lambda$ as above the category of locally free $\Lambda$-modules $V_\Lambda$ equipped with a continuous linear action $\rho_\Lambda \colon G \to \text{Aut}_\Lambda(V_\Lambda)$ having residual pseudorepresentation $\bar D$. If $G$ satisfies Mazur's finiteness condition $\Phi_p$, then $R_{\bar D}$ is Noetherian and $$\hat \psi \colon \mathcal Rep^f_{\bar D} \to \Spf R_{\bar D}$$ is a formally finite type $\Spf R_{\bar D}$-formal algebraic stack. 
\end{thm}
\begin{rem}
In particular, the rigid generic fiber of $\mathcal{R}ep^f$ is thus a rigid stack over $\mathbb{Q}_p$ by our earlier discussions; see also \cite{ulirsch2017tropicalization} for further work on analytification.

In fact, $(\mathcal{R}ep^f)^{rig} = \mathcal{R}ep$. This follows by applying Proposition \ref{compatiblewithrig} with $\mathcal{G} = GL_n$ with trivial $G$-action.
\end{rem}
\begin{rem}
We have used the notation $\mathcal{R}ep^f$ here to denote the formal version of the representation stack, which differs from Wang--Erickson's notation. We do this so that we can call its generic fiber, which we use much more, $\mathcal{R}ep$.
\end{rem}

Now denote by $\mathcal{R}ep^f$ the full groupoid of representations without any residual condition. Since the moduli space of all pseudorepresentations is a disjoint union of spaces labeled by residual type, Wang-Erickson's result implies that $\mathcal{R}ep^f$ is similarly a disjoint union of formal stacks labeled by residual pseudorepresentation.

We are now ready to prove the main result of this section.
\begin{defn}
We say that a topological group $G$ has a liftable conjugation action on a linear algebraic group $A/E$ if $G$ acts continuously on $A$ and there is a faithful representation $A \hookrightarrow GL_m$ and a homomorphism $\phi \colon G \to GL_m(E)$ such that the action is given by conjugation by $\phi$. (Note that we use the convention $gxg^{-1}$ for conjugation.)
\end{defn}

\begin{thm}
Suppose that $G$ has a liftable conjugation action on a linear rigid analytic group $S/E$, where $E$ is a $p$-adic field and $G$ satisfies Mazur's $\Phi_p$ finiteness condition. Then the cohomology functor $H^1(G, S)$ is representable by a rigid analytic stack.
\end{thm}

\begin{rem}
In the case that we are interested in -- namely when $S$ is the geometric monodromy of a local system and Galois acts by inner automorphisms, the map to inner automorphisms lifts from a projective representation to a linear representation by a choice of basepoint. Indeed, the Galois action on the arithmetic fundamental group arises from conjugating an automorphism of the fiber of a local system by the Galois action on the fiber. Since the Galois action on the fiber of any local system is an honest representation, the Galois action on the geometric monodromy group is indeed a liftable conjugation action.

If $S$ is $GL_n/F$ then any conjugation action has a lifting $\phi \colon G_K \to {GL_n} \otimes_F {F(\zeta_N)}$ for some $N$, by a deep theorem of Patrikis (~\cite{patrikis2012variations}). However, we needn't use his theorem: as we have mentioned, in our situation $\tilde \phi$ will already come with a lift.
\end{rem}

\begin{proof} 
Consider the embedding of $S$ in $GL_N = GL_{n,E}$ that defines the conjugation action of $\phi$. Denote the map defining the conjugation action by $\phi \colon G \to GL_n(E)$. The simple but key construction is as follows: the assignment $$c \mapsto c \cdot \phi$$ induces a bijection $$Z^1(G,GL_{n,conj}) \xrightarrow{\sim} Hom(G, GL_n).$$ Here $GL_{n,conj}$ refers to $GL_n$ equipped with the conjugation action given by $\phi$. One can think of this map as untwisting the action $\phi$.

What of the quotient stack $H^1(G,GL_n)$ itself, then? The action of $GL_n$ on $Z^1(G,GL_n)$ in the case of inner action is just the change of basis on the $Hom(G,GL_n)$ side of the bijection, so cohomologous cocycles give equivalent representations and vice-versa. We now conclude by appealing to Wang-Erickson's theorem above: the moduli of representations $\mathcal{R}ep$ is representable by a disjoint union of rigid stacks. Thus so is the cohomology functor  $H^1(G,GL_n)$.

Now we transfer this knowledge to $S$ using the long exact sequence in Galois cohomology for non-normal subgroups. We denote by $S_{conj}$ the group $S$ with conjugation action by $\phi$, and consider the map $$h \colon H^1(G,S_{conj}) \to H^1(G,GL_{n,conj}).$$

Let $T$ be an affinoid algebra. By twisting, if $b_T \in H^1(G, GL_n)$ then the fiber of $h$ above $b$ is locally given by $$\left[({}_{b_T} GL_n/{}_{a_T} S)^G/({}_{b_T} GL_n)^G\right],$$ where $a_T \in H^1(G,S_T)$ is a local lift. Now $({}_{b_T} GL_n)^G$ depends only on $b_T$ and not $a_T$; in other words, there is a family over $H^1(G,GL_n)$ which we can denote $({}_{b} GL_n)^G$ that assigns to any morphism $b_{T'} \to H^1(G,GL_n)$ the pullback $({}_{b_{T'}} GL_n)^G$.

This pullback is representable by a scheme over $T$, since the fixed points are the intersection of fixed points over all $g \in G$, each of which is a closed subscheme of $GL_{n,T'}$. Thus $({}_{b} GL_n)^G$ is representable rigid varieties over the whole of $H^1(G,GL_n)$. Let $\Sigma_i$ be a flattening stratification for the family \[({}_{b} GL_n)^G \to H^1(G,GL_n).\] 

To show strat-representability, we show representability of the stack $H^1(G, S) \times_{H^1(G,GL_n)} \Sigma_i$. Let $a_T$ be a point of this stack, and again $b_T$ its image in $H^1(G, GL_n)$. We claim that stack $({}_{b_T} GL_{n,T}/{}_{a_T} S_T)^G$ is representable. Indeed, the quotient $GL_{n,T}/S_T$ is representable by a scheme over the affinoid $T$. By the same argument on fixed points, $({}_{b_T} GL_{n,T}/{}_{a_T} S_T)^G$ is representable by a closed subscheme of $GL_{n,T}/S_T$. 
 
Since the quotient $$\left[({}_{b_T} GL_n/{}_{a_T} S)^G/({}_{b_T} GL_n)^G\right]$$ is a quotient of a scheme over $T$ by a flat group scheme, its analytification is representable by a rigid stack over $T$ (Lemma \ref{flatquotient}.)

In particular, the map \[H^1(G, S) \times_{H^1(G,GL_n)} \Sigma_i \to \Sigma_i\] satisfies the conditions of Proposition \ref{transfer}, and we conclude that $H^1(G,S_{conj})$ is strat-representable by a rigid-analytic stack.

\end{proof}
\begin{rem}
Let us remark on another direction we could have taken above. Conjugation by $\phi$ preserves $S$ by hypothesis, and so $\phi$ takes values in the normalizer $N_{GL_n}(S)$. We can then play the same game as above, and make the following sketchy heuristic. When analyzing the map on first cohomology groupoids we encounter the fiber $N_{GL_n}(S)/S$; if we were working with schemes over an algebraically closed field, then we would have the classic statement that $N_{GL_n}(S)/S \cdot Z_{GL_n}(S)$ is finite. Intuitively, $H^1(G, S_{conj})$ differs from $H^1(G, N_{GL_n}(S))$ by at most $Z_{GL_n}(S)$ plus a finite defect. 
\end{rem}

\begin{exmp}
Finally, let us take a moment to understand what this ``unwinding'' is doing in the geometric situation of \'etale path torsors. So we now let $S$ be the geometric monodromy of our local system $\mathcal{L}^{\acute{e}t}$ on the scheme $X/F$. If $x,b \in X(F)$ then we have the right torsor $Isom(\mathcal{L}_b, \mathcal{L}_x)$ for $GL(\mathcal{L}_b)$ on which $G_F$ acts by conjugation by the Galois representations on the fibers above $x$ and $b$, denoted $\phi_x$ and $\phi_b$. We get a cocycle by choosing a path $\gamma \in P^{\acute{e}t}_{b,x}$ (this can always be done \ref{nonempty}), as follows. Since $S$ is a quotient of $\mathcal{G}$, we have a map $P^{\acute{e}t}_{b,x} \to Isom(\mathcal{L}_b, \mathcal{L}_x)$. Using this homomorphism to act on the target, one defines a cocycle $c$ that for any element $g \in G_F$ produces the element $$c(g) = \gamma^{-1} \phi_x (g) \gamma \phi_b (g)^{-1} .$$
One easily checks that $c$ is indeed a cocycle with the conjugation action: 
\begin{align} c(gh) = \gamma^{-1}\phi_x(g)\phi_x(h)\gamma \phi_b(h)^{-1}\phi_b(g)^{-1} \\ 
= \gamma^{-1}\phi_x(g) \gamma \phi_b(g)^{-1} \phi_b(g) \gamma^{-1} \phi_x(h)\gamma \phi_b(h)^{-1}\phi_b(g)^{-1} \\
= c(g)\phi_b(g)c(h) \phi_b(g)^{-1}
\end{align}

Thus $$c(g) \phi_b(g) = \gamma^{-1} \phi_x (g)\gamma;$$ i.e. it records the representation $\phi_x$ up to a change of basis.
\end{exmp}

\section{Selmer stacks of groups with conjugation action} \label{reductiveselmer}

In this subsection we show that the ``reductive level'' of the Bloch-Kato cohomology stack defined in Definition \ref{bk} is strat-representable by a rigid stack. To this end we will use a second theorem of Wang-Erickson (inspired by similar theorems of Kisin) that treats the functor parameterizing crystalline representations.

Our first result is that the ``untwisting'' procedure from this section preserves the property of being crystalline.
\begin{prop}
Let $\phi \colon G_K \to GL_n(\mathbb{Q}_p)$ be a crystalline representation, and $GL_{n,inn}$ the group with conjugation action given by $\phi$. Then the isomorphism of stacks $$H^1(G_K,GL_{n,inn}) \simeq H^1(G_K,GL_{n,triv}) \simeq \mathcal{R}ep_{n}(G_K),$$ given by untwisting, restricts to an isomorphism \[H^1_f(G_K, GL_{n,inn}) \simeq \mathcal{R}ep^{crys}(G_K) .\]
\end{prop}
\begin{proof}
We work locally. Suppose that a cohomology class $c \in H^1(G_K,GL_{n,inn}(T))$ is crystalline, where $T = \Spm \Lambda$. Then there exists $b \in GL_n(T \otimes B_{crys})$ such that $$c(g) = b^{-1}\phi(g)({}^{g}b)\phi(g)^{-1}.$$ (Recall that the definition of the Galois action on $\mathcal{G}(R \widehat{\otimes} B_{crys})$ includes the Galois action on $B_{crys}$.) The untwisting sends this cocycle to $$c(g)\phi(g) = b^{-1}\phi(g)({}^{g}b),$$ so that $c\phi$ has a $\Lambda \widehat{\otimes} B_{crys}$-linear and Galois-equivariant change of basis that makes it isomorphic to $\phi$. Denote the underlying family of representations of $c\phi$ by $W$, and that of $\phi$ by $V$. Restating the above, we have a Galois-equivariant isomorphism $$V \otimes_{\Lambda} (\Lambda \widehat{\otimes} B_{crys}) \simeq W \otimes_{\Lambda} (\Lambda \widehat{\otimes} B_{crys}),$$ so that their rings of $G_K$-invariants are equal. But $rank( V) = rank( W)$, so we conclude from the crystallinity of $V$ that $W$ is crystalline.
\end{proof}

Thus knowing about the Bloch-Kato sets discussed above is almost equivalent to knowing about representation stacks. Now to show that the stack of crystalline representations is representable by a formal stack we use a theory of Wang-Erickson's that cuts out a crystalline locus.

We must first define the concept of a Galois type, which is more complicated than what we will need: see \cite[Section 2.7]{kisin2008potentially} for a complete description. In short, to a potentially semi-stable representation $V_B$ one associates a representation $T \colon I_K \to W_{B'}$, where $B'$ is a finite \'etale $B$-algebra, $W_{B'}$ is a free $B'$-algebra, and $I_K$ is the inertia subgroup of $G_K$. Now if $$\tau \colon I_K \to GL_n(\bar{\mathbb{Q}}_p)$$ is any representation with open kernel, we say that $V_B$ is of Galois type $\tau$ if for all $\gamma \in I_K$ we have $$tr(T(\gamma)) = tr(\tau(\gamma)).$$ Galois type measures how potentially semistable or potentially crystalline a representation is. Since any representation that we will care about will be crystalline over the base field (and not simply potentially so), the reader is welcome to mentally set $\tau = 0$ and drop it from our notation.

Here is the theorem that we will need.
\begin{thm}{\cite[Theorem B]{wang2017algebraic}}
Let $\tau$ and $\bold{v}$ be fixed Galois and Hodge types. There exists a closed formal substack $\mathcal{R}ep_{\bar D}^{\tau, \bold{v}} \hookrightarrow \mathcal{R}ep^f_{\bar D}$ such that for any finite $\mathbb{Q}_p$-algebra $B$ and
point $\zeta \colon \Spec B \to \mathcal{R}ep^f_{\bar D}$ , $\zeta$ factors through $\mathcal{R}ep_{\bar D}^{\tau, \bold{v}}$ if and only if the corresponding representation $V_B$ of $G_K$ is potentially semi-stable of Galois and Hodge type $(\tau, \bold{v}).$ Moreover,
\begin{description}
\item[(1)] $\mathcal{R}ep_{\bar D}^{\tau, \bold{v}}[1/p]$ is reduced, locally complete intersection, equi-dimensional, and generically formally smooth over $\mathbb{Q}_p$. If we replace ``semi-stable'' with ``crystalline,'' it is everywhere formally smooth over $\mathbb{Q}_p$.
\item[(2)] If $\bar D$ is multiplicity-free, then $\mathcal{R}ep_{\bar D}^{\tau, \bold{v}}$ is algebraizable, i.e. $\mathcal{R}ep_{\bar D}^{\tau, \bold{v}}$ is the
completion of a closed substack $Rep_{\bar D}^{\tau, \bold{v}}$ of $Rep_{\bar D}$ . The geometric properties
of (1) also apply to $Rep_{\bar D}^{\tau, \bold{v}}[1/p]$, except equi-dimensionality, which applies
to its framed version $Rep_{\bar D}^{\square, \tau, \bold{v}}[1/p]$
\end{description}
\end{thm}

\begin{rem}
Let $\mathcal{R}ep^{crys, \bold{v}}(G_K)^f$ denote the substack referred to in \textbf{(1)} above. In other words, it is the substack of representations which are crystalline of a fixed Hodge type at all residue fields. Its rigid generic fiber is the stack of crystalline representations of a fixed Hodge type of \ref{crystalfamdefn}:

This follows because the stack of representations which are crystalline and of a fixed Hodge type (in the sense of Wang--Erickson) at all closed points is reduced, by the above; the theorem of Berger--Colmez (\ref{pointwise}) says that over reduced affinoid algebras, a family is crystalline at every point if and only if it is crystalline as a family.
\end{rem}

Let us now return to local cohomology stacks, where transferring representability from (some substacks of) $H^1_f(G_K,GL_{n,inn})$ to $H^1_f(G_K, S)$ will look very similar to the last section. Taking the preimage of $\mathcal{R}ep_{\bar D}^{\tau, \bold{v}}$ under the equivalence $$H^1_f(G_K,GL_{n,inn}) \simeq \mathcal{R}ep^{crys}(G_K),$$ we obtain stacks which we denote $H_{\bar D}^{1,\tau, \bold{v}}(G_K,GL_{n,inn}).$ Let $H_{\bar D}^{1,\tau, \bold{v}}(G_K,S)$ denote the 2-fibered product of the two maps $$H_{\bar D}^{1,\tau, \bold{v}}(G_K,GL_{n,inn}) \to H^1_f(G_K,GL_{n,inn})$$ and $$H^1_f(G_K, S) \to H^1_f(G_K,GL_{n,inn}).$$

We will need one more piece of $p$-adic Hodge theory, which Kim proves over a field \cite[discussion before Proposition 2]{kim2005motivic}. Let $K$ be a $p$-adic field, and $K_0$ its maximal unramified subfield.
\begin{lem}
Let $T$ be an affinoid $\mathbb{Q}_p$-algebra and $A$ an algebraic group over $T$ on which $G_K$ acts continuously such that $\mathcal{O}(A)$ is an ind-crystalline family of Galois representations over $T$. Then for any affinoid $T$-algebra $\Lambda$, \[H^0(G_K, A(\Lambda \widehat{\otimes} B_{crys})) = Res_T^{T \otimes K_0}(D_{crys}(A))(\Lambda).\] 
\end{lem}
\begin{proof}
By definition of the Galois action on $A(\Lambda \widehat{\otimes} B_{crys})$, we have 
\begin{align*}
H^0(G_K, A(\Lambda \widehat{\otimes} B_{crys})) = Hom_{G_K}(\mathcal{O}(A), \Lambda \widehat{\otimes} B_{crys}) \\
= Hom_{G_K, B_{crys} \widehat{\otimes} T}(\mathcal{O}(A) \otimes_T (B_{crys} \widehat{\otimes} T ), \Lambda \widehat{\otimes} B_{crys}) \\
\overset{crystallinity} = Hom_{G_K, B_{crys} \widehat{\otimes} T}(D_{crys}(\mathcal{O}(A)) \otimes_T B_{crys} \widehat{\otimes} T, \Lambda \widehat{\otimes} B_{crys}) \\
= Hom_{G_K}(D_{crys}(\mathcal{O}(A)), \Lambda \widehat{\otimes}B_{crys}) \\
= Hom(D_{crys}(\mathcal{O}(A)), (\Lambda \widehat \otimes B_{crys})^{G_K}) \\
= Hom(D_{crys}(\mathcal{O}(A)), \Lambda \otimes K_0) \\
= Res_T^{T \otimes K_0}(D_{crys}(A))(\Lambda)
\end{align*}
Here $Hom$ denotes algebra homomorphisms, the $G_K$ subscript denotes $G_K$-equivariance, and $B_{crys} \widehat{\otimes} T$ subscript denotes linearity in that ring.
\end{proof}
Now by the monoidal property of $D_{crys}$ and functoriality, $D_{crys}(A)$ is a group scheme over $T$. We conclude that the functor $H^0(G_K, A(- \otimes B_{crys}))$ on $T$-algebras is representable by a scheme over $T$.

Finally we can prove:
\begin{prop} \label{Srep}
The restricted Selmer stack $H_{\bar D}^{1,\tau, \bold{v}}(G_K,S)$ is strat-representable in the category of rigid stacks.
\end{prop}
\begin{proof}
Let us reason as in the arguments of Theorem \ref{representability}.

The fibers of the map $$H_{f, \bar D}^{1,\tau, \bold{v}}(G_K,S) \to H_{f, \bar D}^{1,\tau, \bold{v}}(G_K,GL_{n,inn})$$ are the same as the fibers of $$H^1_f(G_K, S) \to H^1_f(G_K,GL_{n,inn}).$$ We take the long exact sequence of a non-normal subgroup:
\[
\begin{tikzcd}
H^0(G_K, S) \arrow{r} & H^0(G_K, GL_{n,inn}) \arrow{r} & H^0(G_K, GL_{n,inn}/S) \\
\arrow{r} & H^1(G_K, S) \arrow{r} & H^1(G_K, GL_{n,inn})
\end{tikzcd}
\]
which maps to
\[
\begin{tikzcd}[column sep=small]
H^0(G_K, S(- \otimes B_{crys})) \arrow{r} & H^0(G_K, GL_{n,inn}(- \otimes B_{crys})) \arrow{r}  & H^0(G_K, GL_{n,inn}/S (- \otimes B_{crys})) \\  
\arrow{r} & H^1(G_K, S( - \otimes B_{crys})) \arrow{r} & H^1(G_K, GL_{n,inn}(- \otimes B_{crys}))
\end{tikzcd}
\]
For $T$ an affinoid algebra, let $b \in H^1_f(G_K,GL_{n,inn})(T)$. Locally on $T$, take a lift $a \in H^1_f(G_K,S)(T)$. Then the exact sequences tell us that fiber above $b$ is locally $(f_1)^{-1}im(f_2)$, where $f_1$ and $f_2$ are
\[
\begin{tikzcd}
 & \left[H^0({}_b GL_n/{}_a S)/ H^0({}_b GL_{n,inn}) \right] \arrow{d}{f_1} \\
H^0({}_b GL_{n,inn}(- \otimes B_{crys})) \arrow{r}{f_2}  & H^0((GL_{n,inn}/S)(- \otimes B_{crys}))
\end{tikzcd}
\]
Again by the short exact sequence, \[im(f_2) = [H^0({}_b GL_{n,inn}(- \otimes B_{crys}))/H^0({}_a S(- \otimes B_{crys}))].\] Let us show that this stack is representable.

Since $b \phi$ is crystalline ($b \in H^1_f$), so is $\mathcal{O}(GL_{n,inn,T})$; this is because the latter ring is the colimit $\bigoplus_{\alpha} End(V_{\alpha}) \otimes T$, where $End(V_{\alpha}) \otimes T$ has the tensor product action of $b \phi$ and $V_{\alpha}$ runs over all irreducible representations of $GL_n$.

By the lemma immediately preceding this proposition, we see that $H^0(GL_{n,inn,T}(- \otimes B_{crys}))$ is representable by the scheme $Res_{T}^{T \otimes K_0}D_{crys}(GL_{n,T})$. Now $\mathcal{O}(S_T)$ is a quotient of $\mathcal{O}(GL_{n,inn,T})$ as a Galois representation. Closure of (ind-)crystallinity under quotients (over affinoid algebras) shows that $H^0({}_a S(- \otimes B_{crys}))$ is also representable by a scheme over $T$.

We may analyze the lower right pseudofunctor similarly. The homogeneous space $GL_{n,T}/S_T$ is affine because $S$ is reductive (\cite[Theorem 12.15]{alper2013good}).
Since the morphism $$GL_{n,T} \to (GL_n/S)_T$$ surjective, we have an injection of coordinate rings $\mathcal{O}(GL_n/S) \otimes T \to \mathcal{O}(GL_n) \otimes T$. The codomain is crystalline, and so closure of crystalline representations under subobjects over reduced affinoid algebras shows that $\mathcal{O}(GL_n/S)$ is crystalline, and thus $H^0((GL_n/S)(-- \otimes B_{crys}))$ is also representable by $D_{crys}(GL_{n,inn,T}/S_T)$. Since the whole span consists of rigid stacks, we see that the fibered product
\[
\begin{tikzcd}
 & \left[H^0({}_b GL_n/{}_a S)/ H^0({}_b GL_{n,inn}) \right] \arrow{d}{f_1} \\
\left[ H^0({}_b GL_{n,inn}(- \otimes B_{crys}))/ H^0({}_a S(- \otimes B_{crys})) \right] \arrow[hookrightarrow]{r}{f_2}  & H^0((GL_{n,inn}/S)(- \otimes B_{crys}))
\end{tikzcd}
\]
exists as a rigid stack, which is what we wanted.

The arguments of Theorem \ref{representability} go through unchanged for the piece of the span \[\left[H^0({}_b GL_n/{}_a S)/ H^0({}_b GL_{n,inn}) \right].\] Namely, there is a stratification of $H^1_f(GL_{n,inn})$ over which $H^0({}_b GL_{n,inn})$ is flat and of finite presentation, and similarly for $Res^{T \otimes K_0}_{T}D_{crys}(S)$. Taking the fiber product of the span above, we see that $H^1_f(S)$ is indeed strat-representable.
\end{proof}

The definition of the global Selmer stack as a fibered product makes it clear that the global substacks mapping to the local Selmer stacks of fixed Galois and Hodge types are strat-representable if the local ones are.

\chapter{de Rham and Crystalline Realizations; $p$-adic Period Mappings}

\section{Hodge theory}
We will now discuss Hain's mixed Hodge structure on the relative completion of the fundamental group of a complex algebraic variety with respect to a polarized variation of Hodge structures. For the remainder of this section we will follow Hain's notation. To that end, let $X$ be a smooth complex algebraic variety with basepoint $x_0$. Let $\mathbb{V}$ be a polarized variation of Hodge structure on $X$, with $V_0$ the fiber of $\mathbb{V}$ above $x_0$, and write $S$ for the Zariski closure of the monodromy of $\mathbb{V}$. The monodromy is a subgroup of the symplectic or orthogonal group on $V_0$ arising from the polarization on $\mathbb{V}$.  We write $\rho \colon \pi_1(X,x_0) \to S$ for the necessarily Zariski-dense monodromy representation.
Finally, denote the relative completion of the Betti fundamental group $\pi_1(X,x_0)$ with respect to $\rho$ by $\mathcal{G}(X,x_0)$.

There is a right flat principal $S$ bundle $P$ over $X$ which is associated in the usual way to the action of right multiplication of $\pi_1(X,x_0)$ on $S$ via $\rho$. More precisely, if $\tilde X$ is the universal cover of $X$ then $P = (S \times \tilde X)/\pi_1(X,x_0)$, where $\gamma \in \pi_1(X,x_0)$ sends $(s,p)$ to $(s\rho(\gamma)^{-1}, \gamma p)$. We will also need the fiber bundle associated to this principal bundle through the right action of $S$ on $\mathcal{O}(S)$: $f(s_1) \mapsto f(s_1 \gamma)$. This bundle is then defined by the similar formula $$\mathcal{O}(P) \defeq (\mathcal{O}(S) \times P)/S.$$ It has fibers which are non-canonically isomorphic to $\mathcal{O}(S)$. In fact, the first key fact is that one has a canonical Hodge structure of weight zero on the algebra $\mathcal{O}(S)$, which induces a Hodge structure of weight zero on the variation $\mathcal{O}(P)$. We explain the proof.
\begin{prop}{\cite[Corollary 13.7]{hain1996hodge}}
Suppose that $S$ is the whole of the group $\Aut(V_0, \langle \cdot, \cdot \rangle)$, i.e. the entire symplectic or orthogonal group. Then there is a pure Hodge structure of weight zero on $\mathcal{O}(S)$, and it induces a Hodge structure of weight zero on the variation $\mathcal{O}(P)$.
\end{prop}
\begin{proof}
The Hodge structure on $\mathcal{O}(S)$ falls out from a structure theorem for $\mathcal{O}(S)$ as follows. Let $V_\alpha$ be a set of representatives for the isomorphism classes of all irreducible representations of $S$.
The structure theorem then says that $$\mathcal{O}(S) = \sum_{\alpha} V_{\alpha} \boxtimes V_{\alpha}^*$$ as an $S,S$ bimodule. One proves that each irreducible representation comes from a variation of Hodge structure on $X$, and thus carries a Hodge structure. (Indeed, a construction due to Weyl that is used to construct all irreducible representations of $Sp(V_0)$ from the fundamental representation actually preserves the property of coming from a local system -- see Hain's proposition for more details.) Now tensoring each representation and its dual gives a Hodge structure of weight zero on their tensor product, and their sum thus has a Hodge structure.

Now \[\mathcal{O}(P) = \sum_{\alpha} \mathbb{V}_{\alpha} \otimes V_{\alpha}^*. \tag{$\ast$}\] From this decomposition we obtain the Hodge structure on the variation $\mathcal{O}(P)$, since each irreducible representation comes from a variation of Hodge structure $\mathbb{V}_\alpha$.
\end{proof}

Before we place a mixed Hodge structure on the coordinate ring of the entire relative completion, we show how the previous argument generalizes to ``neat'' $S$, which we define before stating the theorem.
\begin{defn}
Let $\mathbb{V}$ be a variation of Hodge structure over $X$, a complement of a normal crossings divisor in a compact Kahler manifold. We say that $\mathbb{V}$ is \textit{neat} if it has semi-simple monodromy and if the Hodge structure on the coordinate ring of $\Aut(V_0, \langle \cdot, \cdot \rangle)$ induces one on $\mathcal{O}(S)$.
\end{defn}

We can answer a question raised by Hain \cite[top of p.85]{hain1996hodge}.

\begin{prop}
In fact, any polarizable variation of Hodge structure is neat.
\end{prop}
\begin{proof}
Because of the description of the coordinate ring of $S$ in terms of irreducible representations of $S$, it suffices to show that any representation of $S$ has an induced Hodge structure. Now, indeed, any such representation is a subquotient of a tensor power of V and its dual. Furthermore, since $S$ is reductive, it suffices to consider only subbundles of these tensor powers. But Deligne's semisimplicity theorem \cite[Corollaire 4.2.9]{deligne1971theorie} says that every subbundle of a polarizable VHS has an induced VHS.
\end{proof}

Note that our notations differ slightly from those of Hain here: for him $S$ always means the symplectic/orthogonal group above, while $S$ for us is the (potentially smaller) algebraic monodromy group. We may now proceed to the main theorem.
\begin{thm}{\cite[Thm. 13.1]{hain1996hodge}}
With the notation of the previous definition, assume that $\mathbb{V}$ is neat. Then the coordinate ring $\mathcal{O}(\mathcal{G}(X,x_0))$ has a canonical real mixed Hodge structure with weights $\geq 0$ for which the product, coproduct, antipode and the natural inclusion $$\mathcal{O}(S) \hookrightarrow \mathcal{O}(\mathcal{G}(X,x_0))$$ are all morphisms of mixed Hodge structure. Moreover the canonical homomorphism $\mathcal{G}(X,x_0) \to S$ induces an isomorphism $Gr_0^W\left( \mathcal{O}(\mathcal{G}(X,x_0)) \right) \simeq \mathcal{O}(S)$.
\end{thm}

The proof of such theorems all follow the path set out by Chen \cite{chen1977iterated}: mixed Hodge structures arise from comparison theorems with bar complexes of differentials. We will need this explicit description of $\mathcal{O}(G^{dR})$.
\begin{defn}
Call a local system $\mathbb{M}$ on $X$ \it{rational} if it arises from a rational representation of $S \to GL(V)$ via the composition $\pi_1(X,x_0) \to S \to GL(V)$. (Here rational representation is used the sense of algebraic groups.)
\end{defn}
The equation $(\ast)$ tells us that the representation $\mathcal{O}(S)$ is the direct limit of its finite-dimensional representations, and so we can reasonably define differential forms with values in in $\mathcal{O}(P)$ using forms with values in its finite-dimensional bundles.
\begin{defn}
Let $E^\bullet(X, \mathbb{M})$ denote the smooth de Rham complex of $C^\infty$ forms on $X$ with coefficients in a finite-dimensional flat vector bundle $\mathbb{M}$ (we wrote this $\Omega^\bullet(X,\mathbb{M})$ above.) Then for $\mathbb{V}$, a possible infinite-dimensional flat vector bundle, we define a new complex $$E^\bullet_{fin}(X,\mathbb{V}) = \varinjlim E^\bullet(X, \mathbb{M})$$ where $\mathbb{M}$ runs over the set of rational sub-local systems of $\mathbb{V}$.
\end{defn}

The reader should refer to \cite[Section 7]{hain1996hodge} for a precise description of the relations and differentials in the bar construction, since we do not use it in depth here. Suffice it to say that if $A^\bullet$ is a differential graded algebra with non-negative grading and $N^\bullet,M^\bullet$ are modules over $A^\bullet$, then the reduced bar construction $$B(M^\bullet, A^\bullet, N^\bullet)$$ is a differential graded algebra which is a quotient of $$T(M^\bullet, A^\bullet, N^\bullet) = \bigoplus_{r\geq 0} M^\bullet \otimes  (A^{\bullet > 0}[1])^{\otimes r} \otimes N^\bullet.$$ The element \[m \otimes a_1 \otimes \cdots \otimes a_l \otimes m \in B(M^\bullet, A^\bullet, N^\bullet)\] is traditionally denoted $m[a_1 |\cdots a_l]n$, or in \cite{hain1996hodge} by $(m | a_1 \cdots a_l | n)$.

Now we are concerned with two augmentations: $$\epsilon \colon E^\bullet_{fin}(X,\mathcal{O}(P)) \to \mathbb{R}$$ is defined locally by evaluating a differential form $\sum \omega_i \otimes \phi_i$ ($\omega_i \in E^k(X), \phi_i \in \mathcal{O}(S)$) at the basepoint $x_0$ and the identity element of $S$, while the augmentation $$\delta \colon E^\bullet_{fin}(X,\mathcal{O}(P)) \to \mathcal{O}(S)$$ only evaluates the differential form at the basepoint. Using these augmentations we see $\mathbb{R}$ and $\mathcal{O}(S)$ as $E^\bullet_{fin}(X,\mathcal{O}(P))$-algebras.

\begin{defn}
We define
$$\mathcal{G}^{dR} = \Spec H^0\left (B \left(\mathbb{R}, E^\bullet_{fin}(X,\mathcal{O}(P)), \mathcal{O}(S) \right) \right)$$
\begin{center}
and
\end{center}
$$\mathcal{U}^{dR} = \Spec H^0\left (B \left(\mathbb{R}, E^\bullet_{fin}(X,\mathcal{O}(P)), \mathbb{R}\right)\right).$$
\end{defn}
In fact, $\mathcal{U}^{dR}$ is pro-unipotent \cite[Prop 8.3]{hain1996hodge}, which allows one to compare it to the relative completion using the universal property of the latter.

An essential ingredient to all that follows is that $E^\bullet_{fin}(X,\mathcal{O}(P))$ is a mixed Hodge complex in the sense of Saito, and so the coordinate rings of $\mathcal{U}^{dR}$ and $\mathcal{G}^{dR}$ inherit mixed Hodge structures \cite[3.2.1]{hain1987rham}.

We are ready to give the comparison map $$\Phi \colon \pi_1(X,x_0) \to \mathcal{G}^{dR}(\mathbb{R}).$$ It arises from the following relative iterated integral map.
\begin{defn}
For $\phi \in \mathcal{O}(S)$ and $w_1, \ldots w_r \in E^1_{fin}(X,\mathcal{O}(P))$, we define $$\int(w_1 w_2 \cdots w_r | \phi) \colon P_{x_0, x_0}(X) \to \mathbb{R}$$ by 
$$\int_\gamma(w_1 w_2 \cdots w_r | \phi) \defeq \phi(\rho(\gamma)) \int_{\tilde \gamma} w_1 w_2 \cdots w_r.$$
Here $\tilde \gamma$ is the unique lift of $\gamma$ which begins at the canonical basepoint $\tilde x_0 \in P$ above $x_0$. This integral induces the map $$\Phi \colon \pi_1(X,x_0) \to \Hom_{\mathbb{R}-alg}\left(H^0\left (B \left(\mathbb{R}, E^\bullet_{fin}(X,\mathcal{O}(P)), \mathcal{O}(S) \right) \right), \mathbb{R}\right) = \mathcal{G}^{dR}(\mathbb{R})$$ (because the $H^0$ picks out homotopy invariant iterated integrals).
\end{defn}
An entirely analogous statement holds for the entire path groupoid \cite[Section 12]{hain1996hodge}.

The above iterated integral refers to Chen's usual iterated integrals, with the only difference being that one performs the iterated integral on the bundle $P$ via the lifted path $\tilde \gamma$. Indeed, Hain shows that pullback of forms from $E^1(X, \mathcal{O}(P))$ to $E^1(P, \mathcal{O}(P))$ is an injection (actually, the pulled back forms take values in a foliation related to $P$.) Then, as in the unipotent case, one pulls back the differential $\omega \in E^1(P, \mathcal{O}(P))$ to $[0,1]$ via the path $\tilde \gamma$. There $\omega$ is a differential form with values in $\mathbb{R}$, and it can be integrated in the usual sense to a real number.
\begin{rem}

One is tempted to think of these integrals as detecting the (Malcev unipotent completion of the) fundamental group of the cover defined by $\ker \rho$. The reductive information is captured, on the other hand, by the $\phi(\rho(\gamma))$ term. However, this intuition is usually incorrect. For example, the case of the moduli space of elliptic curves is exceedingly interesting and has been studied by Hain, Brown, and others. In that case $\rho$ is $SL_2(\mathbb{Z}) \to  SL_2(\mathbb{Q})$, so that the kernel is trivial and the associated cover is the universal cover of $M_{1,1}$. Despite the above intuition indicating that the relative completion should be trivial, then, it turns out that the relative completion detects modular forms! A motivically-tailored reference for this story is \cite{brown2014multiple}.

We also mention that these iterated integrals satisfy familiar-looking shuffle and antipode relations, for which the reader may consult Hain's work.
\end{rem}

Let $\mathcal{G}$ denote the relative completion of $\pi_1(X,x_0)$ with respect to $\rho$. The universal property of the relative completion (using \cite[Prop. 8.3]{hain1996hodge}) ensures that there is an induced map $\tilde \Phi \colon \mathcal{G} \to \mathcal{G}^{dR}$. The main theorem, now, is the comparison theorem which says that relative iterated integrals can detect the entire relative completion of the fundamental group.
\begin{thm}
The map $\tilde \Phi$ is an isomorphism of pro-algebraic groups which commutes with the projections to $S$.
\end{thm}
Let us give a brief sketch of the proof. In addition to the map $\Phi$, one constructs a map $\mathcal{G}^{dR} \to \mathcal{G}$. The idea is that the monodromy of any vector bundle in $\mathcal{T}^{rel}$ is given by the iterated integrals above \textit{including} the monodromy of the universal bundle $\mathcal{O}\mathcal{(G})$. In other words, the map $\pi_1(X) \to \mathcal{G}(\mathbb{Q})$ has a factorization $$\pi_1(X) \to \mathcal{G}^{dR}(\mathbb{Q}) \to \mathcal{G}(\mathbb{Q}).$$ Now a crucial algebraic step is that the two maps between $\mathcal{G}$ and $\mathcal{G}^{dR}$ end up being inverses if the induced map $$H_1(\mathcal{U}) \to H_1(\mathcal{U}^{dR})$$ is an isomorphism. Last but not least, this isomorphism reduces to one core geometric fact: for any irreducible $S$-representation $V$, the two groups $$(H_1(\mathcal{U}) \otimes V)^S$$ and $$(H_1(\mathcal{U}^{dR}) \otimes V)^S$$ are both isomorphic to $H^1(X, \mathbb{V})$, where $\mathbb{V}$ is the vector bundle associated to $V$. This fact follows from the same Lyndon--Hochshild--Serre arguments from Section 3.1.

The proof affirms the philosophy: relative completions of fundamental groups are ways of recording the cohomologies of certain vector bundles, local systems, and isocrystals in an $S$-equivariant way. Thus any comparison isomorphism between them will ultimately be a comparison between these cohomologies.

\section{Frobenius weights}
Having discussed the Hodge structure on the relatively unipotent de Rham fundamental group, we pivot to its Frobenius structure. The work of Olsson in \cite{olsson2011towards} places such a Frobenius structure on the relative completion. Let us review a few concepts from this crystalline theory; recall that $\kappa$ is a finite field, and fix an embedding $K_0 \hookrightarrow \mathbb{C}$.
\begin{defn}
A $\phi$-module $M$ over $\kappa$ is a  $K_0$-vector space with an additive map $\phi \colon M \to M$ that is semi-linear with respect to the Frobenius action on $K_0$: $\phi(hm) = \phi(h)\phi(m)$. (We use the same notation for both Frobenii.)

A $\phi-module$ equipped with the datum of a decreasing and exhaustive filtration after extension to some totally ramified extension $K/K_0$ is called a filtered $\phi$-module over $K$.

A $\phi$-module is called an isocrystal if $\phi$ is a bijection, and a filtered isocrystal is an isocrystal that is a filtered $\phi$-module.

The preceding categories will be denoted $\phi-Mod_{\kappa}$, $MF_K^{\phi}$, $Isoc_{\kappa}$, and $FilIsoc_{\kappa}^K$. 
\end{defn}

\begin{defn}
Let $M \in Isoc_{\kappa}$, and suppose that $\phi^a$ is a power of the endomorphism of $M$ that is linear. (It suffices to take $a = [\kappa:\mathbb{F}_p]$.) Then we say that $M$ has weights $\leq w$ if every eigenvalue of $F^a$ has (after embedding it into $\mathbb{C}$) complex norm $p^{aw'/2}$ for $w' \leq w$. We say that a pro-isocrystal has negative weights if its finite-dimensional quotients do.
\end{defn}

In fact, starting with a smooth scheme $\tilde X$ over $\kappa$, one has an entire category of $F$-isocrystals over $\tilde X$ with values in $K_0$. We avoid a full discussion of isocrystals here; cf. \cite{berthelot2006cohomologie}. Since $\kappa$ is finite, the category of $F$-isocrystals is indeed a neutral Tannakian category once we fix a basepoint $b \in \tilde X_{\kappa}(\kappa)$. We denote its fundamental group by $\pi_1^{crys}(X_{\kappa}/K_0,b)$. Any one of the previous variations on this group are fair game now: for example, the unipotent crystalline fundamental group and the relatively unipotent fundamental group with respect to an $F$-isocrystal on $\tilde X$. Then, by the theory of geometry in a Tannakian category, the associated affine ring $\mathcal{O}(\pi_1^{crys}(\tilde X/K,b))$ is a ``ring in isocrystals;'' since isocrystals come with a Frobenius endomorphism, the affine ring and thus the various crystalline fundamental groups inherit a Frobenius endomorphism.

Now suppose that $\tilde X$ has a smooth lift $X/\mathcal{O}_K$ for $K/K_0$. We fix a vector bundle $\mathcal{L}^{dR}$ with integrable connection on $X_{K_0}$, and denote the associated $F$-isocrystal on $\tilde X$ by $\mathcal{L}^{crys}$. Finally, we write $S^{crys}$ for the Tannakian fundamental group of $\langle \mathcal{L}^{crys} \rangle_{\otimes}.$ Then we have the following comparison isomorphism, which is implicit in Olsson's work \cite[4.13]{olsson2011towards} and which we recall in Section \ref{drcrys}.
\begin{prop}

There is a canonical isomorphism $$\pi_1^{rel,crys}(\tilde X/K_0) \otimes_{K_0} K \simeq \pi^{rel,dR}_{1}(X/ K),$$ which induces a Frobenius action on the relative de Rham fundamental group.
\end{prop}

Olsson and Pridham describe the action of Frobenius on $\pi_1^{rel,crys}(\tilde X)$ by relating this group to cohomology of isocrystals in $\mathcal{T}^{rel, crys}$. For example, let $\mathcal{U}^{crys}$ be the unipotent radical of this relative crystalline fundamental group and $\mathfrak{u}^{crys}$ its Lie algebra, which is a pro-Lie algebra in isocrystals. Then we have the following statement about Frobenius weights (\cite{pridham2012l}, which is \cite[Theorem 4.2]{olsson2007f} for the mixed case.)

\begin{prop}
The Lie algebra $\mathfrak{u}^{crys}$ has Frobenius weights at most -1.
\end{prop}

\begin{proof}
Let $\tilde X \hookrightarrow \bar X$ denote a smooth compactification.

The cohomology group $\mathfrak{h} \defeq H^1_{crys}\left(\bar X,j_*\mathcal{O}(S^{crys})\right)^* \oplus H^0(\bar X, R^1j_*\mathcal{O}(S^{crys}))^*$ has Frobenius weights $ \leq -1$ (See, for example, \cite{abe2018theory}.) There is a Frobenius-equivariant identification $\mathfrak{u}^{ab} \simeq \mathfrak{h}$ by \cite[proof of Theorem 3.8]{pridham2012l}, and thus a natural surjection $\mathfrak{u}^{crys} \to \mathfrak{h}$. This surjection has a splitting that preserves generalized Frobenius eigenspaces \cite[4.7]{olsson2007f}.

By universality, the splitting $\mathfrak{h} \overset{\sigma} \to \mathfrak{u}^{crys}$ extends in a Frobenius-equivariant way to the free Lie algebra $\mathfrak{f} \defeq \mathbb{L}\mathfrak{h}$. In particular, from the negative weight of $\mathfrak{h}$ we see that $\mathfrak{u}^{crys}$ has only negative Frobenius weights. 
\end{proof}

\begin{rem}
The reader should be reminded of Hain's statement that $\mathfrak{u}^{dR}_{\mathbb{C}}$ has negative weights in the sense of mixed Hodge theory. The proof there proceeds in essentially the same way, namely by comparing $\mathfrak{u}$ to the cohomology of $X$ with values in the bundle $\mathcal{O}(S^{dR})$.
\end{rem}

\subsection{Why do Frobenius weights matter?}

We begin in this subsection to generalize Besser's notion of Tannakian Coleman integration to the relatively unipotent case.
\begin{rem}
The concrete consequence of our work in this subsection is that the de Rham moduli space is much larger than it first appears, as we will see. But first, a remark on why (Tannakian) Coleman integration is important.

From the perspective of the author, giving Frobenius-invariant analytic continuations of the integrals appearing in the relative completion will make the first inroads to \textit{globalizing} the method of Lawrence and Venkatesh. Because Gauss--Manin only converges in residue discs, Lawrence and Venkatesh are forced to work residue disc by residue disc; their method is thus sensitive to the number of residue discs. But if the smallest prime of good reduction for $X$ is large (and they do their analysis with a prime of good reduction), then the Weil conjectures say that the number of residue tubes will also be large. If we hope to eliminate dependence on the size of the smallest prime of good reduction (or, more generally, if we would like to prove estimates of $X(K)$ that are uniform in some sense) then we require global functions which vanish on rational points. See \cite{katz2016uniform} for a great example of this philosophy. Even if one would like to bound just the number of rational points on a single residue disc, such bounds usually come from global constraints such as Riemann-Roch \cite{coleman1985effective}.
\end{rem}

Let us recall Besser's reformulation of Coleman integration with an eye towards adapting it to the relative setting. Recall that $\pi_1^{un, crys}$ denotes the unipotent crystalline fundamental group; for $x,y \in X(\kappa)$ and their associated fiber functors $\omega_x, \omega_y$ in the category of unipotent isocrystals on $X$, the torsor $P^{crys, un}_{x,y} = Isom(\omega_x, \omega_y)$ is a scheme in isocrystals and has a Frobenius-equivariant right action of $\pi_1^{crys,un}(X, y)$.
\begin{thm}{\cite[Corollary 3.2]{besser2002coleman}}

Let $X/\kappa$ be a smooth scheme, and $x,y \in X(\kappa)$. The unipotent crystalline path torsor $P^{crys,un}_{x,y}$ has a unique Frobenius-invariant trivialization $$p^{crys} \in P^{crys,un}_{x,y}(K_0).$$
\end{thm}

We have the following theorem in our case.
\begin{thm} \label{integration}
Let $X/\kappa$ be a smooth scheme, and $\mathcal{U}^{crys}$ the unipotent radical of the relatively unipotent fundamental group with respect to an isocrystal $\mathcal{L}^{crys}$. Suppose, as always, that the Tannakian fundamental group of $\mathcal{L}^{crys}$ is reductive. Then any $\phi$-equivariant torsor $P$ for $\mathcal{U}^{crys}$ has a unique Frobenius-invariant trivialization $$p^{crys} \in P(K_0).$$ In particular, $(\mathcal{U}^{crys})^{\phi=1}$ is trivial.
\end{thm}
\begin{proof}
Besser's proof works verbatim for the negatively-weight unipotent group $\mathcal{U}^{crys}$; it is entirely group-theoretical, given the knowledge of the weights of Frobenius on $\mathfrak{u}^{crys}$.

Indeed, consider the Lang map, which is a map from $\mathcal{U}^{crys}$ to itself defined on $R$-points by $$g \to g^{-1}\phi(g).$$ It is an easy bit of algebra to show that if the Lang map is an isomorphism then any torsor for $\mathcal{U}^{crys}$ has a unique Frobenius-invariant trivialization, and this algebraic factoid holds for any $\phi$-equivariant torsor for any pro-algebraic group \cite[Corollary 3.2]{besser2002coleman}.

The deep part is the bijectivity of the Lang map, but again this bijectivity follows from the negativity of the weights of Frobenius on $\mathfrak{u}^{crys}$ \cite[proof of Theorem 3.1]{besser2002coleman}.

\end{proof}

The problem with this formulation is that $\mathcal{U}^{crys}$-torsors do not appear in nature; rational points naturally give torsors for the whole of $\pi_1^{rel,crys}$. The most natural formulation of ``relative Coleman integration in the style of Besser'' will be presented in Proposition \ref{growth}, once we have developed a theory of admissible de Rham torsors.

\section{From crystalline to dR to Hodge} \label{drcrys}
One would like to use this sequence of theorems to put a Hodge filtration and Frobenius structure on the de Rham fundamental group unipotent relative to a given vector bundle. In the unipotent case, these structures arise because of the isomorphisms relating $H^1_{crys}(X/K_0)$, $H^1_{dR}(X/K)$, and $H^1_{dR}(X/\mathbb{C})$. In the relatively unipotent case, things get more complicated.

Let us first discuss comparisons for the reductive quotient $S$ in each realization. Once we have these comparisons we will show how the comparison for $\mathcal{U}$ follows. The story for $S$ has two fundamental ingredients: the transition from crystalline to $p$-adic de Rham and then the comparison from there to complex de Rham.

Indeed, suppose that $\mathcal{L}^{dR}$ is a vector bundle with integrable connection on $X_K$ and that $\mathcal{L}^{dR}_{\mathbb{C}}$ is a polarized variation of Hodge structures on $X_{\mathbb{C}}$ such that $$\mathcal{L}^{dR}_{\mathbb{C}} \simeq \mathcal{L}^{dR} \otimes_{K} {\mathbb{C}}.$$ Here we have fixed an embedding of $K$ into $\mathbb{C}$, and $\mathcal{L}^{dR}_K \otimes_{K} {\mathbb{C}}$ is the pullback of $\mathcal{L}^{dR}_K$ to $X_{\mathbb{C}}$. This assumption will hold when the vector bundles in question come from the Gauss--Manin connection of a smooth and proper morphism, by flat base change. We can always base-change a $\mathbb{Q}_p$-vector bundle to a $\mathbb{C}$-vector bundle, which furnishes a map $$S^{dR}_{\mathbb{C}} \to S^{dR}_{\mathbb{Q}_p}$$ and a similar map for $\mathcal{G}$.

The following result relates the $p$-adic de Rham and complex de Rham monodromy of our bundle.
\begin{prop}\cite[Proposition 1.3.2]{katz1987calculation}, (due to O. Gabber)
We have an isomorphism
$$S^{dR}_{\mathbb{C}} \simeq S^{dR}_K \otimes \mathbb{C}.$$
\end{prop}
An interested reader can see this result applied to independence-of-$\ell$ in \cite[Theorem 4.2]{crew1996differential}.

We then come to the crystalline vs. $p$-adic de Rham story for $S$.
\begin{prop}{\cite[Lemma 4.16, Section 4.12]{olsson2011towards}}

If $\mathcal{L}^{dR}$ and $\mathcal{L}^{crys}$ both have nilpotent local monodromy, then we have an isomorphism $S^{crys} \otimes_{K_0} K \simeq S^{dR}_K$.
\end{prop}

\subsection{Moving to $\mathcal{U}$}
Suppose we have an isomorphism on the reductive parts of crystalline and complex realizations of the relative completion as above, for example in the Gauss--Manin case. (Abstractly, we can just assume reductivity and nilpotent local monodromy.) Then this isomorphism extends to $\mathcal{U}$ and the whole of $\mathcal{G}$ as follows.

\begin{prop}
Preserve the notations of the last subsection. We have isomorphisms $\mathcal{G}^{crys} \otimes_{K_0} K \simeq \mathcal{G}^{dR}_K$ and $\mathcal{G}^{dR} \otimes_{K} \mathbb{C} \simeq \mathcal{G}^{dR}_{\mathbb{C}}$.
\end{prop}

\begin{proof}
Pridham's Adams spectral sequence \cite[Proposition 1.12]{pridham2012l} implies that $\mathcal{U}$ depends only on $H^p(X, \mathcal{O}(S))$ and $H^p(Y, R^q j_* \mathcal{O}(S))$ and spectral sequence and cup product maps between these groups in each realization. Here $j \colon X \hookrightarrow Y$ is an embedding into a proper scheme with normal crossing complement. But by flat base change and the crystalline-de Rham comparison theorem with coefficients \cite[4.13.1]{olsson2011towards}, we obtain isomorphisms 
$$H^p_{dR}(X_{K}, \mathcal{O}(S)) \otimes_{K} \mathbb{C} \simeq H^p_{dR}(X_{\mathbb{C}}, \mathcal{O}(S))$$
$$H^p_{crys}(X_{\kappa}, \mathcal{O}(S)) \otimes_{K_0} K \simeq H^p_{dR}(X_K, \mathcal{O}(S))$$
and similarly for $H^p(Y, R^q j_* \mathcal{O}(S))$.

Furthermore, these comparisons respect the Adams and Leray spectral sequences. We conclude that $\mathcal{U}^{crys} \otimes_{K_0} K \simeq \mathcal{U}^{dR}_K$ and $\mathcal{U}^{dR}_K \otimes_{K} \mathbb{C} \simeq \mathcal{U}^{dR}_\mathbb{C}$. Furthemore, since these comparisons are equivariant for the $S$-action on $\mathcal{O}(S)$, they extend to comparisons of $\mathcal{G}^{crys}$ and the complex and $p$-adic $\mathcal{G}^{dR}$.
\end{proof}

\section{de Rham torsor spaces}
We take this section to generalize Kim's study of the homogeneous space $$\mathcal{G}^{dR}/F^0\mathcal{G}^{dR}.$$ We also connect it with the usual theory of period maps and flag varieties due to Griffiths \cite{griffiths1968periods}. The observations in this chapter cement the ties between the Tannakian approach of Kim and the flag variety-theoretic approach of Lawrence and Venkatesh, and provide a basis for generalizing Tannakian Coleman integration, cf. Proposition \ref{growth}. Since all realizations will be de Rham for this section, we often write $\mathcal{G}$ for $\mathcal{G}^{dR}$, etc.

\begin{defn}
Let $\mathcal{G}$ be a group in filtered $\phi$-modules over $K$, i.e. whose coordinate ring is a filtered $\phi$-module over $K$ and whose comultiplication and counit morphisms are morphisms of filtered $\phi$-modules. Let $T$ be a scheme over $K$, and give $\mathcal{O}(T)$ the trivial filtration and Frobenius induced by the Frobenius action on $K$.

We call $P/T$ a torsor for $\mathcal{G}_T$ if it is a torsor in the category of schemes, and if $P$ is a scheme in filtered $\phi$-modules over $T$ and the coaction map is a map of filtered $\phi$-modules.
\end{defn}

\subsection{A new definition for admissible torsors}
First, let us recall why we care about the homogeneous space $\mathcal{G}/F^0\mathcal{G}$. For this, we have the following result of Kim in the unipotent case. Recall that, when $\mathcal{G}$ is a unipotent group in the category of filtered $\phi$-modules, Kim calls a $\mathcal{G}_T$-torsor $P$ \textit{admissible} if it has a Frobenius-invariant trivialization, and if $F^0P$ has some trivialization as a $F^0\mathcal{G}$-torsor. (We do not require any compatibility between the two trivializations, but isomorphisms between admissible torsors do have to be compatible with Hodge filtrations and Frobenius actions.) The following proposition \cite[Proposition 1]{kim2005unipotent} gives a moduli-theoretic interpretation of the above homogeneous space. Recall that the Lang map on a group functor with a self-map $\phi$ is the map $g \mapsto g^{-1}\phi(g)$. If $\phi$ is an automorphism of the coordinate ring of $\mathcal{G}$, the Lang map on $\mathcal{G}$ is the one induced by $\phi$ by pullback.
\begin{prop}
Let $\mathcal{G}$ be a pro-unipotent group in the category of filtered $\phi$-modules over $K$ such that $\mathcal{G}$ has a bijective Lang map. For any $T \in Sch_{K}$, the points $\left( \mathcal{G}/F^0\mathcal{G}\right) (T)$ parameterize admissible $\mathcal{G}_T$-torsors.
\end{prop}
One sends an admissible torsor with trivializations $p^{crys}$ and $p^H$ to the unique element of $\mathcal{G}_T$ sending $p^{crys}$ to $p^H$ via right-multiplication. Changing the choice of $p^H$ shows that the aforementioned assignment is well-defined up to right multiplication by $F^0\mathcal{G}_T$, and the map is constructed. One then proves surjectivity and injectivity without too much difficulty (see the proof of the next proposition.)

We generalize the above proposition to the context where the Lang map is not necessarily bijective.
\begin{defn}
Let $X$ be a scheme in $Isoc_{\kappa}^K$ with corresponding Frobenius automorphism $\phi \colon X \to X$, the pullback of the Frobenius on $\mathcal{O}_X$. We define the scheme $X^{\phi=1}$ by the fibered product
\[
\begin{tikzcd}
X^{\phi=1} \arrow{r} \arrow{d} & X \arrow{d}{id \times \phi} \\
X \arrow{r}{\Delta} & X \times X
\end{tikzcd}
\]

The definition is similar if $X$ is a scheme over an affinoid algebra $\Lambda$ that is semi-linear with respect to the Frobenius action on $\Lambda$ given by extending that of $K$.
\end{defn}

We are ready for our first of two generalizations of the concept of admissibility. Recall (Definition \ref{whatisatorsor}) that a torsor $P$ for an algebraic group $N$ over a rigid space $Z$ is a $N^{an}$-torsor over $Z$ for the Tate topology. If $N$ is furthermore a scheme in filtered isocrystals, we require that locally on $Z$, the algebraic restrictions $P^{alg}$ are schemes in filtered isocrystals over $K$ (which agree on overlaps) and that the coaction map is a morphism of isocrystals (not necessarily preserving filtration) over $K$.

\begin{defn}
Let $\mathcal{G}$ be a pro-algebraic group in filtered $\phi$-modules over $K$ (i.e. coordinate ring and all group structure maps live in $MF_K^{\phi}$), and $T$ a rigid-analytic variety over $K$. Then we call a $\mathcal{G}_T$-torsor admissible if it has reductions of structure group to both $F^0\mathcal{G}$ and $\mathcal{G}^{\phi=1}$. In other words, it is equipped with two analytic subvarieties, one of which is a $F^0\mathcal{G}$-torsor and one of which is a  $\mathcal{G}^{\phi=1}$-torsor in the above sense.
\end{defn}

The point of this definition is:
\begin{prop} \label{admissible}
Let $X$ be a variety with good reduction over a $p$-adic field $K$ that is the complement of a normal crossings divisor relative to $\mathcal{O}_K$, and $\mathcal{L}^{dR}$ a $p$-adic vector bundle with integrable connection on $X$. Let $\mathcal{L}^{crys}$ denote the associated isocrystal on the special fiber, and suppose that there is a $\mathbb{Q}_p$-local system on $X$ which is associated to $\mathcal{L}^{crys}$ in the sense of \cite[9.2.2]{olsson2011towards}; suppose that all of these realizations have unipotent local monodromy \cite[4.14]{olsson2011towards}.

If $S^{dR}$ and $\mathcal{G}^{\phi=1}$ are semisimple and simply connected, then for all $b, z \in X(K)$, the de Rham path torsor $P^{dR}_{b,z}$ for $\pi_1^{rel, dR}(X,b)$ is admissible.

\end{prop}
\begin{rem}
The condition that $\mathcal{G}^{\phi=1}$ be semisimple and simply connected seems very strong, but it is enough in geometric situations that $S^{\phi=1}$ be semisimple and simply connected, since the two groups are isomorphic if $S^{\phi=1}$ is connected (\ref{growth}.) This can be checked by hand with knowledge of the Frobenius action in applications. The group $S^{dR}$ will always be semisimple in geometric situations, due to ``Deligne semisimplicity'' \cite{deligne1971theorie} in any of the categories in which we work.

There is always an associated $\mathbb{Q}_p$-local system in the Gauss--Manin case, and local monodromy will always be locally nilpotent there.
\end{rem}
\begin{proof}
Write $\mathcal{G} \defeq \pi_1^{rel, dR}(X,b)$ and $P \defeq P^{dR}_{b,z}$. First of all, we claim that the subscheme $P^{\phi=1}$ is a $\mathcal{G}^{\phi=1}$-torsor. By compatibility of the action map with Frobenius, $\mathcal{G}^{\phi=1}$ certainly preserves $P^{\phi=1}$. Furthermore, if $T$ is a field over which $P$ has a point, then for any two points $x,y \in P^{\phi=1}(T)$, there exists a unique element of $\mathcal{G}(T)$ taking $x$ to $y$. By uniqueness and Frobenius-compatibility, it is easy to see that this element must lie in $\mathcal{G}^{\phi=1}$. By assumption, $\mathcal{G}^{\phi=1}$ is simply connected and semisimple, so that by Kneser's theorem \cite[Section 3.1]{ion2001galois} $P^{\phi=1}$ has a point.

Moving on to the Hodge filtration, consider the coaction map $$\mathcal{O}(P^{dR}) \overset{\mu_{dR}} \to \mathcal{O}(P^{dR}) \otimes \mathcal{O}(\mathcal{G}^{dR}).$$ Note that this coaction is a morphism of filtered algebras. The key is that this coaction is induced functorially from the \'etale realization in the sense that if we denote the coaction of the \'etale fundamental groups by $\mu_{\acute{e}t}$, we have $$D_{dR}(\mu_{\acute{e}t}) = \mu_{dR}.$$

We now show that $F^0P^{dR}$ is an $F^0\mathcal{G}^{dR}$-torsor. The fact that the torsor coaction is compatible with filtrations implies that the action map restricts to $$F^0\mathcal{G}^{dR} \times F^0P^{dR} \to F^0\mathcal{G}^{dR}.$$ Suppose that $T$ is a scheme such that $F^0P^{dR}$ has a $T$-point, or in other words a morphism $$\mathcal{O}(P^{dR})/F^1 \mathcal{O}(P^{dR}) \to \mathcal{O}(T).$$ Composing the coaction map with this morphism we get a map $$\mathcal{O}(P^{dR}) \to \mathcal{O}(T) \otimes \mathcal{O}(\mathcal{G}^{dR}),$$ and an induced map $$\mathcal{O}(T) \otimes \mathcal{O}(P^{dR}) \overset{i} \to \mathcal{O}(T) \otimes \mathcal{O}(\mathcal{G}^{dR})$$ which is an isomorphism by the torsor property. Our goal is to show that the induced map $$\mathcal{O}(T) \otimes \mathcal{O}(P^{dR})/F^1 \to \mathcal{O}(T) \otimes \mathcal{O}(\mathcal{G}^{dR})/F^1$$ is an isomorphism, for which it will be sufficient to show that the map $i$ is strictly compatible with the filtrations on each side.

By Olsson's comparison, both $\mathcal{O}(S^{dR})$ and $\mathcal{O}(P^{dR})$ are colimits of modules which are admissible in the usual sense of $p$-adic Hodge theory. But morphisms between admissible filtered $\phi$-modules are strict \cite[Theorem 8.2.11]{brinon2009cmi}, so by taking colimits the coaction map $\mathcal{O}(P^{dR}) \to \mathcal{O}(P^{dR}) \otimes \mathcal{O}(S^{dR})$ is strict.

Furthermore, the Galois cohomology of $S^{dR}$ vanishes by Kneser's theorem. Since $F^0S^{dR}$ is a parabolic subgroup of $S^{dR}$ \ref{parabolic}, its Galois cohomology injects into that of $S^{dR}$  \cite[p. 129]{ion2001galois}. Thus $F^0S^{dR}$ has trivial Galois cohomology; but $F^0\mathcal{G}$ is an extension of $F^0S^{dR}$ by the unipotent $F^0\mathcal{U}$, which itself has trivial Galois cohomology over $K$. We conclude that $F^0\mathcal{G}$ itself has trivial Galois cohomology, and thus that $F^0 P$ has a point, and we are done.

\end{proof}

\begin{prop}
Let $\mathcal{G}$ be a pro-algebraic group in $FilIsoc_{\kappa}^K$. Then the pro-rigid analytic stack over $K$ $$\left[ \mathcal{G}^{\phi=1, an} \backslash \mathcal{G}^{an}/F^0\mathcal{G}^{an}\right],$$ is the moduli space of admissible torsors for $\mathcal{G}$.
\end{prop}
\begin{rem}
This quotient is an algebraic stack because quotients of rigid stacks by smooth rigid groups are again rigid stacks (\ref{smoothquot}); $\mathcal{G}^{\phi=1}$ is smooth because it is a group scheme over a field of characteristic zero (this is ``Cartier's theorem'' \cite[Expos\'e VII, Theorem 3.3 (ii)]{grothendieck1962schemas}); thus its analytification is smooth.
\end{rem}

\begin{rem}
We recall that this stack is defined as the stackification of the pseudofunctor which sends a rigid space $Z$ to the quotient groupoid \[\left[ \mathcal{G}^{\phi=1, an}(Z) \backslash \mathcal{G}^{an}(Z)/F^0\mathcal{G}^{an}(Z)\right]\]

Since this agrees with the algebraic stack \[\left[ \mathcal{G}^{\phi=1} \backslash \mathcal{G}/F^0\mathcal{G}\right]\] on affinoid algebras, we will often abuse notation and write the latter when it is clear.
\end{rem}

\begin{proof}
Locally in the Tate topology, the bijection works as follows, and the result follows by stackification. Let $P$ be an admissible $\mathcal{G}_T$-torsor.

Since we are looking locally, we may choose crystalline and Hodge trivializations $p^{crys}$ and $p^H$ of $P$, and consider the element $\psi(P) = (p^{crys})^{-1}p^H$. To check that this association is well-defined, consider a different choice of trivializations, say $q^{crys}$ and $q^H$, of $P$. Now by the torsor property there exists a unique element $z \in \mathcal{G}(T)$ such that $$p^{crys} = q^{crys} z.$$ Applying Frobenius to both sides of this equation and using that the crystalline trivializations are fixed, we see that $$z = \phi(z).$$ Thus $z \in \mathcal{G}^{\phi=1}$.

Reasoning with Kim, one also has $$p^{H} = q^{H} h,$$ for $h \in F^0\mathcal{G}(T)$. Putting this information together, we have $$(p^{crys})^{-1}p^H = z^{-1} (q^{crys})^{-1}q^H h,$$ so that the association $\psi$ into $\left( \mathcal{G}^{\phi=1} \backslash \mathcal{G}/F^0\mathcal{G}\right) (T)$ is well-defined. Surjectivity follows in the same way as in Kim's proof, i.e. for an element $g \in \mathcal{G}(T)$ one lets $P_g$ be the $\mathcal{G}_T$-torsor whose underlying scheme is $\mathcal{G}$, but which has Frobenius action $\phi'(g) = g^{-1}\phi(gh)$. This action evidently fixes $g^{-1}$ (and the identity serves as a Hodge trivialization), so that $\psi(P_g) = g$.

Now for injectivity. For two torsors $P_1,P_2$ we choose Frobenius and Hodge trivializations $p_i^{crys}, p_i^H$ and let $\psi_i = \psi(P_i)$. Suppose that $[ \psi_1 ] = [ \psi_2 ]$ in the quotient, so that there exist $z \in Z(\phi), h \in F^0\mathcal{G}$ such that $$\psi_1 = z^{-1}\psi_2 h.$$ Take a new Frobenius trivialization of $P_2$ via the formula $q^{crys}_2 = p^{crys}_2 z$. (It is Frobenius-invariant because $z$ is fixed by Frobenius.) Setting $\psi_2' = (q_2^{crys})^{-1}p_2^H$, we have the relation $$\psi_1 = \psi'_2 h.$$

Then our choice of Frobenius trivializations gives an isomorphism between $P_1$ and $P_2$: $$P_1 \to \mathcal{G} \to P_2$$ $$p^{crys}_1 x \mapsto x \mapsto q^{crys}_2 x.$$

To show injectivity, we need to check that this isomorphism is compatible with Hodge filtrations. Since $p^{crys}_1 = p^H_1 \psi_1^{-1}$, and similarly for $q^{crys}_2$, the above isomorphism is given by $$p^H_1 (\psi_1^{-1} x) \mapsto p^H_2 (\psi'_2)^{-1} x,$$ in other words by $$p^H_1  x  \mapsto p^H_2 (\psi_1 (\psi'_2)^{-1} x).$$ Now indeed $\psi_1 (\psi'_2)^{-1} \in F^0\mathcal{G}$, so that the previous map is just the operation of changing the Hodge trivialization. This operation preserves the Hodge filtration on torsors, and thus implies injectivity.

\end{proof}
Fortunately, one needn't worry about the growth of $\mathcal{G}^{\phi = 1}$ in the sequence of pushouts $\mathcal{G}_n$: it only appears ``on the $S$-level,'' as we see here.
\begin{prop} \label{growth}
Preserve the notations of the previous proposition. Suppose, furthermore, that $\mathcal{U}^{crys}$ has a bijective Lang map, and that Frobenius acts by conjugation by a Frobenius semi-linear map on $S$. Let $Z(\phi)$ denote the centralizer of $\phi$ in $S$.

Then the natural map $\mathcal{G}^{\phi=1} \to S^{\phi=1}$ is an isomorphism on connected components of the identity, and $\pi_0(\mathcal{G}^{\phi=1})$ injects into $\pi_0(S^{\phi=1})$. In particular, $dim \mathcal{G}^{\phi=1} = dim Z(\phi)$.
\end{prop}
\begin{proof}
The fact that the exact sequence $$1 \to \mathcal{U} \to \mathcal{G} \to S \to 1$$ is $\phi$-equivariant provides the homomorphism of group functors $\alpha \colon \mathcal{G}^{\phi=1} \to S^{\phi=1}$. The bijective Lang map assumption implies that $\mathcal{U}^{\phi=1}$ is trivial (Theorem \ref{integration}), which shows that this homomorphism is injective.

We push out via the lower central series of $\mathcal{U}$, pass to Lie algebras, and apply $\phi - 1$ to each element of the exact sequence $$0 \to \mathfrak{u}_n \to \mathfrak{g}_n \to \mathfrak{s} \to 0.$$ By the snake lemma, $$\mathfrak{g}_n^{\phi=1} \to \mathfrak{s}^{\phi=1} \to \mathfrak{u}_n/(\phi - 1)\mathfrak{u}_n$$ is exact. But by rank-nullity, $$\mathfrak{u}_n/(\phi - 1)\mathfrak{u}_n \simeq \mathfrak{u}_n^{\phi=1}$$ and the latter group is zero, so the level-$n$ Frobenius-invariant parts of the Lie algebras are isomorphic. By taking inverse limits along the lower central series, we have $\mathfrak{g}^{\phi=1} \simeq \mathfrak{s}^{\phi=1}$.

Surjectivity on the level of Lie algebras, along with the injectivity of $\mathcal{G}^{\phi=1} \to S^{\phi=1}$, shows that $\mathcal{G}^{\phi=1} \to S^{\phi=1}$ is an isomorphism on connected components of the identity. Injectivity on the level of groups implies, in particular, injectivity on the level of $\pi_0$. By assumption, $Z(\phi) = S^{\phi=1}$, and we are done.
\end{proof}
\begin{rem}
One can deduce equality of connected components in some special situations, although the author does not know if it should hold in general. For example, if $S$ is simply connected and $\phi$ is semi-simple, a theorem of Springer--Steinberg and Digne--Michel says that $Z(\phi)$ is connected, which would imply equality on connected components above. This theorem is applicable, for example, when $\mathcal{L}^{dR}$ come from geometry in a context where Frobenius semisimplicity is known and $S$ is the symplectic group. 
\end{rem}

\subsection{Flag varieties and Rapoport--Zink spaces}
A word of motivation on the machinery that follows. The ideas behind this section were inspired by Hain's construction of the Albanese manifold for the relative completion \cite[Section 13.1]{hain2016hodge}, and the reader who is more comfortable with complex constructions than $p$-adic ones is directed there for a sense of where we are going. Recall that the theory of period domains associates to each polarizable Hodge structure $V_b$ a partial flag variety $$F = S/B,$$ where $S$ is a symplectic or orthogonal group for $V_b$ and $B$ is a parabolic subgroup fixing the flag associated to the Hodge filtration on $V_b$. (The deep part of the theory actually cuts out a smaller space which classifies flags that satisfy Riemann's bilinear relations.)

The complex case of period mappings proceeds as follows: suppose that $f \colon Y \to X$ is a smooth and proper map of smooth complex manifolds. Then locally (in the $C^{\infty}$ sense) around a point $b \in X$ the fibration trivializes, by a theorem of Ehresmann. Let $N$ be such a trivializing neighborhood, which we may shrink and assume is simply connected. For any point $b' \in X$ this trivialization gives a diffeomorphism $Y_b \simeq Y_{b'}$, hence an isomorphism of singular cohomology $$H^i_{sing}(Y_b, \mathbb{Z}) \simeq H^i_{sing} (Y_{b'}, \mathbb{Z}).$$ Finally, the comparison between singular and de Rham cohomologies induces an isomorphism $$H^i_{dR}(Y_b, \mathbb{C}) \simeq H^i_{dR} (Y_{b'}, \mathbb{C}).$$ Furthermore, these cohomologies have the same Hodge filtration type, meaning their associated gradeds have matching dimensions. Thus we obtain an analytic map from $N$ to the flag variety parameterizing flags with a fixed filtration type.

The theory of the Gauss--Manin connection allows a similar approach in the $p$-adic world, and we would like to emulate the complex case. But we have our eyes set on the entire relative completion: when we start talking about varying torsors instead of varying filtrations on a fixed vector space, the Tannakian framework for filtrations and flag varieties becomes essential.

Let us dive into the weeds and apply the theory of $p$-adic period domains as explained in the wonderful textbook of Dat, Orlik, and Rapoport \cite{dat2010p}. In order to recall the construction, we recap some definitions. For $K/K_0$ fields, the category $Fil^K_{K_0}$ is the category of $K_0$-vector spaces $V$ equipped with a filtration on $V \otimes_{K_0} K$, with morphisms taking one filtration into another after base change to $K$.
\begin{defn}
Suppose $V$ is a vector space with filtration $F$. Then the degree of $(V, F)$ is \[deg(V, F) = \sum_i i \cdot gr_i (V, F).\] The slope of $(V, F)$ is \[\mu(V, F) = deg(V, F)/dim(V).\]
\end{defn}

\begin{defn}
Let $(\mathcal{T}, \omega)$ be a neutral Tannakian category over $K_0$, and $K$ an extension of $K_0$. By a filtration over $K$ on $\omega$ we refer to a factorization of $\omega$ through $Fil_{K_0}^K$. We denote the set of all such filtrations by $Fil_K\omega$, and gradings are defined similarly.
\end{defn}
\begin{defn}
Write $\omega_{S}$ for the standard fiber functor on $Rep(S)$, for $S/K$ an algebraic group. A $\mathbb{Q}$-1-parameter subgroup, written $\mathbb{Q}$-1PS, of an algebraic group $S$ is a map $\lambda \colon \mathbb{D} \to S$ defined over $\bar K$, where $\mathbb{D} = \varprojlim \mathbb{G}_m$. (The limit here is taken by divisibility.)
\end{defn}
One shows that $\mathbb{Q}$-1-parameter subgroups are in bijection with $\mathbb{Q}$-gradings of $\omega_S$. If $\mathbb{D} \to S$ factors through the first projection $\mathbb{D} \to \mathbb{G}_m$, this is just the simple fact that a morphism $\mathbb{G}_m \to S$ is equivalent to a $\mathbb{Z}$-grading on $Rep(S).$ All of the filtrations we consider will be of this type.

Now $Gal(\bar K/K)$ acts on the set of $\mathbb{Q}-1PS$ of $S$ by conjugation, since they are each defined over $\bar K$.
\begin{defn}
A PD-pair $(S, \mathcal{N})$ is a reductive group $S$ over $K$ along with a $G_K$-conjugacy class of $\mathbb{Q}$-1PS $\mathcal{N}$ defined over $\bar K$.
\end{defn}
And, last but not least, we need the concept of a type of a filtration. Recall that a $\mathbb{Q}$-grading $V = \oplus_{\alpha \in \mathbb{Q}} V_\alpha$ on a vector space induces a decreasing $\mathbb{Q}$-filtration by the formula $F^\beta V = \oplus_{\alpha \geq \beta} V_\alpha$, and this construction generalizes in an obvious way to filtrations and gradings on fiber functors.
\begin{defn}
Let $\mathcal{F} \in Fil_K\omega_S$. Suppose there exists a $\mathbb{Q}$-1PS $\sigma \colon \mathbb{D} \to S$ (i.e. a $\mathbb{Q}$-grading of $\omega_S$) whose associated filtration is isomorphic to $\mathcal{F}$. If $\mathcal{N}$ denotes the conjugacy class of $\sigma$. We call $\mathcal{N}$ the type of $\mathcal{F}$.
\end{defn}
\begin{rem}
When there exists such a 1-PS, we call $\mathcal{F}$ splittable. By the next theorem, if $\mathcal{T}$ is $Rep(S)$ for $S$ a reductive group in characteristic zero, then any filtration on $\omega_S$ is splittable.

\end{rem}

Let $\mathcal{F}$ be a filtration on $\omega_{S}$. We will need the subgroup $P_{\mathcal{F}} \subseteq S$, defined by $P_{\mathcal{F}}(R) = \{a \in Aut(\omega)(R) | a(\mathcal{F}^i V \otimes R ) \subseteq \mathcal{F}^i V \otimes R \}$.

The key to constructing flag varieties is that filtrations on fiber functors can be seen group-theoretically through parabolic subgroups.
\begin{thm*}{cf. \cite[Theorem 4.2.13]{dat2010p}}
Let $S$ be a reductive algebraic group over $K$, a field of characteristic zero. Then $P_{\mathcal{F}}$ is a parabolic subgroup of $S$, and $Lie(P_{\mathcal{F}}) = F^0Lie(S)$, where $Lie(S)$ is given a filtration via $\mathcal{F}$ and the adjoint representation of $S$. Furthermore, $\mathbb{Q}$-filtrations of $\omega_S$ which are defined over $\bar K$ are in bijection with $P_{\mathcal{F}}$ conjugacy classes of $\mathbb{Q}-1PS$ of $S$.
\end{thm*}

\begin{defn}
Let $(S, \mathcal{N})$ be a PD-pair, and denote by $\Gamma_\mathcal{N}$ the stabilizer of $\mathcal{N}$ in $G_K$. The reflex field of $\mathcal{N}$ is the field $E = (\bar K)^{\Gamma_{\mathcal{N}}}$.
\end{defn}
Now any PD-pair gives rise to a generalized flag variety which we will denote $\mathcal{F}(S, \mathcal{N})$. Simply put, it classifies filtrations on the fiber functor of $Rep(S)$ of ``type $\mathcal{N}$.''
\begin{thm*}{\cite[Theorem 6.1.4]{dat2010p}}

There is a smooth, projective variety $\mathcal{F}(S, \mathcal{N})$, defined over the reflex field $E$, with the property that $$\mathcal{F}(S, \mathcal{N})(E') = \{F \in Fil_{K}(\omega_S \otimes_K E') | F \text{ has type } \mathcal{N} \}.$$ Here $E'$ runs over finite extensions of $E$.
\end{thm*}
\begin{rem}
In our case, $E$ will be a trivial extension of $K$ -- this is because the parabolic associated to the filtration we will care about is $F^0S$, which is defined over $K$.
\end{rem}

The interaction with $p$-adic Hodge theory comes through particular kinds of filtrations on fiber functors, the so-called semistable and weakly admissible filtrations. They depend on a concept of ``slope'' for representations of PD-pairs. Following the definition of \cite{dat2010p} we let $\kappa$ denote a perfect field of characteristic $p >0$, $K_0$ the fraction field of its Witt vectors, and $K$ an extension of $K_0$. Now every isocrystal has a unique slope decomposition, $$W = \bigoplus_{\alpha \in \mathbb{Q}} W_{\alpha};$$ this means that for $\alpha = \frac{r}{s}$ (and $(r,s)=1$) there exists a $W(\kappa)$-lattice $\Lambda$ where $\phi^s$ acts as $p^r$ on $W_{\alpha}$.

Finally, given a filtered isocrystal over $K$ (i.e. an isocrystal $W$ over $\kappa$ with a filtration $\mathcal{F}$ on $W \otimes_{K_0} K$), we define another filtration $\mathcal{F}_0$ on $W$ by $$\mathcal{F}_0^{\beta} = \bigoplus_{\alpha \leq -\beta} W_{\alpha}.$$ Actually, we must briefly distinguish the whole isocrystal from the underlying vector space $W$ and so we write $\mathbb{W} = (W, \mathcal{F}, \phi).$

\begin{defn} \label{weakdefn}
Define the slope and degree of $\mathbb{W} = (W, \mathcal{F}, \phi) \in FilIsoc^K_\kappa$ by $$deg(\mathbb{W}) = deg(W, \mathcal{F}) + deg(W, \mathcal{F}_0)$$ and $$\mu(\mathbb{W}) = \mu(W, \mathcal{F}) + \mu(W, \mathcal{F}_0).$$

We say that $\mathbb{W}$ is semi-stable if any sub-isocrystal $\mathbb{W}' = (W, \mathcal{F}', \phi) \subseteq \mathbb{W}$ has slope at most $\mu(\mathbb{W})$, where the filtration on $W'$ is induced by that of $W$. We call $\mathbb{W}$ weakly admissible if it is semi-stable and has degree zero.
\end{defn}

Recall that a group in isocrystals over $\kappa$ is a group scheme over $W(\kappa)$ whose coordinate ring, comultiplication, inverse, and counit maps all live in the category of isocrystals. Without comment, we use the same terminology for any symmetric monoidal category. If $\mathbb{D}_{isoc}$ denotes the Tannakian fundamental group of the category of isocrystals over $\kappa$, then $\mathbb{D}_{isoc}$ acts (by definition) on the coordinate ring of any group in isocrystals over $\kappa$, and thus on the group itself. We call this action the canonical action.
\begin{defn}
An augmented group scheme consists of a group in isocrystals $S$ and a functor $\nu \colon Rep(S) \to Isoc_{\kappa}$ such that conjugating by the induced map $\mathbb{D}_{isoc} \to S$ is the same action as the canonical action given above.
\end{defn}
\begin{exmp}
The natural fiber functor given by a $\kappa$-rational point on the category of isocrystals on a scheme $X$ over a finite field $\kappa$ factors through the category of isocrystals over $\kappa$. This structure gives the crystalline fundamental group the structure of an augmented group scheme.
\end{exmp}

\begin{defn}
Let $S$ be an augmented group scheme. In particular, the adjoint representation and the augmentation combine to give $Lie(S)$ the structure of an isocrystal. A filtration $\mathcal{F}$ over $K$ on $\omega_S$ is called semi-stable or weakly admissible if the filtered isocrystal $(Lie(S), \mathcal{F})$ is semi-stable or weakly admissible.
\end{defn}

Having dispensed with these preliminaries, the mainstay of the theory of $p$-adic period domains is as follows \cite[Proposition 9.5.3]{dat2010p}.
\begin{thm}
Let $(S, \mathcal{N})$ be a PD-pair over a $p$-adic field $K_0$ such that $S$ is also an augmented group scheme. There exists an open rigid-analytic subspace $$\mathcal{F}(S, \mathcal{N})^{wa} \hookrightarrow \mathcal{F}(S, \mathcal{N})^{an}$$ which parameterizes weakly admissible $K$-filtrations on $\omega_{Rep(S)}$ of type $\mathcal{N}$.
\end{thm}
This means, in other words, that for each field $K'$ over the reflex field of $(S, \mathcal{N})$, $\mathcal{F}(S, \mathcal{N})^{wa}(K')$ is in bijection with $K'$-filtrations on $N$ that are weakly admissible.

\subsection{Applying period domains to group schemes that come with a filtration}
In our work, the filtration on a fiber functor will come from a filtration on the coordinate ring of the fundamental group of the underlying Tannakian category. Let $S$ be a reductive group in the category $FilIsoc_{K_0}^K$, i.e. a group whose coordinate ring has the structure of a $K$-filtered isocrystal over $K_0$ with $F^0\mathcal{O}(S)\otimes K = \mathcal{O}(S) \otimes K$, and for which the comultiplication and counit maps are compatible with the filtration. The following construction was suggested by Paul Ziegler.
\begin{prop} \label{filtration}
Suppose that the comultiplication map on $\mathcal{O}(S)$ is strict for the given filtration. There is a unique filtration $\mathcal{F} \in Fil_{K}\omega_S$, ($\omega_S$ being the fiber functor associated to representations of $S$), such that the induced right action of $S$ on $\mathcal{O}(S)$ reproduces the given filtration on $\mathcal{O}(S) \otimes_{K_0} K$.
\end{prop}
\begin{proof}
If $V$ is a representation of $S$, then the coaction \[V \to V \otimes_{K_0} \mathcal{O}(S) \] is injective. (Indeed, the counit splits the injection.) Tensoring up to $K$, we have a map \[\iota \colon V_K \to V_K \otimes \mathcal{O}(S)_K, \] where now we've written subscripts instead of $\otimes_{K_0} K$. Give the $V_K$ on the right the trivial filtration, so that the tensor product on the right has a filtration induced only by the one on $\mathcal{O}(S)_K$. Then we let \[F^iV_K = \{ x \in V_K | \iota(x) \in F^i (V_K \otimes \mathcal{O}(S)_K)\}.\]

This filtration is functorial. When $V = \mathcal{O}(S)$ the coaction map is the comultiplication map $\mu$; the counit then provides a sequence \[\mathcal{O}(S) \overset{\mu} \to \mathcal{O}(S) \otimes \mathcal{O}(S) \overset{1 \otimes \epsilon} \to \mathcal{O}(S)\] whose composition is the identity. We must show that even though strictness means strictness for the tensor product filtration on $\mathcal{O}(S) \otimes \mathcal{O}(S)$, it is still strict when the $\mathcal{O}(S)$ on the left of the tensor product is given the trivial filtration.

Write $fil(x) = max_i (x \in F^i\mathcal{O}(S))$. Consider $f \in \mathcal{O}(S)$; $\mu(f) = \sum g_i \otimes h_i$, and 
\begin{align} \label{counit}
f = \sum \epsilon(g_i) \otimes h_i. \end{align}
By strictness, $fil(f) = fil (\sum g_i \otimes h_i)$, and then by \ref{counit}, we have $fil(f) = fil(\sum(h_i))$; in other words, $\mu$ is still strict for the one-sided filtration.
\end{proof}

\begin{defn}
We write $\mathcal{F}_{can}$ for the canonical filtration on $\omega_S$ induced by the previous proposition.
\end{defn}
\begin{rem}
Let $S/K$ be a group in $FilIsoc_{\kappa}^K$. Then, setting \[\mathfrak{m} = \ker(identity \colon \mathcal{O}(S) \to K),\] we have $Lie(S) = (\mathfrak{m}/ \mathfrak{m}^2)^\vee$ acquires a filtration via the subspace, quotient, and dual filtration. Furthermore, this filtration agrees with $\mathcal{F}_{can}Lie(S)$ (from the adjoint representation) due to Proposition \ref{filtration}. Now if $\mathcal{O}(S)$ is weakly admissible, then $\mathcal{F}_{can} Lie(S)$ is also weakly admissible. To see this, we can use the result of Fontaine and Colmez that every weakly admissible filtered $\phi$-module is actually admissible, i.e. comes from a crystalline representation.

In particular, $\mathcal{O}(S)$ comes from a crystalline representation $H$ of $G_{K}$. This representation is in fact a commutative Hopf algebra, because the functor \[V_{crys} \colon MF^{\phi}_K \to G_K-Rep/\mathbb{Q}_p\] (which is the inverse of $D_{crys}$) is functorial and preserves tensor products. Now $Lie(\Spec(H))$ is crystalline, because subobjects, quotients, and duals of crystalline representations are crystalline. Finally, $D_{crys}\left(Lie(\Spec(H))\right) \otimes K = Lie(S)$, so we see that $Lie(S)$ is admissible and hence weakly admissible.
\end{rem}

Now suppose that $S$ is defined over $K$. As before, we define \[F^0S = \Spec(\mathcal{O}(S)/F^1\mathcal{O}(S)).\] Recall the existence of a parabolic subgroup $P$ that sends the filtration on the fiber functor to itself.
\begin{prop} \label{parabolic}
If $\omega_S$ is given the filtration $\mathcal{F}_{can}$, we have an equality $F^0S = P$.
\end{prop}
\begin{proof}
Let $V_R$ be an $R$-representation of $S_R$, for $R$ a $K$-algebra. Then if $v \in F^iV_R$ and $g \in S(R)$ corresponding to $g^* \colon \mathcal{O}(S) \to R$, the action of $g$ is given by the composition 
\begin{align}
V_R \hookrightarrow V_R \otimes_R \mathcal{O}(S_R) \overset{1 \otimes g^*} \to V_R 
\end{align}
The key point is that the comodule map is strict for filtrations, even when $V_R \otimes \mathcal{O}(S_R)$ is given $F_{can}$ and $\mathcal{O}(S_R)$ has its given filtration, which one checks using strictness of the coproduct for $S$. Then we have \[F^iV_R = V_R \cap (\bigoplus_{j+k = i \text{ and } j,k \geq 0} F^k V_R \otimes F^j\mathcal{O}(S_R) ).\] If $g^*$ vanishes on $F^1\mathcal{O}(S_R)$, then $F^iV_R$ is sent to $F^i V_R$ by $g^*$, and we are done.
\end{proof}
In other words, $S/F^0S$ is the flag variety parameterizing filtrations on $\omega_S$ of type $\mathcal{F}_{can}$.

\subsection{Sharpening the definition of admissible torsors even more}
We now define a concept of weak admissibility for torsors.

Let $S$ be an augmented group scheme over $Isoc_{\kappa}$, and $E$ a finite extension of $K$. Consider $P$, a torsor for $S_E$ with reductions of structure to $S^{\phi=1}_E$ and $F^0S_E$. Denote these reductions of structure by $P^{\phi=1}$ and $F^0P$ and take $E$-trivializations $p^{crys}$ and $p^H$, respectively.

Finally, let $s \in S(E)$ denote the element $(p^{crys})^{-1} p^H$. The filtration $\mathcal{F}_{can}$ on $\omega_{S}(Lie(S))$ can be shifted by $s$: Let $\mathcal{F}_s$ denote the new filtration which is given by \[\mathcal{F}_s^i(Lie(S)) = Image(ad(s) \colon \mathcal{F}^i_{can} Lie(S) \to Lie(S)).\] (Here $ad$ means the adjoint action of a group on its Lie algebra.) In fact, this new filtration obviously extends to the whole of $\omega_S$. Recall that $Lie(S)$ has a semi-linear action of Frobenius, just because $S$ is an augmented group scheme.

\begin{defn}
We say that $P$ is weakly admissible if $Lie(S)$, equipped with the filtration $\mathcal{F}_s$ and this Frobenius action, is weakly admissible in the sense of Definition \ref{weakdefn}.
\end{defn}

\begin{rem} \label{independentchoice}
It is not hard to see that the notion of weak admissibility does not depend on the choice of trivializations for the reductions of structure: Suppose another choice results in the filtration $\mathcal{F}_{s'}$. Then $s' = (x^{crys})^{-1}s (x^H)$, for $x^{crys} \in S^{\phi=1}$ and $x^H \in F^0(S)$. First, $ad(x^H)$ preserves the filtration by Proposition \ref{parabolic} (thus it doesn't even change the induced filtration). Second, $ad(x^{crys})$ is Frobenius-equivariant by definition (i.e. it is a usual map of filtered $\phi$-modules); but then the filtration given by $Image(ad(x^{crys}) \colon F^i Lie(S) \to Lie(S))$ is weakly admissible if and only if $F$ is, for any filtration $F$. (This is the exact same argument as Proposition \ref{preservation}, so we leave it there.)
\end{rem}

We may now state the sharpest admissibility criterion that we will use in this work. 
\begin{defn} 
Let $\mathcal{G}$ be a group scheme in the category $FilIsoc_{\kappa}^{K}$; furthermore, fix a map $$\pi \colon \mathcal{G} \to S,$$ $S$ a reductive group which is an augmented group scheme in the same category. (We require $\pi$ to be a map of schemes in filtered isocrystals.) Finally, let $P$ be a $\mathcal{G}$-torsor in $K'$-filtered $\phi$-modules for some $K'/K$. 

We say that $P$ is $\pi-$doubly admissible if $P$ is admissible in the sense of having reductions of structure to $F^0\mathcal{G}$ and $\mathcal{G}^{\phi=1}$, and if furthermore $P_S$, the pushforward of $P$ along $\pi$, is weakly admissible in the sense just defined. In general we will suppress $\pi$ from our notation and just say that $P$ is doubly admissible.
\end{defn}

\begin{rem}
A well-known result of Colmez and Fontaine \cite{fontaine2000construction} says that a filtered $\phi$-module is weakly admissible if and only if it is of the form $D_{crys}(V)$ for $V$ a crystalline representation. Thus we could just as easily say that the second condition for double admissibility is that the pushout $P_S$ is admissible, in the sense that $\mathcal{F}_sLie(S)$ is admissible as a filtered $\phi$-module.
\end{rem}

\subsection{The de Rham moduli space: moduli of doubly admissible torsors}

Armed with the notion of double admissibility for torsors, we refine the notion of the de Rham moduli space. But first, an elementary observation.
\begin{prop} \label{preservation}
Let $(S,\mathcal{N})$ be a PD-pair in filtered $\phi$-modules over $K$. The rigid subspace $\breve{\mathcal{F}}^{wa}(S, \mathcal{N})$ of the flag variety is preserved by the natural action of the subgroup $S^{\phi=1}$.
\end{prop}
\begin{proof}
Let $s \in S^{\phi=1}$, and consider $Lie(S)$ with some filtration $\mathcal{F}$ on the fiber functor of $S$ of type $\mathcal{N}$. Let $W$ be a sub-isocrystal of $Lie(S)$; the $\phi=1$ condition implies that $\phi(sw) = s\phi(w)$ for $w \in W$. Thus sub-isocrystals of $W$ are in bijection with sub-isocrystals of $s(W)$. It is obvious that $W$ and $s(W)$ have the same Hodge polygon, since their filtrations have the same type.

The Newton polygon of any such $W$ is also fixed by $S^{\phi=1}$: since $\phi$ acts on $s(W)$ via $\phi(sw) = s\phi(w)$, the slopes of $s(W)$ are the same as those of $W$. Thus the Newton polygon of $s(W)$ lies over its Hodge polygon iff the same is true of $s(W)$.
This shows that $(Lie(S), \mathcal{F})$ is weakly admissible iff $(sLie(S), s\mathcal{F})$ is.
\end{proof}

\begin{defn}
Consider a smooth $K$-variety $X$ with good reduction and a relative smooth normal crossing compactification, a vector bundle $\mathcal{L}^{dR}$ on $X$ with reductive monodromy and nilpotent local monodromy, and a basepoint $b \in X(K)$. Consider Hain/Olsson's filtration on the coordinate rings of $\pi_1^{dR,rel}$ and $S^{dR}$, and let $(S^{dR},\mathcal{N})$ be the PD-pair associated to $\mathcal{F}_{can}$. We define the de Rham moduli space $\mathcal{M}^{dR}$ associated to this data as the pullback
\[
\begin{tikzcd}
\mathcal{M}^{dR} \arrow{r} \arrow{d} & \left[ (\pi_1^{rel,dR})^{\phi=1} \backslash \pi_1^{dR,rel}/ F^0\pi_1^{dR,rel} \right] \arrow{d} \\
\left[ (S^{dR})^{\phi=1} \backslash \breve{\mathcal{F}}(S^{dR}, \mathcal{N})^{wa} \right] \arrow{r} & \left[ (S^{dR})^{\phi=1} \backslash S^{dR}/ F^0S^{dR} \right] 
\end{tikzcd}
\]
We similarly define $\mathcal{M}^{dR}_n$ for each of the finite-dimensional quotients $\mathcal{G}^{dR}_n$.
\end{defn}

The following statement follows easily from the earlier description of admissible torsors and the definition of $\mathcal{M}^{dR}$.
\begin{prop}
The stack $\mathcal{M}^{dR}$ is the moduli space of doubly admissible $\pi_1^{rel,dR}$-torsors with respect to the natural quotient $\pi_1^{rel,dR} \to S^{dR}$.
\end{prop}

We have the following fact, which follows from Proposition \ref{admissible} and the fact that the fibers of an admissible vector bundle are admissible (and thus weakly admissible) isocrystals.
\begin{cor}
Preserve the notations of the previous definition. If $\mathcal{L}^{dR}$ is admissible in the sense of Fontaine, then the relative de Rham path torsors $P_{b,z}^{dR}$, with $b,z \in X(K)$, are doubly admissible.
\end{cor}
\begin{proof}
We have already discussed the existence of Hodge and Frobenius reductions of structure for $P = P_{b,z}^{dR}$, so we discuss only the weak admissibility of the pushout $P_S$.

The rest of the proposition will follow from Proposition \ref{logisadm}, since $P^{\acute{e}t}_{b,z}$ lands in $H^1_f(G_K, \pi_1^{rel, \acute{e}t})$.

\end{proof}

\section{Analyticity of the Albanese for the relative completion; the crystalline fundamental group} \label{analyticity}
We call the map 
\begin{align}
j \colon X(\mathcal{O}_K) \to \mathcal{M}^{dR} \\
z \mapsto P^{dR}_{b,z}
\end{align}
 the relatively unipotent Albanese map, and we write $j_n$ for its finite-level versions. It is essential for the Chabauty--Kim method that the Albanese map be $p$-adically analytic. By this we mean that for every point $z \in X(\mathcal{O}_K)$, there is a $p$-adic disc $D$ around $z$ and a lift $D \to \mathcal{G}^{dR}_n(K)$ that is given by convergent power series after embedding $G^{dR}_n$ in affine space. We follow Kim's general strategy \cite[pp. 11-12, bottom of p.14-top of p.15]{kim2005unipotent}. We suppose, as always, that $\mathcal{L}^{dR}$ has unipotent local monodromy and reductive global monodromy and $X$ has good reduction and is the complement of a smooth normal crossings divisor over $\mathcal{O}_K$.

Fix a base point $b \in X(\mathcal{O}_{K})$, with $\bar b \in X_0(\kappa)$ in the special fiber of (our smooth model over $\mathcal{O}_{K}$ of) $X$.
\begin{prop}
The maps $j_n$ are $p$-adically analytic.
\end{prop}

\begin{proof}
To keep the notation tidy, we denote by $\mathcal{G}^{dR}$ any of its finite-level pushouts for the rest of this proof.
Consider the $\mathcal{G}^{dR}$-bundle $\mathcal{P}^{dR}$ over $X \times X$ with the property that its fiber over $(x,y)$ is $P^{dR}_{x,y} \defeq Isom(\omega^{dR}_x, \omega^{dR}_y)$. We denote by $P^{dR}$ its pullback to $\{b\} \times X$.

The analyticity will follow from the compatibility of the de Rham and crystalline functors, which is the statement that $$P^{dR}_{b,y} =  P^{crys}_{\bar b, \bar y} \otimes K$$ for any $y \in X(\mathcal{O}_{K})$. The fiber functor on $\mathscr{L} \in \mathcal{T}^{rel, crys}$ is given by $$(\mathcal{V}, \nabla) \mapsto \{ s \in \mathcal{V}(]\bar b[) \text{ such that }  \nabla(s) = 0 \}$$ after tensoring up to $K$. Here $(\mathcal{V}, \nabla)$ is a lift of $\mathscr{L}$ to a vector bundle with integrable connection on $X_K$.

It suffices to consider the analyticity of the Albanese only in the residue disc of $b$, so we take $y \in ]\bar b[$. In this case $$ P^{dR}_{b,y} \simeq P^{crys}_{\bar b, \bar y} \otimes_{K_0} K = \pi_1^{rel, crys}(X_0, \bar b) \otimes_{K_0} K.$$ The identity on the right hand side corresponds to the following path in $ \gamma \in P^{dR}_y$: for any vector bundle with integrable connection $(\mathcal{V}, \nabla) \in \mathcal{T}^{rel, dR}$ and $v_b$ in the fiber of $\mathcal{V}$ above $b$, there is a unique element $s \in \mathcal{V}(]\bar b[)$ such that $\nabla(s) = 0$ and $s(b) = v_b$. (This follows from the fact that the formal solutions to $\nabla(s) = 0$ are $p$-adically convergent in the entire residue disc, which in turn follows from the fact that $\mathcal{V}$ is an $F$-isocrystal \cite[4.24(ii)]{olsson2011towards}.) Then $\gamma(\mathcal{V})(v_b) = s(y)$.

The key point here is that $s(y)$ is $p$-adic analytic in $y$. In particular, taking the universal bundle $\mathcal{O}(G^{dR})$, we have a $p$-adic analytic section $\sigma^{crys} \in \mathcal{O}(G^{dR})(]\bar b[)$. Let us now construct a representative of $j_n(x)$ in a neighborhood of $y$ by constructing a Hodge trivialization and a crystalline trivialization and ``subtracting'' one from the other. 

Now $\mathcal{O}(F^0(P^{dR}))$ is an ind-vector bundle over $X$, so its rigid analytification is an ind-vector bundle over $X^{an}$. Now there exists a neighborhood $N \subseteq ]\bar b[$ of $y$ such that the torsor $F^0P^{dR}$ is trivial; $\mathcal{O}(F^0P^{dR})^{an}$ is trivial in that same neighborhood. Here $an$ refers to analytification of locally free sheaves. Let $\sigma^{H}$ be an analytic trivialization of $\mathcal{O}(F^0(P^{dR}))$ in $N$. Write $N \simeq Spm(K\langle T_1, \ldots, , T_r \rangle)$ Now the standard torsor coaction isomorphism gives rise to an isomorphism of ind-locally free sheaves $$\mathcal{O}(P^{dR})^{an} \hat \otimes (\mathcal{O}(P^{dR}_{b,b})^{an} \hat \otimes_K \mathcal{O}(X^{an})) \overset{\tau} \simeq \mathcal{O}(P^{dR})^{an} \hat \otimes \mathcal{O}(P^{dR})^{an}$$ 

The map $\sigma^{crys} \colon K\langle T_1, \ldots, T_r\rangle \to \mathcal{O}(P)$ allows us to pull back the torsor isomorphism to an isomorphism $$\mathcal{O}(_bP_b)^{an} \hat \otimes_K K\langle T_1, \ldots, T_r\rangle \overset{\iota^{crys}} \simeq \mathcal{O}(P)^{an} \hat \otimes_{\mathcal{O}(X^{an})} K \langle T_1, \ldots, T_r\rangle$$ over $N$. By definition, this is the map that, on the level of torsors over $N$, sends the identity section to $p^{crys}$.

We can do the same for $\sigma^H$ (replacing $D$ by $N$), and form the chain
$$(\mathcal{O}(_bP_b)^{an} \otimes_K K\langle T_1, \ldots, T_r\rangle) \overset{\iota^{crys}} \simeq \mathcal{O}(P)^{an} \otimes_{\mathcal{O}(X^{an})} K \langle T_1, \ldots, T_r\rangle) \overset{\iota^{H}} \simeq (\mathcal{O}(_bP_b)^{an} \otimes_K K\langle T_1, \ldots, T_r\rangle).$$ Now the identity of $P_{b,b}$ induces a map $$\mathcal{O}(_bP_b)^{an} \hat \otimes_K K\langle T_1, \ldots, T_r\rangle \to K\langle T_1, \ldots, T_r\rangle.$$

Post-composing this with the chain gives another rigid analytic map $$\xi^* \colon \mathcal{O}(_bP_b)^{an} \otimes_K K\langle T_1, \ldots, T_r\rangle) \to K\langle T_1, \ldots, T_r\rangle,$$ i.e a map $\xi \colon N \to (P_{b,b})^{an}$. The analyticity of $y \mapsto P^{dR}_{b,y}$ follows, since $\xi$ is a representative for $j$ in $N$.

\end{proof}

\subsection{Density of $p$-adic period mappings}
We would like to mimic Kim in showing that the map sending a point to its de Rham torsor is Zariski-dense in the de Rham period space \cite[Theorem 1]{kim2005unipotent}. We indicate the direction of a proof, but would need to do more work in defining complex-analytic Albanese maps for it to be complete. The modular curve case is presented in \cite[Section 13.1]{hain2016hodge}.

The following is a lemma of Lawrence and Venkatesh \cite[Lemma 3.2]{lawrence2018diophantine}. Let $F$ be a number field, and $v$ a place of $F$.
\begin{lem}
Suppose given power series $B_0, \ldots, B_N \in F[[z_1, \ldots, z_m]]$ such that all $B_i$ are absolutely convergent, with no common zero, both in the $v$-adic and complex discs \[U_v = \{\underline{z} : |\underline{z}|_v < \epsilon \} \text{ and } U_\mathbb{C} = \{\underline{z} : |\underline{z}|_\mathbb{C} < \epsilon\}.\]
Write \[\underline{B}_v = U_v \to \mathbb{P}^N_{F_v}\] \[\underline{B}_\mathbb{C} = U_v \to \mathbb{P}^N_{\mathbb{C}}\] for the corresponding maps. Then there exists a $K$-subscheme $Z \in \mathbb{P}^N_F$ whose base change to $F_v$ (resp. $\mathbb{C}$) is the Zariski closure of the image of $\underline{B}_v$ (resp. $\underline{B}_\mathbb{C}$.)
\end{lem}
Now let $X$ be defined over $F$, and suppose that there exists a holomorphic Albanese map \[j^{hol} \colon N' \to G^{dR}_{\mathbb{C}}/F^0 G^{dR}_{\mathbb{C}}\] in a complex neighborhood of the basepoint $b \in X(\mathbb{C})$ that satisfies the following conditions: First, $j^{hol}$ is given by power series defined over $F$, and second, this power series provides a representative for the $p$-adic analytic $j_n$ in a neighborhood of $b$. 

By embedding $\mathcal{G}_n^{dR}/F^0\mathcal{G}_n^{dR}$ in projective space, we then have the following statement:
\begin{cor}
Let $f \colon Y \to X$ be a Gauss--Manin situation (cf. \ref{gm}) over a number field $F$ such that $f$ and $X$ have good reduction at a place $v$ of $F$. To prove the Zariski-density of \[j_v \colon ]\bar b[ \to \mathcal{G}_{F_v}^{dR}/F^0\mathcal{G}_{F_v}^{dR}\] in the residue disc $]\bar b[$, it suffices to show the Zariski density of the complex period map \[j^{hol} \colon N \to \mathcal{G}_{\mathbb{C}}^{dR}/F^0\mathcal{G}_{\mathbb{C}}^{dR}\] in a complex disc containing $b$.
\end{cor}

From here it is likely that one could imitate Brown's proof from \cite[Section 4.4.1]{brown2017class} to prove linear-independence of iterated integrals. Given this linear independence, it is not far to Zariski-density \cite[Theorem 1]{kim2005unipotent}.

\chapter{Applications to Fundamental Groups}

\subsection{Cutting down the Selmer stack in the Gauss--Manin case}

We now come to the definition of $H^{1,type}(G,S)$ that will allow us to demonstrate a small sliver of algebraicity of the Selmer stack; this condition is the \textit{weight-restricted} global Selmer stack. Consider the pushout, then the untwisting map, then the restriction to residue fields, as follows:
$$H^{1, \bold{v}}_{f, \bar D}(G_T, \mathcal{G}^{rel,et}) \to H^{1, \bold{v}}_{f, \bar D}(G_T, S^{et}) \to \mathcal{R}ep^{\bold{v}}_{\bar D}(G_T) \to \prod_{v \not \in T} \mathcal{R}ep_{\bar D}(Gal(\bar \kappa_v/\kappa_v)).$$

Write $\mathcal{R}ep^{w}_d(G_{\kappa_v})$ for the groupoid of $d$-dimensional families $V/Y$ of representations of the absolute Galois group of the residue field of $F_v$ whose Frobenius eigenvalues at closed points of $Y$ are algebraic integers and have complex absolute value $|\kappa_v|^{w/2}$. (These representations are said to have Frobenius weight $w$.) Write $\mathcal{R}ep^{w}_{\bar D}(Gal(\bar \kappa_v/\kappa_v))$ for those with a fixed residual pseudorepresentation.

Similarly, for a fixed $w$ we have the substack $\mathcal{R}ep^{w}_{\bar D}(G_{F,T})$ of $\mathcal{R}ep_{\bar D}(G_{F,T})$. This is the stack of global representations which are pure of weight $w$, i.e. representations all of whose restrictions to decomposition groups outside of $T$ land in the stacks $\mathcal{R}ep^{w}_{\bar D}(Gal(\bar \kappa_v/\kappa_v))$.

\begin{defn}
The \textit{weight-restricted} global Selmer stack with weight $w$, $p$-adic Hodge type $\bold{v}$, and residual pseudorepresentation $\bar D$ is given by $$H^{1, \bold{v}, w}_{f, \bar D}(G_T, \mathcal{G}^{rel,et}) \defeq H^{1, \bold{v}}_{f, \bar D}(G_T, \mathcal{G}^{rel,et}) \times_{\prod_{v \not \in T} \mathcal{R}ep_{\bar D}(Gal(\bar \kappa_v/\kappa_v))} \prod_{v \not \in T} \mathcal{R}ep_{\bar D}(Gal(\bar \kappa_v/\kappa_v)).$$
\end{defn}
In other words, the weight-restricted Selmer stack is the fibered product \[H^{1, \bold{v}}_{f, \bar D}(G_T, \mathcal{G}^{rel,et}) \times_{\mathcal{R}ep_{\bar D}(G_{F,T})}  \mathcal{R}ep^{w}_{\bar D}(G_{F,T}).\]
\begin{rem}
It is not obvious that this should be representable by any sort of geometric space, since the moduli spaces of local weight-restricted representations are simply defined as substacks of the underlying rigid-analytic stacks. Thus, in the next definition, we avoid using weight restrictions on the local stacks because we want to ensure representability.
\end{rem}

We introduce the following notation.
\begin{defn} \label{type}
The type $\bold{r}$ of a $G_{F,T}$-equivariant $\mathcal{G}^{rel,et}$ torsor is the data
\[\bold{r} = (\bold{v}, \{\bar D_i\}, w)\]
of a set of residual pseudorepresentations $\bar D_i$, a $p$-adic Hodge type $\bold{v}$, and a Frobenius weight $w \in \mathbb{Z}$. In this situation we set \[H^1_{f, \bold{r}}(G_{F,T},\mathcal{G}^{rel,et}) = \cup_{i} H^{1,\bold{v},w}_{f,\bar D_i}(G_{F,T}, \mathcal{G}^{\acute{e}t})\] 

The type $\bold{r}$ in the local case, i.e. of a $G_{F_v}$-equivariant $\mathcal{G}^{rel,et}$-torsor for $v|p$, is exactly analogous except we omit the weight $w$. The space \[H^1_{f, \bold{r}}(G_{F_v},\mathcal{G}^{rel,et})\] is defined exactly analogously as a union over all residual pseudorepresentations, omitting any weight restriction.
\end{defn}

Fortunately, we may restrict ourselves to a finite set of residual representations in both the global and local cases. Let us recall an idea of Faltings that underlies his proof of Mordell's conjecture, which one proves using Hermite-Minkowski finiteness.
\begin{prop} \label{semisimp}
Let $F$ be a number field, and $T$ a finite set of places of $F$. There are only finitely many semisimple representations of $G_{F,T}$ that are pure of a fixed weight.
\end{prop}
Now suppose that we are in the following situation:

\begin{center} \label{gm}
(\textbf{The Gauss--Manin case}) $\mathcal{L}^{et, glob}$ comes from the $q$-th relative geometric $p$-adic \'etale cohomology of a family $f \colon Y \to X$ that has a smooth and proper model over $\mathcal{O}_{F, \Sigma}$.
\end{center}

Since the fiber $Y_x$ over any $x \in X(\mathcal{O}_{F, \Sigma})$ is smooth and proper, the associated global Galois representation $\rho_x$ is unramified outside $T = \Sigma \cup \{p\}$ and pure of weight $q$. Thus we can apply the above proposition to see that there are only finitely many options for the semisimplification of $\rho_x$; a fortiori there are only finitely many options for the pseudorepresentation associated to $\rho_x$, and thus only finitely many options for the residual pseudorepresentation associated to $\rho_x$. Call these pseudorepresentations $\bar D_i$.

On the local side, the local Galois group $G_{F_v}$ is topologically finitely generated (Jannsen and Wingberg gave a list of generators and relations \cite{jannsen1982struktur}.) Any $\mathbb{Q}_p$-representation of $G_{F_v}$ has a mod-$p$ residual representation that is unique up to semi-simplification. In any case, regardless of the choice of lattice it has only finitely many residual representations due to finite generation. We denote by $\bar E_i$ the finitely many associated pseudorepresentations. 

Last but not least, by embedding $F_v$ into $\mathbb{C}$ and using the $p$-adic to complex de Rham comparison theorem we see that all of the local representations $(\rho_x)_{G_v}$ ($v |p$, $x \in \mathcal{X}(\mathcal{O}_{F_v})$) have the same $p$-adic Hodge type.

We summarize these observations as follows, assuming as above that we are in the Gauss--Manin case.

\begin{prop} \label{points}
Let $x,b \in \mathcal{X}(\mathcal{O}_{F, \Sigma})$. The global \'etale path torsor for the relative completion $\pi_{b,x}^{rel,\acute{e}t, glob}$ lands in the stack $H^1_{f, \bold{r}}(G_{F,T},\mathcal{G}^{rel,et})$, where $\bold{r}$ has only a finite list of pseudorepresentations, a single weight, and a single $p$-adic Hodge type.

The local \'etale path torsor for the relative completion, $P_{b,x}^{rel,\acute{e}t}$, lands in the stack $$H^1_{f, \bold{r}}(G_{F_v}, \mathcal{G}^{\acute{e}t}),$$ where $\bold{r}$ has only a finite list of pseudorepresentations and a single $p$-adic Hodge type.
\end{prop} 

\subsection{$\mathcal{R}ep^w(G_T)$ is algebraic when $\mathcal{L}$ is geometric}
Proposition \ref{semisimp} has an even stronger consequence, which will allow us to make the first step in passing from the rigid world (where dimension theory is flabby) to the algebraic world (where dimension theory is much more restrictive.) Indeed, in the Gauss--Manin case we will see that the ``bottom level'' of the Selmer stack, Wang--Erickson's representation stack, is an algebraic stack. Recall that a morphism of rigid stacks is said to be algebraizable if it is in the essential image of the map from algebraic stacks over $Spec(\mathbb{Q}_p)$ to rigid stacks over $Spm(\mathbb{Q}_p)$.

We will need Wang-Erickson's hypothesis \textbf{FGAMS}, which is the hypothesis that formal GAGA holds for adequate moduli spaces realized as quotients by $GL_d$; see \cite[Theorem 3.16]{wang2017algebraic} for a discussion of when this hypothesis might hold.

\begin{prop}
\begin{description}
\item[1] For any fixed weight $w$, the weight-restricted representation stack $\mathcal{R}ep^w_{\bar D_i}(G_T)$ is algebraizable.
\item[2] If \textbf{FGAMS} holds, then the fixed weight and crystalline locus $\mathcal{R}ep^{crys, w}_{\bar D_i}(G_T)$ is algebraizable.
\item[3] For any fixed weight $w$, the image of the map $\mathcal{R}ep^w_{\bar D_i}(G_T) \to \mathcal{R}ep_{\bar D_i}(G_{v})$ lies in an algebraizable substack of the target.
\item[4] If \textbf{FGAMS} holds, then for any fixed weight $w$, the image of the map $\mathcal{R}ep^{crys,w}_{\bar D_i}(G_T) \to \mathcal{R}ep^{crys}_{\bar D_i}(G_{v})$ lies in an algebraizable substack of the target.
\end{description}
\end{prop}

\begin{proof}
For \textbf{1}, consider the pseudorepresentation map $$semisimp \colon \mathcal{R}ep^w_{\bar D_i}(G_T) \to PseudoRep_{\bar D_i}(G_T).$$ Let $E$ be an extension of $\mathbb{Q}_p$ over which $\mathcal{R}ep^w_{\bar D_i}(G_T)$ has a point. Now $semisimp$ can take only finitely many values by Proposition \ref{semisimp}. Thus $semisimp \circ \psi$ has finite image $\{\rho_1, \dots, \rho_l \} \in PseudoRep(G_T)(E)$.

Wang--Erickson proves that the map $semisimp$ is algebraizable \cite[Theorem A]{wang2017algebraic}. In particular, $semisimp^{-1}(\{\rho_1, \dots, \rho_l\})$ is the analytification of an algebraic stack of finite type over $E$.

The argument is similar in the local cases \textbf{3} and \textbf{4}. Namely, the restrictions $(\rho_i)_{G_{F_v}}$ may no longer be semisimple, but their semisimplifications provide a finite list $\rho'_i \defeq \left((\rho_i)_{G_v}\right)^{ss}$ of semisimplifications of the local representations associated to global points. Now Wang-Erickson proves algebraicity of 
\[semisimp \colon \mathcal{R}ep(G_{F_v}) \to PseudoRep(G_{F_v}) \]
unconditionally, and the algebraicity of 
\[semisimp \colon \mathcal{R}ep^{crys, \bold{v}}(G_{F_v}) \to PseudoRep^{crys, \bold{v}}(G_{F_v}) \]
assuming FGAMS \cite[Corollary 6.8]{wang2017algebraic}.

These imply \textbf{2}, because the global crystalline representation stack is the fibered product of the global stack with the local crystalline stack.

\end{proof}

The main upshot of this algebraicity is that $X(\mathcal{O}_{F, \Sigma})$ has the possibility of landing in an algebraizable substack of $H^1_{f, \bold{r}}(G_T, \mathcal{G}^{rel,et})$.

Based on the algebraicity of the global fixed-weight Selmer stack, we ask the following.
\begin{quest}
In the Gauss--Manin case and assuming \textbf{FGAMS}, is the weight-restricted Selmer stack $H^{1}_{f, \bold{r}}(G_T, \mathcal{G}^{rel,et})$ strat-representable in the category of algebraic stacks?
\end{quest}

\section{Bloch--Kato logarithm}
The only part of the Chabauty--Kim diagram of which we have said nothing is the Bloch--Kato logarithm. For a linear crystalline representation $V$, the logarithm is a map $$H^1_f(G_K, V) \to D_{crys}(V)/F^0 D_{crys}(V).$$ In Kim's unipotent formulation, the logarithm is promoted to an algebraic map of pro-varieties $$H_f^1(G_K, \pi_1^{un, \acute{e}t}(\bar X, b)) \to \pi_1^{un,dR}(X, b)/F^0\pi_1^{un,dR}(X, b),$$ given by sending a torsor $P$ to the torsor $D_{dR}(P) = D_{crys}(P) \otimes_{K_0} K$, obtained by applying Fontaine's de Rham or crystalline functors to the coordinate ring of $P$.

We note that Sakugawa develops a theory of the Bloch--Kato logarithm for non-unipotent groups in \cite{sakugawa2017non}. Some remarks are in order which we hope will show the advantage of the stacky approach to fundamental groups. First of all, Sakugawa focuses on Selmer sets, i.e. pointed sets of non-abelian cocycles. Second, Sakugawa needs to place a number of conditions on his non-abelian Galois representations in order to construct his logarithm for non-unipotent groups because he stays in the land of sets.

The first condition is that the Lang map be a bijection on the de Rham realization, and as we have shown, this holds on our $\mathcal{U}$ for weight reasons. However, it may not hold on all of $\mathcal{G}$: Since the Frobenius action on $S$ is by conjugation, we see that the non-injectivity of the Lang map is precisely measured by the centralizer of Frobenius acting on $S$. We are reminded that Lawrence and Venkatesh spend much of their effort bounding this Frobenius centralizer, and so it is not likely that we would be able to apply Sakugawa's result in the non-unipotent contexts we care about.

Second, Sakugawa imposes the condition that the Galois cohomology of the de Rham realization of $S$ be trivial. We need to do the same, but it is possible (see Remark \ref{etaleremark}) that this restriction can be lifted.

\subsection{The stacky logarithm and the Chabauty--Kim diagram} \label{diagram}
Now if $$P \in H^{1}_{f,\bold{r}}(G_{F_v}, \pi_1^{rel,\acute{e}t})(\Lambda)$$ 
is a torsor, we would like to set $$D_{crys}(P) \defeq \Spec \left( \mathcal{O}(P) \otimes_{\Lambda} ( \Lambda \widehat{\otimes}  B_{crys}(F_v)) \right)^{G_{F_v}};$$ but we have to be careful: $P$ is not affinoid, as it is a torsor for the non-affinoid $(\pi_1^{rel \acute{e}t})^{an}$.

Thus we must patch. We observe that if a family $V/ \Lambda$ is crystalline then $D_{crys}(V)$ commutes with arbitrary base change. In particular, we have that 
\begin{align} \label{restriction}
D_{crys}(V)|_{U} = D_{crys}(V|_{U}),
\end{align} where $U$ is an affinoid open subset of $X$.

\begin{defn}
Let $P$ be as above. Then $D_{crys}(P)$ is defined as follows: Locally for the Tate topology, $P$ corresponds to a torsor $P^{alg}$ which has the property that its coordinate ring is an ind-crystalline family of Galois representations over $\Lambda'$ ($\Spm \Lambda'$ an admissible subset of $\Spm \Lambda$.) This is by Remark \ref{selmerdescription}. Now we let the restriction of $D_{crys}(P)$ to $\Spm \Lambda'$ be $D_{crys}(P^{alg})$. Finally, the gluing data for the various $P^{alg}$ induces gluing data on $D_{crys}(P^{alg})^{an}$ by (\ref{restriction}), and we write $D_{crys}(P)$ for the glued torsor. We define $D_{dR}$ similarly.
\end{defn}
We have that $D_{crys}(P)$ is a $\pi_1^{rel, crys}$-torsor over $\Lambda \otimes K_0$, by Olsson's comparison theorem, as we explain below. This follows locally in the Tate topology by applying $D_{crys}$ to the coaction isomorphism \[\mathcal{O}(P) \otimes_{\Lambda} ( \Lambda \otimes_{\mathbb{Q}_p} \mathcal{O}(\pi_1^{rel, \acute{e}t}) ) \to \mathcal{O}(P) \otimes_{\Lambda} \mathcal{O}(P).\] Indeed, $D_{crys}$ preserves tensor products, is functorial, and \[D_{crys}( \Lambda \otimes_{\mathbb{Q}_p} \mathcal{O}(\pi_1^{rel, \acute{e}t})) \simeq \mathcal{O}(\pi_1^{rel, crys}) \otimes_{\mathbb{Q}_p} \Lambda.\]

The coordinate ring of $D_{crys}(P) \otimes_{K_0} K (= D_{dR}(P))$ is locally given the usual ``extrinsic'' Hodge filtration: the subspace filtration induced from $\mathcal{O}(P) \otimes_{\Lambda} (  \Lambda \widehat{\otimes} B_{dR})$, with the trivial filtration on $\mathcal{O}(P)$ and the canonical filtration on $B_{dR}$. Furthermore, it has a Frobenius action coming from the Frobenius on $B_{crys}.$ 
\begin{prop} \label{logisadm}
Suppose that both $S^{dR}$ and $(S^{dR})^{\phi=1}$ are semisimple and simply connected. We have that $D_{dR}(P)$ is a doubly admissible $\pi_1^{rel,dR}$-torsor.
\end{prop}
\begin{proof}

This proof is basically a repeat of the argument that $P^{dR}_{b,z}$ is doubly-admissible, except with torsors over non-constant rigid bases.

The Frobenius reduction of structure is obvious Tate-locally from the compatibility of the torsor map with Frobenius, and the Hodge reduction of structure comes from the fact that the action of $\pi_1^{rel,dR} = D_{dR}(\pi_1^{rel,et})$ on the local algebraic torsors $D_{dR}(P^{alg})$ is strict for the Hodge filtration because this filtration is induced by the filtration on $B_{dR}$, cf. Remark \ref{strictmors} and Proposition \ref{admissible}.

We are ready to consider weak admissibility of the pushout $P^{dR}_{S}$. We must check this on closed points of $\Spm \Lambda$, so we assume that $\Lambda$ is a finite extension of $\mathbb{Q}_p$, and all of our calculations are calculations with algebraic torsors. For each trivialization $p \in P^{dR}(\Lambda)$ we have an isomorphism $P^{dR} \overset{g} \simeq \pi_1^{rel,dR}$. Furthermore, $p \in F^0P^{dR}$ if and only if the map $g$ is filtration-preserving. This follows because $g$ is defined by the tensor product
\[
\begin{tikzcd}
& \Lambda \otimes \mathcal{O}(\pi_1)  \\
\mathcal{O}(P^{dR}) \otimes \mathcal{O}(P^{dR})  & \mathcal{O}(P^{dR}) \otimes \mathcal{O}(\pi^{rel, dR}_1) \arrow{l}{\simeq} \arrow{u}{p \otimes id} 
\end{tikzcd}
\]
and because the map vertical map is filtration-preserving if and only if $p$ vanishes on $F^1\mathcal{O}(P)$, by strictness of the coaction as in Proposition \ref{parabolic}.

Let $p^{cr}$, $p^{H}$ be trivializations of the reductions of structure $(P^{dR}_S)^{\phi=1}$ and $F^0P^{dR}_S$ over $\Lambda$. These exist by the same argument with Kneser's theorem. Now $p^{cr}$ and $p^{H}$ define isomorphisms \[f^{cr}, f^H \colon S^{dR}_{\Lambda} \to (P^{dR}_S)_{\Lambda},\] with the property that $f_{cr}$ is Frobenius-equivariant on coordinate rings and $f^H$ is filtration-preserving on coordinate rings.

This produces a chain of isomorphisms \[S_{dR} \overset{f^{H}} \to P_S \overset{(f^{cr})^{-1}} \to S^{dR},\] whence an isomorphism $ \iota \colon \mathcal{O}(S^{dR}) \to \mathcal{O}(S^{dR})$. Let $\mathcal{F}_s$ denote the image of the filtration on $\mathcal{O}(S^{dR})$ via $\iota$. Then we must show that the filtration from $\mathcal{F}_s$ on $Lie(S^{dR})$ is weakly admissible. But for this it is sufficient to show that $\mathcal{F}_s$ is weakly admissible on $\mathcal{O}(S^{dR})$ itself.

But this follows from the properties of $f^{H}$ and $f^{crys}$; since it is filtration-preserving and the torsor action is strict, the pushforward of the canonical filtration along $f^H$ gives the canonical filtration on $\mathcal{O}(P_S^{dR})$. Then we have a Frobenius-equivariant isomorphism $\mathcal{O}(P^{dR}_S) \to \mathcal{O}(S^{dR})$ from a weakly admissible $\phi$-module to another $\phi$-module; thus the image filtration is weakly admissible. We conclude that $\mathcal{F}_sLie(S)$ is weakly admissible, and we are done.

\end{proof}
Using the compatibility of Olsson/Pridham's comparison theorem with filtrations, we thus have that the following diagram commutes.
\[
\begin{tikzcd}
\mathcal{X}(\mathcal{O}_{F, \Sigma}) \arrow{r} \arrow{d}{x \mapsto P^{\acute{e}t}_{b,x}} & \mathcal{X}(F_v) \arrow{d}{x \mapsto P^{\acute{e}t}_{b,x}} \arrow{dr}{x \mapsto P^{dR}_{b,x}} \\
\bigcup_{i} H^{1,\bold{v}}_{f,\bar D_i}(G_{F,T}, \mathcal{G}_n^{\acute{e}t}) \arrow{r}{res} & \bigcup_{i} H^{1,\bold{v}}_{f,\bar E_i}(G_{F_v}, \mathcal{G}_n^{\acute{e}t}) \arrow{r}{D_{dR}} & Res^{F_v}_{\mathbb{Q}_p} \mathcal{M}_{dR, n}
\end{tikzcd}
\]

The focus of the entire theory is on the following functor.

\begin{defn}
The morphism $$log_{p,n} = D_{dR} \circ res \colon H^1_{f, \bold{r}}(G_T, \mathcal{G}_n^{rel,et}) \to Res^{F_v}_{\mathbb{Q}_p}  \mathcal{M}^{dR} \subseteq Res^{F_v}_{\mathbb{Q}_p}  [(\mathcal{G}_n^{rel, dR})^{\phi = 1} \backslash \mathcal{G}_n^{rel, dR}/F^0\mathcal{G}_n^{rel, dR}],$$ is called the Bloch-Kato logarithm.
\end{defn}

\subsection{Dimension theory for algebraic stacks}
Essential for the classic method of Chabauty--Kim is the statement that if $f \colon X \to Y$ is a morphism of schemes with $dim(X) < dim(Y)$, then $f(X)$ lies in a closed proper subscheme of $Y$. We now record the generalization of this statement to the world of algebraic stacks; if the global Selmer stack is algebraizable, then the map $$log_p \colon H^1_{f, \bold{r}}(G_T, \pi_1^{rel, \acute{e}t}) \to Res^{F_v}_{\mathbb{Q}_p}[(\pi_1^{dR})^{\phi=1} \backslash \pi_1^{dR}  / F^0 \pi_1^{dR}]$$ is obviously a map of algebraic stacks because it is functorial on the category of $\mathbb{Q}_p$-algebras. The material in this section comes from \cite[\href{https://stacks.math.columbia.edu/tag/0DRE}{Tag 0DRE}]{stacks-project}, which was originally a preprint of Emerton and Gee.

Recall that the topological space $|\mathcal{X}|$ has as an underlying set the equivalence classes of morphisms from spectra of fields to $\mathcal{X}$. A presentation of an algebraic stack is a groupoid $(U,R)$ in algebraic spaces over $S$, and an equivalence $f \colon [U/R] \to \mathcal{X}$. We denote by $e \colon U \to R$ the identity map and $s \colon R \to U$ the source map, and by $R_u$ the fiber of $s$ over a point $u \in U$.
\begin{defn}{\cite[\href{https://stacks.math.columbia.edu/tag/0AFN}{Tag 0AFN}]{stacks-project}}

Let $\mathcal{X}$ be a locally Noetherian algebraic stack over a scheme $S$. Let $x \in |\mathcal{X}|$ be a point of $\mathcal{X}$. Let $[U/R] \to \mathcal{X}$ be a presentation where $U$ is a scheme and let $u \in U$ be a point that maps to $x$. We define the dimension of $\mathcal{X}$ at $x$ to be the element $\dim_x\mathcal{X} \in \mathbb{Z} \cup \infty$ such that $$\dim_x\mathcal{X} =  \dim_uU - \dim_{e(u)}R_u $$
\end{defn}
The dimension is independent of the presentation and lift of $x$ that are chosen \cite[\href{https://stacks.math.columbia.edu/tag/0AFM}{Tag 0AFM}]{stacks-project}.

Kim's work uses the following basic statement:
\begin{prop} \label{schemedim}
Let $f \colon X \to Y$ be a map of schemes of finite type over a field $k$ with $\dim_k(X) < \dim_k(Y)$. Then $Z$, the smallest closed subscheme of $Y$ containing $f(X)$, is strictly contained within $Y$.
\end{prop}
\begin{proof}
It is obvious that we may reduce to the case where $X$ and $Y$ are irreducible and affine. Furthermore, since $f$ is necessarily quasi-compact, the image of the reduced morphism \[f_{red} \colon X_{red} \to Y_{red}\] is the reduced subscheme underlying the scheme-theoretic image $\im(f)$ \cite[\href{https://stacks.math.columbia.edu/tag/056B}{Tag 056B}]{stacks-project}. In any case, we see that it suffices to assume that $X = \Spec(A)$ and $Y = \Spec(B)$ are reduced, irreducible, and affine.

In such a case we have that the dimension of $X$ is given by the transcendence degree of $k(X)/k$, and the same is true of $Y$ \cite[\href{https://stacks.math.columbia.edu/tag/00P0}{Tag 00P0}]{stacks-project}. Suppose the image of $f$ were dense; then we would have an injection of coordinate rings $B \hookrightarrow A$ \cite[\href{https://stacks.math.columbia.edu/tag/0CC1}{Tag 0CC1}]{stacks-project}.

This injection leads to an inclusion of fields $k(Y) \hookrightarrow k(X)$. But transcendence degree can only increase in extensions, so that the dimension of $X$ must be at least the dimension of $Y$, a contradiction.
\end{proof}

We need a stacky version. We restrict to the case of atlases whose covering morphisms are representable by schemes because it is all we'll need.
\begin{prop}
Let $f \colon \mathcal{X} \to \mathcal{Y}$ be a map of stacks of finite type over a field $k$ with $\dim_k(\mathcal{X}) < \dim_k(\mathcal{Y})$. Suppose that the atlas covering maps for $\mathcal{X}$ and $\mathcal{Y}$ are representable by schemes. Then $\mathcal{Z}$, the smallest closed substack of $\mathcal{Y}$ through which $f$ factors, is strictly contained within $\mathcal{Y}$.
\end{prop}
\begin{proof}
By \cite[\href{https://stacks.math.columbia.edu/tag/0CMK}{Tag 0CMK}]{stacks-project}, the quasi-compactness of $f$ implies that formation of scheme-theoretic image commutes with flat base change. Let $U$ be an atlas of $\mathcal{Y}$, so that $(U, U \times_{\mathcal{Y}} U)$ is a presentation of $\mathcal{Y}$; then the base-change of $f$ along the morphism $p \colon U \to \mathcal{Y}$ yields a map of schemes $f' \colon W \to U$. We have that $(W, W \times_{\mathcal{X}} W)$ is a presentation of $X$. Now the dimension hypothesis can be rewritten as $$\dim W - \dim(W \times_{\mathcal{X}} W)_{w} < \dim U - \dim(U \times_{\mathcal{Y}} U)_w$$ for all $w \in W, u \in U$. Here the subscripts refer to the fiber of the projection onto the first factor.

Finally, $$\dim(W \times_{\mathcal{X}} W)_{w} = \dim(U \times_{\mathcal{Y}} U)_w$$ by \cite[\href{https://stacks.math.columbia.edu/tag/0DRN}{Tag 0DRN}]{stacks-project}, since we have a Cartesian diagram
\[
\begin{tikzcd}
 W \times_{\mathcal{X}} W \arrow{r} \arrow{d} & U \times_{\mathcal{Y}} U \arrow{d} \\
 W \arrow{r} & U
\end{tikzcd}
\]
where the vertical maps are projections onto the first factor (i.e. the source map).

Thus $\dim W < \dim U$, and by Proposition \ref{schemedim} the scheme-theoretic image of $W$ is strictly contained in $U$. But then, by compatibility with base-change, the scheme-theoretic image of $\mathcal{X}$ in $\mathcal{Y}$ could not be the entirety of $\mathcal{Y}$.
\end{proof}

\section{An example: the projective line minus three points} \label{projline}

We now apply the relative Chabauty--Kim philosophy to determine integral points in the classical case of $X = Y_1(4)_{\mathbb{Z}[1/4]}$, with $S$ some finite set of primes containing 2 and using a fixed auxilliary prime $p \not \in S$. Here $Y_1(4)$ denotes the moduli space of elliptic curves with a marked 4-torsion point. Our proof of finiteness is conditional on a version of the Fontaine--Mazur conjecture and the algebraicity of a certain Selmer stack, however.

We hope the reader will bear with the absurdity of assuming such a strong conjecture and an unknown algebraicity problem for a result such as Siegel's theorem on integral points on affine curves for some very special affine curves; Siegel's theorem has many elementary proofs, but we the reader might appreciate that our method implies that combinations of integrals of Eisenstein series ought to vanish on $X(\mathbb{Z}[1/S])$. In short, this section should be viewed as an exposition of how one would apply Chabauty--Kim for the relative completion if one knew that all relative $p$-adic Hodge theory considerations behaved nicely and all motivic conjectures were true.

Now $X$ is isomorphic over $\mathbb{Z}[1/4]$ to $\mathbb{P}^1 \setminus \{0,1,\infty\}_{\mathbb{Z}[1/4]}$. Over $\mathbb{P}^1 \setminus \{0,1,\infty\}_{\mathbb{Z}[1/4]}$ one has the universal family of elliptic curves corresponding to its interpretation as a modular curve: this family is given by the normal form \[y^2 = \left(x+\frac{t}{4}\right)\left(x^2 + x + \frac{t}{4}\right).\]

\begin{rem}
Deligne realized that this family is 2-isogenous to the (shifted) Legendre family \[y^2 = x(x-1)(x + t -1).\] See \cite[Section 2]{katz2010legendre} for all of the above statements. Obviously, the automorphism $x \mapsto 1-x$ of $\mathbb{P}^1 \setminus \{0,1,\infty\}_{\mathbb{Z}[1/4]}$ shows that we may as well use the Legendre family $y^2 = x(x-1)(x - \lambda)$ itself. Whether we choose to apply the relative Chabauty--Kim method to the first family above or the Legendre family does not matter to us. This is because isogenies induce isomorphisms on \'etale, crystalline, and de Rham cohomology of elliptic curves as long as 2 does not divide the characteristic of any fields we work with, and so the families of Galois representations, isocrystals, and vector bundles we use are isomorphic.
\end{rem}

Truthfully, the computations that follow are not particular at all to $Y_1(4)$ and $X$ can be any modular curve for which Eisenstein classes have been exhibited and shown to be non-trivial. This includes at the very least $Y_1(N)$ for $N \geq 4$ \cite[Section 3.1]{kings2017rankin}. For all such curves and universal families, the method of Lawrence and Venkatesh fails because the de Rham moduli space collapses \cite[Section 1.3(a)]{lawrence2018diophantine}. Let us show how our method remedies this problem.

Taking the local system or vector bundle $\mathbb{H}$ induced by the first cohomology of the universal family in each realization, we have fundamental groups $\pi_1^{rel}(X,b)$ in each realization associated to any basepoint $b$. We consider the abelianization $A$ of the unipotent part of $\pi_1^{rel}(X,b)$ in addition to the reductive part, so that we are left with an extension of the form $$1 \to A \to \pi' \to SL_2 \to 1.$$ As detailed in \cite[Remark 13.3]{hain2016hodge}, for any $0 \leq n_0 \leq N_0$ we have that $$\prod_{n_0 \leq j \leq N_0}  Sym^{2j}(H^1(E))(2n+1) = \prod_{n_0 \leq j \leq N_0}  Sym^{2j}(H_1(E))(1)$$ is a quotient of $A$ in the de Rham and Betti realizations. Here $E$ is the elliptic curve above $b$. Write $Q$ for this quotient.

The same quotient exists in the \'etale and crystalline realizations. The idea is that non-trivial Eisenstein classes have been given in every realization, so that $$\cdot \prod_n Sym^{2j} (H_1(E))(1)$$ will always be a quotient of $A$. The existence of the \'etale quotient is instructive, and is shown for the modular curve $\mathcal{M}_{1,1}$ in \cite[Proposition 20.1]{hain2018universal}. For higher levels and all realizations, see \cite[Section 3.1]{kings2017rankin}.

It is a general fact that the Chabauty-Kim diagram may be applied to any motivic quotient of the unipotent fundamental group by pushing out all torsors to the quotient in question; the proofs apply verbatim to the case of the pushout above. We may thus apply our relative Chabauty--Kim diagram to the quotient $\pi$, which is the relative completion pushed out to the quotient: $$1 \to Q \to \pi \to SL_2 \to 1$$ for each $n$. We will show how to choose a $n_0, N_0$ such that the dimension hypothesis holds for $\pi$.

Now let us analyze the spaces $H^1_f(G_T, \pi)$ and $[\pi^{\phi=1} \backslash \pi/F^0]$. This latter de Rham space is easiest to study. The flag variety $SL_2 / F^0$ has dimension one. Thus taking the stack-theoretic quotient by the centralizer of Frobenius (which could be the whole of $SL_2$) gives a space of dimension at least -2, (speaking stack-theoretically).

Moving up the tower and analyzing the Hodge filtration on $A$, we see that $F^0Q$ is trivial. Now \[\dim Q = \sum_{j = n_0}^{ N_0} 2j+1 = \frac{1}{2} (N_0 - n_0 +1) (2N_0 + 1 + 2n_0 + 1) = N_0^2 - n_0^2 + 2N_0 +1.\] Therefore $[\pi^{\phi=1} \backslash \pi/F^0]$ has dimension at least $\dim Q - 2 =  N_0^2 - n_0^2 + 2N_0 -1$.

Let us analyze the \'etale side. We make the following assumptions.
\begin{assumption}
The stack $H^1_{f, \bold{r}}(G_T, \pi)$ is strat-representable in the category of algebraic stacks, where $\bold{r}$ is the type associated to the finite list of semisimple global representations of elliptic curves over $\mathbb{Z}[1/S]$ and $w=1$. Furthermore, the $p$-adic relative unipotent Albanese map $j_1$ (i.e. at the level of $\mathcal{U}^{ab}$) is dense.
\end{assumption}
\begin{rem}
Strat-representability of the global Selmer stack is not a problem for our Chabauty arguments, because rational points land in one of the substacks stratifying $H^1_{f, \bold{r}}(G_T, \pi)$, and since there are finitely many, we can just make the density argument on each one. This results in finitely many $p$-adic analytic functions that vanish on rational points.
\end{rem}

Although one can bound $H^1_f(G_T, SL_2)$ using the $R=T$ theorems of Taylor and Wiles, we are conditioning our proof on so much that we might as well treat it as a black box and set $r = \dim H^1_f(G_T, SL_2)$.

Let us move up the tower on the \'etale side. Now $H^1_f(G_T, \pi)$ fibers over $H^1_f(G_T, SL_2)$ with fiber over a cocyle $c \in H^1_f(G_T, SL_2)(\bar{\mathbb{Q}}_p)$ equal to a subquotient of $H^1(G_T, _cQ)$, where $_cQ$ is the twist as defined in our discussion on stacky cohomology. It suffices to bound the dimension of each of these twists. 

 Denote by $\phi$ the Galois representation $H^1_{\acute{e}t}(\bar E_b, \mathbb{Q}_p)$. Then the fact that $SL_2$ acts on $\phi$ by right-multiplication gives $$_c\phi^\vee(1) = \phi^\vee c^\vee(1) = (c\phi)^\vee(1) $$ for the twist of the Tate module; more generally, we have $$_cQ = \prod_{n_0 \leq j \leq N_0} Sym^{2j}((c\phi)^\vee)(1).$$ 

To estimate the dimensions of these twists, we note that, by the main theorem of \cite{kisin2009fontaine} which proves Fontaine--Mazur in dimension two, the representations $c \phi$ are motivic. (This is because we have put unramifiedness and crystallinity conditions on $c\phi$, and restricted the set of pseudorepresentations in the stack $H^1_{f, \bold{r}}(G_T, SL_2)$. The conditions that Kisin requires on the reduction of $c\phi$ are satisfied if we let our set of residual pseudorepresentations $\{\bar D_i\}$ be the smallest set consisting of only those coming from $H^1$ of the curves in the universal family, because all such residual representations satisfy Kisin's lifting conditions.)

The Euler characteristic formula \cite[Chapter 1, Section 5]{milne1986arithmetic} gives $$\dim H^1(G_T, _c Q) = \dim H^2(G_T, _cQ) + \dim _cQ^{-},$$ where the negative denotes the negative eigenspace of complex conjugation. This term has dimension \[\sum_{j= n_0}^{N_0} j = 1/2(N_0^2 - n_0^2 + N_0 + n_0).\] Indeed, comparison with Hodge theory shows that complex conjugation acts by $-1$ on a one-dimensional subspace of $(c\phi)^\vee(1)$, and so it acts trivially on a $j$-dimensional subspace of the relevant symmetric power. This term dominates the \'etale side: it is approximately half the dimension of the corresponding piece of the de Rham side.

Since $G_T$ has cohomological dimension two, it suffices to bound $H^2(G_T, (c\phi)^{- \otimes 2j}(1))$ for each $j$ from $n_0$ to $N_0$, where a negative tensor power indicates a dual, as usual. We consider the localization map $$H^2(G_T, (c\phi)^{-\otimes 2j}(1)) \to \prod_{v \in T} H^2(G_v, (c\phi)^{-\otimes 2j}(1)).$$

\begin{prop}{\cite[Observation 2]{kim2005unipotent}}

Conditional on the Fontaine--Mazur conjecture for Galois representations of dimension $2j+1$, this localization has trivial kernel.
\end{prop}
\begin{proof}
By Poitou--Tate duality, this kernel is isomorphic to the kernel of the localization $$H^1(G_T, (c\phi)^{\otimes 2j}) \to \prod_{v \in T} H^1(G_v, (c\phi)^{\otimes 2j}),$$ which is denoted $Sha^0_T((c\phi)^{\otimes 2j})$. Any element of $H^1(G_T, (c\phi)^{\otimes 2j})$ can be viewed as an extension $$0 \to (c\phi)^{\otimes 2j} \to E \to K' \to 1,$$ where $K'/\mathbb{Q}_p$ is the field of definition of $c\phi$, viewed as a trivial representation. The fact that such an extension lies in $Sha^0_T((c\phi)^{\otimes 2j})$ implies that this extension is crystalline at $p$, since the localized extension at $p$ is trivial and the direct sum of two crystalline representations is crystalline.

By the Fontaine--Mazur conjecture, a representation unramified almost everywhere and crystalline at $p$ is motivic, and in particular is endowed with a weight filtration such that the above exact sequence is strict for the weight filtration. An argument with this filtration finishes the proof, since $(c\phi)^{\otimes 2j}$ is pure of weight $2j$, while $K'$ has weight zero. The strictness of the map $E \to K'$ implies that there is some element $x \in W_0 E$ which maps to $1 \in K'$; now by weight considerations $W_0E \cap (c\phi)^{\otimes 2j}$ is trivial when $j \geq 1$, so that $E \simeq W_0E \oplus (c\phi)^{\otimes 2j}$ and we have shown that $E$ splits.
\end{proof}

We proceed to compute the dimension of $\prod_{v \in T} H^2(G_v, (c\phi)^{-\otimes 2j}(1)).$ By local duality, this is the same as $\prod_{v \in T} H^0(G_v, (c\phi)^{\otimes 2j}).$ When $v=p$, the Hodge--Tate decomposition gives $$H^0(G_p, (c\phi)^{\otimes 2j}) \hookrightarrow H^0(G_p, (c\phi)^{\otimes 2j} \otimes \mathbb{C}_p) \simeq H^0(G_p, [Z_1 \otimes \mathbb{C}_p(-1) \oplus Z_2 \otimes \mathbb{C}_p]^{\otimes 2j}),$$ where $Z_1$ and $Z_2$ have dimension one. The last space in this sequence has dimension one, corresponding to the term $H^0(G_p, Z_2^{\otimes 2j} \otimes \mathbb{C}_p)$.

When $v \neq p$, the weight filtration on $c\phi$ has one-dimensional associated gradeds with weights zero and two, or a two-dimensional weight 1 eigenspace. (See, for example, \cite[Theorem 2]{saito2009hilbert}.)
 Then $H^0(G_v, (c\phi)^{\otimes 2j})$ again has dimension at most one. We conclude that $$\dim H^1(G_T, (c\phi)^{- \otimes 2j}(1)) \leq 1/2(N_0^2 - n_0^2 + N_0 + n_0) + \#T,$$ and the dimension hypothesis holds so long as $r + 1/2(N_0^2 - n_0^2 + N_0 + n_0) + \#T < N_0^2 - n_0^2 + 2N_0 - 1.$ Comparing leading terms shows that $N_0$ can be chosen large enough to ensure this inequality. This concludes the proof.

It would be particularly interesting to find the explicit iterated integrals that vanish on global points. These integrals should be related to elliptic polylogarithms.

\appendix
\chapter{Tannakian Categories}

\begin{defn}
For a topological group $\Gamma$ and a topological field $K$, we will denote by $Repf_K(\Gamma)$ the category of continuous finite-dimensional representations of $\Gamma$. If $\Gamma$ is an algebraic group over a field $F$, $Repf_K(\Gamma)$ means finite dimensional algebraic $K$-representations of $\Gamma$ in the sense of algebraic groups.
\end{defn}
The theory of \textit{neutral Tannakian categories} is the characterization of categories that are equivalent to one of the form $Repf_K(\Gamma)$ for $\Gamma$ a (pro-)algebraic group. Such categories are abundant in mathematics, and their study is rich precisely because they bring to bear all of the tools of representation theory. We do not aim here to give a full account of Tannakian categories; some nice references are (\cite{szamuely2009galois}, \cite{deligne1982tannakian}). Instead we recall one result that tends to appear in a ``second reading" for many students and is essential for our understanding of the fundamental group.

\section{Interpretation via Zariski closures and universality} \label{envelopes}
The Zariski closure of a representation arises throughout our work, particularly in the example of the geometric monodromy group. There is, indeed, a Tannakian interpretation to the Zariski closure.
\begin{prop}
Let $\rho \colon \Gamma \to V$ be a finite-dimensional representation of a group over a field $K$. Consider the smallest Tannakian subcategory of $Repf_K(\Gamma)$ containing $V$, and denote this subcategory by $\langle V \rangle_\otimes$. Then the Tannakian fundamental group of $\langle V \rangle_\otimes$ is canonically isomorphic to the Zariski closure of the image of $\rho$ in $GL_{n,K}$.
\end{prop}

This interpretation via Zariski closures plays a second role, namely it induces the universality of the Tannakian fundamental group as follows. Let $\Pi$ be the Tannakian fundamental group of $Repf_K(\Gamma)$. Consider the collection $\mathcal{I}$ of pairs $(G, \phi)$ where $G$ is an algebraic group over $K$ and $\phi \colon \Gamma \to G(K)$ is a continuous and Zariski-dense homomorphism. Now $\mathcal{I}$ naturally forms an inverse system, and the assertion is that $\Pi$ is the inverse limit of $\mathcal{I}$. Thus it satisfies the universal property of the inverse limit. All of this is stated and proven in an open neighborhood of \cite[6.5.17]{szamuely2009galois}, except for the continuity proviso. This last add-on is a straightforward application of the aforementioned theorem, along with the fact that $\langle V \rangle_\otimes$ contains only continuous representations, and that $Repf_K(\Gamma)$ is the direct limit of the full subcategories $\langle V \rangle_\otimes$.

\chapter{Relevant Concepts in Arithmetic Geometry: \'Etale Local Systems, Gauss--Manin, and Chen's $\pi_1^{dR}$ Theorem}
We take this section to introduce the Tannakian categories that feature in the Chabauty--Coleman--Kim process. An expert can safely skip this introduction.

\section{Recollections on \'etale maps and local systems}
The star of twentieth-century arithmetic geometry is the \'etale topology. The basic theory of \'etale maps and the \'etale topos can be found in the classic book of Freitag and Kiehl \cite{freitag2013etale}.
We take this section to recall some of the properties of \'etale maps that we will need for our later work. Of the two properties we emphasize, the first will help us interpret the Tannakian fundamental group, and the second will play into our main example in formal geometry. After reviewing these two properties we show that the Tannakian \'etale fundamental group is the unipotent completion of the profinite fundamental group.

\begin{center}
    \begin{tabular}{ | p{6cm} | p{8cm} |}
    \hline
    \textbf{Property} & \textbf{Enables us to...} \\ \hline
    Homotopy Lifting & Translate from deck transformations to automorphisms of fiber to automorphisms of fiber functor \\ \hline
    
     Remarkable Equivalence & Lift objects from finite fields to local rings \\ \hline

    \hline
    \end{tabular}
\end{center}

\bigbreak

The first such property is the ``homotopy lifting property" for \'etale maps, the proof of which can be found in \cite[Corollary 5.3.3]{szamuely2009galois}.
\begin{lem}
Suppose $U,V$, and $Z$ are schemes, with $Z$ connected. Let $p \colon U \to V$ be an \'etale map, and $\iota \colon Z \to V$ any morphism. Then a lift of $\iota$ to $U$ is determined by the lift of any geometric point.
\end{lem}
\begin{rem}
We call this the lifting property not just because it involves lifting, but because it emulates the lemma of the same name for covering spaces in topology, in which the lift of a curve on a space is determined by the lift of its starting-point.
\end{rem}
When $U \to V$ is finite \'etale, $U$ is connected, and $\Aut(U/V)$ acts transitively on the fiber above a geometric point, we call $U/V$ Galois. As far as the Tannakian point of view is concerned, the most important case of the lifting lemma is when $Z=U$ in the above and $U/V$ is Galois. In that case, the lemma non-canonically identifies $Aut(V/U)$ with the fiber above any geometric point.
\begin{defn}
We define the profinite \'etale fundamental group of a scheme $X$ by $$\pi_1^{alg}(X) = \lim_Y \Aut(Y/X),$$ where $Y$ runs over all Galois covers of $X$. (We will reserve the phrase \'etale fundamental group and the notation $\pi_1^{\acute{e}t}$ for the Tannakian counterpart.)
\end{defn}

The other amazing property of \'etale maps is another form of lifting statement, known as ``Une equivalence remarquable de catégories'' -- let us remark upon it, then.
\begin{lem}
Let $S$ be a scheme, and $S_0 \subseteq S$ a closed subscheme with the same underlying topological space. Then there is an equivalence of categories induced by pullback to $S_0$, given as follows: $$\left\{ U \overset{\acute{e}tale} \to S \right\} \overset{\sim} \to \left\{ V \overset{\acute{e}tale} \to S_0 \right\}.$$ 
\end{lem}

Again, there are too many complete references for this fact (the original being \cite[IV, 18.1.2 ]{dieudonne1971elements}), and so we simply highlight the ``kernels of algebra" which grow into the proof.
\begin{pfsketch}
Regarding the existence of a lifting of an \'etale map $$U \to S_0,$$ one first studies the affine case. Indeed, the fact that $S$ and $S_0$ have the same topological space means that Zariski-locally the morphism $S_0 \to S$ looks like $\Spec R \to \Spec R/I,$ where $I$ is locally nilpotent (i.e. every element of $I$ has some power that is zero.) Likewise the local structure of \'etale maps says that Zariski-locally on the target we have an isomorphism $$U \simeq \Spec (R/I)[x_1, \ldots, x_r]/(f_1, \ldots, f_r)$$ where the Jacobian of the set of $f_i$s is invertible in $R/I$. Now our local lift of $U$ is just given by lifting these $f_i$ to elements of $R$.

A tiny bit of commutative algebra shows that an element $a \in R$ is a unit if and only if it is a unit in $R/I$. The ``only if" direction is easy. Assume that $a$ is a unit in $R/I$. Then there exist $b \in R$ and $w \in I$ such that $ab = 1 + w$, so $$(ab-1)^k = 0$$ for some $k$. Expanding the left hand side, we see that we can move over $(-1)^k$ and factor out an $a$, and we see that $a$ is a unit in $R$. Thus the Jacobian of the lifts of $f_i$ is a unit in $R$, and we have shown that the lift is \'etale.

We now look to patch these local liftings. Let $U_i$ and $U_j$ be affine subschemes of $U$ over an affine scheme $S_0$ with $U_{i,j} \defeq U_i \times_U U_j$, and $V_i$ a lift of $U_i$ to $S$. Now there is a unique open subscheme $V_{i,j}$ given by restricting the structure sheaf of $V_i$ to the open subset $$|U_{i,j}| \subseteq |U_i| \subseteq |V_i|.$$ From this description we see that $V_{i,j} \times_{S} S_0 \simeq U_{i,j},$ and thus that we have an isomorphism $$\bar \theta_{i,j} \colon V_{i,j} \times_{S} S_0 \simeq V_{j,i} \times_{S} S_0.$$

We lift $\bar \theta_{i,j}$ to an isomorphism $\theta_{i,j} \colon V_{i,j} \to V_{i,j}$ by proving the fully faithfulness in our lemma. To show that the map is full and faithful, one first reduces to the affine case; from there one uses the fact that \'etale morphisms are formally smooth; for details the reader is directed to \cite[\href{https://stacks.math.columbia.edu/tag/025H}{Tag 025H}]{stacks-project}. Supposing that the reduction map is fully faithful, there is an isomorphism $\theta_{i,j}$ lifting $\bar \theta_{i,j}$. In fact, $\theta_{i,j}$ satisfy the cocycle condition because $\theta_{i,k}^{-1} \circ \theta_{j,k} \circ \theta_{i,j}$ is the unique lifting of $\bar \theta_{i,k}^{-1} \circ \bar \theta_{j,k} \circ \bar \theta_{i,j}$, which is the identity.  \hfill \qedsymbol
\end{pfsketch} 
The remarkable equivalence is fascinating because it says that \'etale coverings of a scheme are remarkably insensitive to the algebraic nature of the scheme. We will use the remarkable equivalence repeatedly in our examples of formal schemes, because it often allows one to lift objects from $\mathbb{F}_p$, say, to compatible collections of objects over $\mathbb{Z}/p^n\mathbb{Z}$.

\begin{defn}
Let $S$ be a scheme. The large \'etale site on $S$, denoted $(Sch/S)_{\acute{e}t}$, is the site associated to the Grothendieck topology whose objects are schemes over $S$, and whose coverings are jointly surjective \'etale $S$-maps.

The small \'etale site on $S$, denoted $S_{\acute{e}t}$, is the full subcategory of the big \'etale site whose objects are schemes which are themselves \'etale over $S$. 
\end{defn}
\begin{rem}
The big and small \'etale sites play very different roles in the geometric theory. A good rule of thumb is that the small \'etale site is an enhanced version of the Zariski topology on $S$, while the large \'etale site is the category on which to test properties like descent. In fact, in view of the remarkable equivalence for \'etale maps, the small \'etale sites of a scheme and any nilpotent thickening of it are isomorphic. If one wants, one can say something even stronger: the small \'etale site is invariant under universal homeomorphism \cite[\href{https://stacks.math.columbia.edu/tag/04DZ}{Tag 04DZ}]{stacks-project}. (Universal homeomorphism, we recall, is the property of a morphism being a homeomorphism after any base change.) This topological invariance is a beautiful byproduct of the \'etale theory -- the \'etale site treats a scheme as much as possible like a topological space, but keeps enough algebraic data that sheaf cohomology, for example, is still very interesting.
\end{rem}

As the lifting property demonstrated by analogy, the \'etale topology is the correct context in which to discuss local systems if we want them to behave the same as holomorphic or topological local systems. We will make this statement rigiorous soon enough. Unfortunately, local systems with $\mathbb{Q}_p$-coefficients are not quite as simple as locally constant $\mathbb{Q}_p$-sheaves on the \'etale site.
\begin{defn}
Let $S$ be a scheme. The category of $\mathbb{Q}_p$-local systems (or lisse sheaves) on $X$ has as its objects inverse systems $\{\mathcal{F}_n\}_{n\geq 1}$, where  $\mathcal{F}_n$ is a locally constant $\underline{\mathbb{Z}/p^n \mathbb{Z}}$-sheaf for the \'etale topology, with the condition that the projections $\mathcal{F}_{n+1} \to \mathcal{F}_n$ induce isomorphisms $$\mathcal{F}_{n+1} \otimes \underline{\mathbb{Z}/p^n\mathbb{Z}} \simeq \mathcal{F}_n.$$ The maps in the category of $\mathbb{Q}_p$- local systems are given by $$Hom_{\mathbb{Q}_p-l.s.}(\mathcal{F}, \mathcal{G}) \defeq \left[ \varprojlim_n Hom_{\underline{\mathbb{Z}/p^n\mathbb{Z}}}(\mathcal{F}_n, \mathcal{G}_n) \right] \otimes \mathbb{Q}_p.$$
\end{defn}
Why don't we use locally constant $\mathbb{Q}_p$-sheaves on the \'etale site? One could snarkily say that the whole theory of the \'etale topology is an answer to this question. But this answer does not tell us why the Zariski topology with locally constant sheaves is the incorrect option. This question has a much easier answer, detailed in \cite[Chapter 1, ``The inadequacy of the \'etale topology'']{milne1998lectures}.

There are basically two related problems: first of all, higher Zariski cohomology of an irreducible variety with values in a constant sheaf is trivial. For applications to the Weil conjectures, then, Grothendieck and Deligne certainly needed a better cohomology theory. The second problem is a question of local triviality. Indeed, let $f \colon X \to Y$ be a morphism of schemes with $X$ irreducible. Then, again, the higher pushforwards of the Zariski constant sheaf $\mathbb{Q}_p$ are trivial. Thus the Zariski topology reveals little about the map $f$, let alone any sort of monodromy.

\begin{rem}
There is a way to get around these pro-systems of sheaves, as discovered by Bhatt and Scholze in \cite{bhatt2013pro}.
In their pro-\'etale site, the $\mathbb{Q}_p$-local systems presented here are honest locally-constant $\mathbb{Q}_p$-sheaves.
\end{rem}

In the topological case a covering space produces a local system by pushing forward the constant sheaf, and we define an analogous arithmetic analogue: for an \'etale cover $f \colon Y \to X$, with $p$ invertible on $X$ and $Y$, we consider the pushforward $f_* \mathbb{Q}_p$, and the smooth and proper base change theorem says that this system of pushforwards is a $\mathbb{Q}_p$-local system. (See \cite[Theorem 20.2]{milne1998lectures}.) On the other hand, given a $\mathbb{Q}_p$-local system $\mathcal{F}$ on $X$, one is then challenged to find an \'etale cover $p \colon Y \to X$ such that $p_* \underline{\mathbb{Q}_p}$ is isomorphic to $\mathcal{F}$. Such a construction indeed exists -- although one is forced to take pro-covers, a natural adjustment based on the differences between the algebraic and analytic worlds.

The theorem that we are ultimately searching for is the following. (For relatively elementary proofs of the finite level-statement below, see \cite[Theorem 6.52]{amazeenetale} and \cite[Theorem 5.4.2]{szamuely2009galois}.) Recall that a scheme is geometrically unibranch if its local rings at all points are geometrically unibranch, a sortof integrality condition \cite[\href{https://stacks.math.columbia.edu/tag/0BPZ}{Tag 0BPZ}]{stacks-project}. Normal schemes are geometrically unibranch.
\begin{thm}[{\cite[Lemma 7.4.7 + Lemma 7.4.10]{bhatt2013pro}}]

Suppose that $X$ is connected and geometrically unibranch. Then there is a natural equivalence between the categories of $\mathbb{Q}_p$ local systems on $X$ and the category of continuous $\mathbb{Q}_p$ representations of $\pi_1^{alg}(X,x)$, where continuity references the $p$-adic topology in the target and the profinite topology on the source.
\end{thm}

Let us take a moment to open the dictionary between these two worlds. Let us consider $\mathbb{Z}/p^n\mathbb{Z}$-modules for $\pi_1^{alg}$ and $\underline{\mathbb{Z}/p^n\mathbb{Z}}$-modules on $X$ for now, and then we will reduce to this case. (We begin in the finite world, as do all good things.) The main idea behind the dictionary is that both modules on $X_{\acute{e}t}$ and finite representations of $\pi_1^{alg}$ correspond to finite \'etale covers of $X$.

Indeed, the most difficult part of the proof is the following representability result.
\begin{prop} 
Let $\mathcal{F}$ be a locally constant sheaf of abelian groups on $X_{\acute{e}t}$ with finite stalk $A$ at all geometric points. Then $\mathcal{F}$ is represented in the larger category $(Sch/X)_{\acute{e}t}$ by a commutative group scheme $p \colon Y \to X$ that is finite and \'etale over $X$.
\end{prop}
Now by definition the profinite \'etale fundamental group of $X$ acts on $Y$, and thus the stalk of $Y$ over the basepoint $x$ (which is the module $\mathcal{F}_x = A$.) This provides one direction of the equivalence.

On the other hand, suppose we are given a $\mathbb{Z}/p^n\mathbb{Z}$-module $M$ on which $\pi_1^{alg}$ acts continuously. By continuity this action factors through a finite quotient $\pi_1^{\acute{e}t}/N$. Now the Galois theory of the \'etale fundamental group says that there is an \'etale cover $\pi \colon P \to X$ whose automorphisms are $H = \pi_1^{\acute{e}t}/N$. There is a $\pi_1^{\acute{e}t}/N$-action on the constant sheaf $M_P$; we have for all $\sigma \in \pi_1^{\acute{e}t}/N$ a map $$\sigma^* M_P \to M_P$$ which takes a section over an \'etale open $V$, $s \in \sigma^*M_P(V) = Map(\pi_0 \sigma^{-1}(V) \to M)$, to the section $\sigma^{-1}s$ using the action of $H$ that we were given on the codomain. We also see that $\pi_* \sigma^* M_P \simeq M_P$, so that the maps $\sigma^*M_P \to M_P$ become genuine automorphisms $$\pi_*M_P \to \pi_*M_P$$ by functoriality. The \'etale sheaf we desire is then $$(\pi_* M_P)^H,$$ the sections of $\pi_*M_P$ that are fixed under the action of $H$.

These two functors are inverses -- a beginning reader is encouraged is read the proof in the references we have given. To reduce to the case of finite modules, one notes that any representation of $\pi_1^{alg}$ must land in a lattice  $$GL_n(L) \subseteq GL_n(\mathbb{Q}_p),$$ by compactness of the profinite \'etale fundamental group. Then one forms the reductions $$L/p^kL,$$ which are free $\mathbb{Z}/p^k\mathbb{Z}$-modules of rank $n$ to which one can apply the previous theorem.

However, to move in the other direction from a $\mathbb{Q}_p$-local system $L$ to a representation of the fundamental group via the case of finite modules requires one to find a $\mathbb{Z}_p$-local system $L'$ such that $L = L' \otimes \mathbb{Q}_p$. In general this is only possible when $X$ is geometrically unibranch; for a counterexample see \cite[Example 7.4.9]{bhatt2013pro}.

We may now mimic the construction of the topological fundamental group via deck transformations.
\begin{prop}
The category of $\mathbb{Q}_p$-local systems, equipped with the fiber functor that sends a local system to its fiber over a point $b$, is a Tannakian category whose associated pro-group scheme we denote $\pi_1^{\acute{e}t}(X,b)^{alg}$. We call this group the pro-algebraic \'etale fundamental group.
\end{prop}
The previous theorem shows that the proalgebraic \'etale fundamental group is the $\mathbb{Q}_p$-algebraic envelope of the profinite fundamental group of $X$ when $X$ is geometrically unibranch. This is a simple consequence of the universal property of the algebraic envelope. (See \ref{envelopes} for standard details on algebraic envelopes and Tannakian fundamental groups.)

\section{Vector bundles with integrable connection}
The next essential component of the motivic fundamental group is the category of vector bundles with integrable connection. Recall that (the sheaf of sections of) a vector bundle over a scheme $X$ is a locally free sheaf of $\mathcal{O}_X$-modules. It is common practice to conflate the sheaf of sections and the total space of the vector bundle, which we can recover using the formula $E = \Spec \left( \Sym(\mathcal{E}^{\vee}) \right)$.
\begin{defn}
Let $\mathcal{E}$ be (the sheaf of sections of) a vector bundle. A connection on $\mathcal{E}$ is a map $$\nabla \colon \mathcal{E} \to \mathcal{E} \otimes_{\mathcal{O}_X} \Omega^1_X$$ satisfying linearity and the Leibniz rule, namely:
$$\nabla(s_1 + s_2) = \nabla(s_1) + \nabla(s_1)$$ and $$\nabla(fs) = f \nabla(s) + s \otimes df$$ for $U \in X$ open, $s_1,s_2,s \in \mathcal{E}(U)$, and $f \in \mathcal{O}_X(U)$.

Finally, we denote by $\mathcal{E}^{\nabla=0}$ the kernel of $\nabla$, called the flat sections of $\nabla$.
\end{defn}
The same definition works for complex manifolds.

Not all vector bundles are created equal, of course. We often imagine that there is a way to parallel transport -- or ``flow'' -- vectors from one fiber of the bundle to another fiber, and that's what the connection gives us: for every path from $p_1$ to $p_2$ in the base, a connection gives an automorphism of the fiber.
But if we are working in the topological category then a path and a homotoped version of the path might give different automorphisms of the fiber. The condition that ensures the homotopy-invariance of the endomorphism in the topological category is \textit{flatness}.

Suppose $X$ is a scheme over a field $K$, and construct the sheaf of differentials valued in $\mathcal{E}$: $\Omega^\bullet_{X/K} \otimes_{\mathcal{O}_X} \mathcal{E}.$ Denote the usual differential on forms by $d \colon  \Omega^n_{X/K} \to  \Omega^{n+1}_{X/K}$. Now consider the map $$D^p \defeq d \otimes 1 + (-1)^p 1 \wedge \nabla \colon \Omega^p_{X/K} \otimes_{\mathcal{O}_X} \mathcal{E} \to \Omega^{p+1}_{X/K} \otimes_{\mathcal{O}_X} \mathcal{E}.$$
\begin{defn}
A connection $\nabla$ on a vector bundle $\mathcal{E}$ is called flat or integrable if the composition $$\mathcal{E} \overset{D^1 \circ D^0} \to \mathcal{E} \otimes_{\mathcal{O}_X} \Omega^2_X$$ is identically zero.
\end{defn}
\begin{rem}
We can give the reader one convenient way to think of flatness in the algebraic category, where homotopy invariance cannot appear in the same way as in the topological category. Flatness in fact ensures that $D^{q+1} \circ D^q = 0$ for each $q$, so that $\Omega^\bullet_{X/K} \otimes_{\mathcal{O}_X} \mathcal{E}$ becomes a chain complex. This handy definition thus allows us to define the de Rham cohomology of $(\mathcal{E}, \nabla)$ as the sheaf cohomology (or hypercohomology in the algebraic category) of this complex.
\end{rem}

Very often, our bundles with come with quite a bit more structure -- this structure is encoded by the notion of a polarized variation of Hodge structures.
\begin{defn}
If $X$ is a complex manifold, then a real variation of Hodge structures on $X$ of weight $k$ is a complex vector bundle with flat connection $(\mathbb{V}, \nabla)$ together with a real structure $\mathbb{V}_{\mathbb{R}}$ and a decreasing filtration by holomorphic sub-bundles $F^p \subseteq \mathbb{V}$ on $\mathbb{V}$ satisfying the following conditions:

First, the filtration $F^p$ must induce a Hodge structure of weight $k$ on each fiber of $\mathbb{V}$. Second, the flat connection $\nabla$ on $\mathbb{V}$ must satisfy Griffiths transversality, which is the requirement that $$\nabla(F^p) \subseteq F^{p-1} \otimes \Omega^1_{X}.$$
\end{defn}

There is an extremely strong sense in which the category of vector bundles with integrable connection over $X$ captures much of its de Rham theory. For example, this category will most definitely determine $H^1_{dR}(X^{an}, \mathbb{C})$. It will also contain the variation of Hodge structures coming from the de Rham cohomology of the fibers of all smooth and proper families $Y \to X$ -- this is the Gauss--Manin connection that we will define in the next section.

The category $Vect(X)$ of finite dimensional vector bundles with integrable connection over a scheme $X/K$ is a Tannakian category over $K$, once one fixes a point $b \in X(K)$, the fiber functor being given by the vector bundle fiber over the basepoint. We denote its Tannakian fundamental group by $\pi_1^{dR}(X,b)$. For what follows, we often have a scheme $X$ over a $p$-adic field $K$ and a fixed embedding $K \hookrightarrow \mathbb{C}$. Then we denote by $\pi^{dR}_{1,K}$ the fundamental group of $X$ and  $\pi^{dR}_{1,\mathbb{C}}$ the de Rham fundamental group of $X_{\mathbb{C}}$.

If we instead consider the subcategory of unipotent objects in $Vect(X)$ (the thick subcategory generated by the structure sheaf with its trivial connection), we denote the corresponding fundamental group by $\pi_1^{dR,un}(X,b)$.

We quickly note the following comparison theorem. Then we have the following comparison isomorphism which is used to equip unipotent $p$-adic fundamental groups with Hodge structures: there is an isomorphism $$\pi^{un,dR}_{1,\mathbb{C}} \simeq \pi^{un,dR}_{1,K} \otimes_{K} \mathbb{C}.$$ One obtains this comparison by comparing the universal objects in each category, ultimately arising from a comparison between complex and $p$-adic de Rham cohomologies compatible with cup product. A similar comparison between algebraic and holomorphic de Rham cohomologies provides a comparison with the ``analytic de Rham fundamental group'' of holomorphic flat vector bundles on the analytification of $X$. We therefore conflate the two types of complex vector bundles from now on. For more details on when these comparison theorems hold for relative completions of fundamental groups, see our discussion about universal objects in the section on relative completions for the motivically inclined.

An essential property of $\pi^{un,dR}_{1,\mathbb{C}}$ is that its coordinate ring has a mixed Hodge structure -- one obtains this filtration from the following theorem about the complex unipotent fundamental group.
\begin{thm}{\cite{chen1977iterated}}
\begin{description}
\item[Coordinate Ring] There is an explicit description of $\mathcal{O}(\pi^{un,dR}_{1,\mathbb{C}})$ in terms of iterated integrals of differential forms on $X$. Concisely, $$\mathcal{O}(\pi^{un,dR}_{1,\mathbb{C}}) \simeq H^0\left( B(\Omega^\bullet_{X/\mathcal{C}}) \right).$$ Here $B$ denotes the bar complex, which we touch on in the main body of the text.

\item[Comparison] We have an isomorphism $$\pi_1^{top}(X,b)^{un} \otimes_{\mathbb{Q}} \mathbb{C} \simeq \pi^{un,dR}_{1,\mathbb{C}},$$ where the $un$ superscript refers to the (unipotent) Malcev completion of the topological fundamental group.

\item[Hodge Structure and Induced Hodge Structure] The above description of the coordinate ring induces a mixed Hodge structure on $\mathcal{O}(\pi^{dR}_{1,\mathbb{C}})$: the weight filtration is given by the length filtration of iterated integrals (plus the number of logarithmic terms in the non-projective case), and the Hodge filtration is given by the number of holomorphic forms in the integral. i.e. it is induced by the Hodge filtration on $H^1_{dR}$. A complete description of these filtrations is given in \cite{hain1987geometry}.

Owing to the comparison theorem above and the one mentioned before the theorem, $\mathcal{O}(\pi_1^{top}(X,b)^{un})$ and  $\mathcal{O}(\pi^{un, dR}_{1,\mathbb{Q}_p})$ inherit rational and $\mathbb{Q}_p$-mixed Hodge structures, respectively.

\end{description}
\end{thm}

\section{The Gauss--Manin connection}

One essential source of vector bundles with integrable connection is the fiberwise de Rham cohomology of proper smooth morphisms of smooth schemes $f \colon X \to S$. For such a family, we define the relative de Rham cohomology to be $H^q(X/S) \defeq \mathbb{R}^{q} f_* (\Omega_{X/S}^{\bullet})$. By \cite[\href{https://stacks.math.columbia.edu/tag/02G1}{Tag 02G1}]{stacks-project} the smoothness of $f$ implies that $\Omega_{X/S}^{\bullet}$ is locally free on $X$; by properness, its derived pushforward is at least coherent.
We will define a connection on this sheaf. (The existence of a connection implies that the sheaf is locally free. See \cite[Proposition 8.8]{katz1970nilpotent}.)

After analytic descriptions going back to Gauss and Manin in various forms, Katz and Oda provided a purely algebraic description in the mid-twentieth century \cite{katz1968on}. Their key insight was that there is a natural filtration on the complex $\Omega^\bullet_X$ given by $$F^i = Im(\Omega^{\bullet-i}_X \otimes \pi^*\Omega^i_S \to \Omega^\bullet_X).$$ We now introduce a fact from linear algebra that will allow us to leverage this filtration to create a connection.

\begin{fact}
If $$0 \to E \to F \to L \to 0$$ is an exact sequence of vector spaces over a field (which we suppress in our notation), we consider the natural filtration on $F$ given by $$Fil^k \Lambda^p F = Im(F^{\otimes(p-k)} \otimes  E^{\otimes k}  \to \Lambda^p F).$$ Then we have that $$Fil^k \Lambda^p F/Fil^{k+1} \Lambda^p F = \Lambda^k E \otimes \Lambda^{p-k} L.$$
\end{fact}
The proof of this fact is an exercise to the reader. 

Now we apply the previous linear algebra fact to the exact sequence $$0 \to \pi^*\Omega^1_S \to \Omega^1_X \to \Omega^1_{X/S} \to 0,$$ which arises from the smoothness of $\pi$, and we have the calculation of the associated graded: the $i$\ts{th} graded piece of the induced filtration on $\Omega^\bullet_X$ is $\pi^* \Omega^i_S \otimes \Omega^{\bullet-i}_{X/S}$.

\begin{defn}
Consider the exact sequence $0 \to gr^1 \to F^0/F^2 \to gr^0 \to 0$ obtained from the aforementioned filtration on $\Omega^\bullet_X$; the connecting homomorphism $\mathbb{R}^q \pi_* gr^0 \overset{\delta} \to \mathbb{R}^{q+1} \pi_* gr^1$ produces a map

\begin{align}
H^q_{dR}(X/S) \simeq \mathbb{R}^q \pi_* gr^0 \overset{\delta} \to \mathbb{R}^{q+1} \pi_* gr^1 \simeq \mathbb{R}^{q+1} (\pi_* gr^0[-1] \otimes \pi^* \Omega^1_S)
\\ \simeq \mathbb{R}^q\pi_* gr^0 \otimes \pi^* \Omega^1_S = H^q_{dR}(X/S) \otimes \Omega^1_S. \end{align}
(The second isomorphism above is the calculation of the associated graded, and the third is given by the projection formula.)

The resulting composition $\nabla_{GM} \colon H^q_{dR}(X/S) \to H^q_{dR}(X/S) \otimes \Omega^1_S$ is called the Gauss--Manin connection.

\end{defn}

Here is a concrete description of the connection when $X$ and $S$ are both affine and $q=1$, from Besser \cite[1.7.1]{besser2012heidelberg}.
\begin{exmp}{(Affine Gauss--Manin for $H^1_{dR}(X/S)$.)}

Suppose that $X$ and $S$ are smooth affine schemes over a field $K$, and that $\pi \colon X \to S$ is a smooth $K$-morphism. Then the Gauss--Manin connection on the first relative cohomology has a simple description. Let $\omega \in \Omega^1_{X/S}$ be a closed form with respect to the relative differential. We pick a lift $\tilde \omega \in \Omega_X$; then $d \tilde \omega \in \Omega^2_X$, and in fact $d \tilde \omega$ is in the kernel of the map $\Omega^2_X \to \Omega^2_{X/S}$, because $\omega$ is closed and differentiation commutes with the map to relative differentials. Now using the linear algebra fact again for $0 \to F^1 \to F^0 \to gr^0 \to 0$ shows that this kernel is precisely $F^1$; projecting to the associated graded we obtain an element $\tilde \omega' \in \Omega_S \otimes \Omega^1_{X/S}$ that must in fact live in $\tilde \omega' \in \Omega_S \otimes (\Omega^1_{X/S})^{d=0}$. Taking cohomology classes gives an element $\tilde \omega' \in \Omega_S \otimes H^1(X/S)$.
\end{exmp}

The $F$-isocrystal structure on relative crystalline cohomology implies that formal solutions to the Gauss--Manin connection over a $p$-adic field converge on residue tubes \cite[3.1.2]{katz1973travaux}, as we now explain.
In other words, the \textit{non-unipotent} vector bundle with integrable connection coming from the relative de Rham cohomology can be integrated on residue tubes at the very least. In fact, one can say even more than convergence -- where the connection does converge, it is Frobenius-equivariant for a natural Frobenius action on both fibers. Lawrence and Venkatesh recall this idea in the work of Berthelot \cite[Prop. 3.6.4]{berthelot2006cohomologie}. Let us follow suit.

For $F$ a number field, we denote by $\mathcal{O}_{F_v}$ the ring of integers of the completion of $F$ at $v$, and if $X$ is a scheme over $\mathcal{O}_{F_v}$, we denote by $X_0$ its special fiber. Write $\kappa$ for the residue field of $\mathcal{O}_{F_v}$.
\begin{prop}
Let $\tilde f \colon \mathcal{Y} \to \mathcal{X}$ be a smooth and proper morphism of schemes over $\mathcal{O}_{F, v}$, for $v$ a place of $F$, and suppose that $x,x' \in \mathcal{X}(\mathcal{O}_{F,v})$ are points lying in the same residue disc.

Then the formal solution to the Gauss--Manin equations between the fibers of $Y$ above $x$ and $x'$, which we denote by $p_{dR}$, can be identified with the comparison isomorphism to crystalline cohomology of the special fiber via the following commutative diagram:
\[
\begin{tikzcd}
H^q_{dR}(Y_x,F_v) \arrow{dd}{p_{dR}} \arrow{dr}{comp} & \\
& H^q_{crys}(X_0/W(\kappa)) \otimes_{W(\kappa)} F_v\\
H^q_{dR}(Y_{x'},F_v) \arrow{ur}{comp} & 
\end{tikzcd}
\]
\end{prop}
(Note that the comparison to the special fiber endows the de Rham cohomology groups above with a semilinear action of Frobenius, and endows the crystalline cohomology with a filtration after tensoring to $F_v$.) In particular, $p^{dR}$ converges in the whole residue disc. The commutativity of this diagram is essential, because it says that the Gauss--Manin connection is Frobenius-equivariant. This proposition, then, is an analogue of the Tannakian interpretation of Coleman integrals by Besser in a non-unipotent case -- although it only works locally. It thus plays an essential role in our extension of the Chabauty diagram to the relative completion. We spell out this convergence for the entire relative completion in Section \ref{analyticity}.

Finally, we relate one classical and one contemporary fact about vector bundles.
\begin{rem}
The astute reader might wonder what happened to the \'etale topology here. The last section was so concerned that local systems on a scheme be locally trivial in the \'etale topology. If we are aiming for a Riemann-Hilbert correspondence that equates local systems and vector bundles with integrable connection, then, why are our vector bundles not described using local triviality in the \'etale topology? The reason is the following statement, the third miracle of the \'etale topology that we've seen thus far.
\begin{thm}
There is an equivalence of categories $$QCoh(X_{\acute{e}t}) \to QCoh(X_{Zar})$$ induced by the inclusion of the Zariski site into the \'etale site. \cite[\href{https://stacks.math.columbia.edu/tag/03DR}{Tag 03DR}]{stacks-project}
\end{thm}
\end{rem}

\chapter{Stacks, Sites, and Pseudofunctors}
As will have been clear by now, the arithmetic-geometric toolkit uses in an essential way both the theories of stacks and non-archimedean analytic spaces. After a brief review of stacks, we take this section to introduce the \textit{Geometric Contexts} of Porta and Yu \cite{porta2016higher}. Geometric contexts will allow us to create a theory of stacks ``modeled on'' any reasonable geometric theory.

Two brief comments: We have chosen to restrict ourselves to classical (i.e. 1-truncated) stacks; hopefully non-homotopical readers will be grateful. Second, we have taken the pseudofunctor approach to stacks here since our eventual applications will be for moduli spaces.

The idea of a stack formed in the Grothendiek school from investigations in moduli theory; indeed, early uses for stacks included the work of Deligne and Mumford on the geometry of the moduli space of curves. The aforementioned mathematicians, especially Artin, realized that many moduli spaces could be exhibited as quotients of schemes. But a quotient of a scheme might not be a scheme!

\section{Stacks and geometric contexts}
If we're going to use functors with values in groupoids, we had better understand the collection of all groupoids.

\begin{defn}
A (strict) 2-category $\mathcal{C}$ consists of the following data:
\begin{description}
\item[(1)] A collection of objects, $Ob(\mathcal{C})$.
\item[(2)] For any two objects $x,y \in Ob(\mathcal{C})$, a category $Mor(x,y)$. Its objects are called 1-morphisms of $\mathcal{C}$ and its morphisms are called 2-morphisms of $\mathcal{C}$. When one composes 2-morphisms in the category $Mor(x,y)$, one calls it vertical composition.
\item[(3)] For any three $x,y,z \in Ob(\mathcal{C})$, a functor $Mor(x,y) \times Mor(y,z) \to Mor(x,z)$.
The image of two 1-morphisms is called their composition, and the image of two 2-morphisms is called their horizontal composition.
\end{description}
This data must satisfy the following requirements:
\item[(1)] The collection of objects and 1-morphisms, equipped with composition of 1-morphisms, forms a category.
\item[(2)] Horizontal composition of 2-morphisms is associative.
\item[(3)] The 2-identity morphism $id_{id_x}$ of the identity 1-morphism $id_x$ is a unit for horizontal composition.
\end{defn}
\begin{exmp}
The collection of groupoids, equipped with functors as 1-morphisms and natural transformations as 2-morphisms, forms a 2-category.
\end{exmp}

\begin{defn}
Let $\mathcal{C}$ be a 2-category. Then we say that a diagram
\[
\begin{tikzcd}
a \arrow{r} \arrow{d} & b \arrow{d} \\
c \arrow{r} & d
\end{tikzcd}
\]
2-commutes if there is a 2-isomorphism (invertible 2-morphism) between the top-right and bottom-left compositions.

If an object $w$ sits in a 2-commutative diagram
\[
\begin{tikzcd}
w \arrow{r} \arrow{d} & y \arrow{d} \\
x \arrow{r} & z
\end{tikzcd}
\]
then we say that $w$ is the fibered product of $x$ and $y$ over $z$ if, for any other 2-commuting diagram
\[
\begin{tikzcd}
t \arrow{r} \arrow{d} & y \arrow{d} \\
x \arrow{r} & z
\end{tikzcd}
\]
there exists a 1-morphism $\gamma \colon t \to w$ such that the above diagram, now with all maps from $w$ included, 2-commutes. In other words, $w$ 2-represents the functor of categories $Mor(-,x) \times_{Mor(-,z)} Mor(-,y)$. This universal property characterizes $w$ up to 2-equivalence, and in that case we write $w=x \times_z y$. 
\end{defn}

\begin{prop}
The 2-category of groupoids has 2-fibred products. If $H_1, H_2,$ and $H_3$ are groupoids with functors $F_1 \colon H_1 \to H_3$ and $F_2 \to H_3$, then their fibred product is given explicitly by the formula $$Ob(H_1 \times_{H_3} H_2) = Ob(H_1) \times_{Ob(H_3)} Ob(H_2),$$ $$Mor((x_1,x_2),(y_1,y_2)) = Mor(x_1,y_1) \times Mor(x_2,y_2)$$
\end{prop}
Now that we have this machinery we may introduce a notion that will allow us to emulate the functor of points in the 2-categorical context.
\begin{defn}
A pseudofunctor $F$ from a category $\mathcal{C}$ to a 2-category $\mathcal{D}$ is a map on objects and for any two objects, a map on 1-morphisms; for each chain $x \overset{\alpha }\to y \overset{\beta} \to z$ in $\mathcal{C}$ there is specified a 2-morphism $F(\beta \circ \alpha) \to F(\beta) \circ F(\alpha)$.
\end{defn}

We may now begin to abstract the conditions for a good theory of stacks, which will be based on the concept of a \textit{geometric context} as defined in \cite{porta2016higher}.
\begin{defn}
A \textit{geometric context} $(\mathcal{C}, \tau, \bold{P})$ consists of a Grothendieck site $(\mathcal{C}, \tau)$ along with a class of morphisms in $\mathcal{C}$ satisfying:
\begin{itemize}
\item Every representable presheaf on $\mathcal{C}$ is a sheaf.
\item The class $\bold{P}$ is closed under isomorphism, composition, and pullback.
\item Every covering in $\tau$ consists of morphisms in $\bold{P}$
\item For every morphism $f \colon X \to Y$ in $\mathcal{C}$, if there is a covering $\{ U_i \to X\}$ such that the composites $\{U_i \to Y\}$ are in $\bold{P}$, then so is $f$ itself.
\end{itemize}
\end{defn}

\begin{rem}
We have replaced Porta and Yu's first condition on hypercompleteness with the simple sheaf condition. However, by their Corollary 2.5 we have that the two are equivalent for a 1-site.
\end{rem}

We note that Porta and Yu use the following homotopical definition of stack -- the reader is welcome to look at the work of Hollander below for the definition of a 2-sheaf.
\begin{defn}
A stack in groupoids on $\mathcal{C}$ is a presheaf on $\mathcal{C}$ valued in groupoids  that satisfies the 2-sheaf condition.
\end{defn}
\begin{defn}
A classical stack in groupoids is a pseudofunctor $F \colon \mathcal{C} \to Grpd$ such that all descent data is effective \cite[\href{https://stacks.math.columbia.edu/tag/02ZC}{Tag 02ZC}]{stacks-project} and for any $U \in \mathcal{C}$ and any two objects $x,y \in F(U)$, the presheaf $Isom(x,y)$  on $\mathcal{C}/U$ is a sheaf.
\end{defn}
We have the following theorem.
\begin{prop}{(\cite[Theorem~1.1]{hollander2008homotopy})}
A classical stack is a stack, and vice-versa.
\end{prop}
Even though Porta and Yu use the language of higher stacks, we therefore lose nothing by considering their theorems in the classical language of stacks.

We will repeatedly use the following operation.
\begin{fact}(\cite[\href{https://stacks.math.columbia.edu/tag/0435}{Tag 0435}]{stacks-project}) \label{stackification}
Given $F$ a category fibered in groupoids over a site $\mathcal{C}$. Then there is a stack $\tilde F$, called the stackification of $F$, which is universal in the following sense: there is a 1-morphism of categories fibered in groupoids $$s \colon F \to \tilde F,$$ and given another stack $G$ and a 1-morphism $F \to G$ there exists a 1-morphism $\tilde F \to G$ that factors the map $F \to G$. This property characterizes $\tilde F$ up to unique 2-isomorphism.

The defining properties of stackification, up to unique 2-isomorphism, are:
First, given $T \in \mathcal{C}$ and $x,y \in F(T)$, the map $Mor(x,y) \to Mor(s(x),s(y))$ identifies the latter functor as the sheafification of the former.
Second, for every $x' \in \tilde F(T)$ there is a $\tau$-covering $\{T_i\}$ of $T$ such that $x'|_{T_i}$ is in the essential image of $s|_{T_i}$.
\end{fact}

We can now define a geometric stack, based on the choice of a geometric context $(\mathcal{C}, \tau, \bold{P})$. Do note that we use the notation $X$ for the Yoneda functor of a sheaf ((-1)-stack) representable by an object $X$, as is common.
\begin{defn}
\begin{description}
\item[Representability] A morphism $F \to G$ of stacks is called representable if, for all stacks $X$ representable in $\mathcal{C}$ and all maps $X \to G$, the 2-fibered product $X \times_G F$ is representable in $\mathcal{C}$.
\item[Properties of Representable Morphisms]
A map of stacks $F \to G$ is said to be in $\bold{P}$ if it is representable and for any representable $X$ as above and map $X \to G$, the map $X \times_G F \to X$ is in $\bold{P}$. Note that, by representability and Yoneda, $X \times_G F \to X$ is in $\mathcal{C}$, so this condition makes sense.
\item[Effective Epimorphism] 
Let $\coprod \mathcal{F}_i \to \mathcal{G}$ be a family of sheaves of groupoids on the site $\mathcal{C}$. We call this family an effective epimorphism if its Cech nerve is a simplicial resolution of $\mathcal{G}$. Fortunately, a theorem of Lurie says that a morphism of objects in an infinity-topos is an effective epimorphism if and only if the 1-truncation of the morphism is an effective morphism. The truncated version is unwound as follows. A family $ \coprod \mathcal{F}_i \to \mathcal{G}$ of sheaves of sets on $\mathcal{C}$ is an effective epimorphism if for any $X \in \mathcal{C}$ the sequence \[Mor(\mathcal{G}, h_X) \to \prod_i Mor(\mathcal{F}_i,h_X) \rightrightarrows \prod_{i,j} Mor(\mathcal{F}_i \times_{\mathcal{G}} \mathcal{F}_j,h_X)\] is an equalizer diagram.
\item[Atlas]
An atlas of a stack $F$ is a family of representable maps $\{U_i \to F\}$ from stacks which are representable in $\mathcal{C}$ such that each map is in $\bold{P}$, and the map of sheaves $\coprod U_i \to F$ is an effective epimorphism.
 \end{description}
\end{defn}
Intuitively, an atlas of a stack can be thought of as a space which one wants to quotient by some equivalence relation to obtain the stack.
\begin{defn}[``Geometric Space'' (0-Stack)] \label{zerostack}
We say that a sheaf on the site $\mathcal{C}$ is a geometric space (the analogue of algebraic space in our context) if it has an atlas of representable stacks (representable by elements of $\mathcal{C}$) and its diagonal is representable. We say that a morphism of stacks is representable by geometric spaces if the pullback by any test object in $\mathcal{C}$ is representable by a geometric space. We reserve the stand-alone word representability for morphisms whose pullback by a test object is actually representable in $\mathcal{C}$.
\end{defn}
Do note that the terminology ``geometric space'' is non-standard, but it accords closely with the use of the terminology ``algebraic space'' in the classical context.
\begin{defn}[Geometric Stack (1-Stack)]
We say a stack $F$ is geometric if its diagonal morphism $F \to F \times F$ is representable by geometric spaces and it has an atlas that is an $``0-\bold{P}''$ morphism from a representable stack (again, representable in $\mathcal{C}$). This second condition means that the atlas is representable by geometric spaces, and also that for each test object $X \in \mathcal{C}$, the geometric space $A$ representing the pullback of the atlas by $X$ has its own atlas $U$ such that the composition $U \to A \to X$ is in $\bold{P}$.
\end{defn}
There are a couple of properties that one immediately wishes to check.
\begin{prop}
Let $F \to G$ be a 1-morphism of stacks in groupoids over $\mathcal{C}$ that is representable by geometric spaces. Let $Z \to G$ be a 1-morphism of stacks in groupoids, not necessarily representable. Then the 2-fibered product map $F \times_G Z \to Z$ is representable by algebraic spaces. Furthermore, if $F \to G$ has property $\bold{P}$, so does the pullback.
\end{prop}
\begin{proof}
Let $X \in \mathcal{C}$ be a test object and $X \to Z$ a test morphism; and consider the pullback 
\begin{align} \label{pb}
(F \times_G Z) \times_Z X \to X.
\end{align}
On the other hand, considering the composition $X \to Z \to G$ gives a pullback $F \times_G X$ which is representable by a geometric space by assumption. There is a canonical isomorphism between the two pullbacks induced by their universal properties, and thus the pullback (\ref{pb}) is representable by a geometric space.

The second statement follows in the same way.
\end{proof}
\begin{prop}
Let $F \to G$ and $G \to H$ be maps of geometric stacks with property $0-\bold{P}$. Then their composition has property $0-\bold{P}$.
\end{prop}
\begin{proof}
Let $X \in \mathcal{C}$. Then $X \times_H G = (X \times_H G) \times_G F$, which one can see using the explicit definition of 2-fibred product. Now $X \times_H G$ is a geometric space by assumption. For any second test object $Y \in \mathcal{C}$, we have that $Y \times_{X \times_H G} (X \times_G F) = Y \times_G F$, which is again a geometric space by assumption. 
\end{proof}
From our perspective in this work, for example when proving representability of nonabelian Galois cohomology, we will want to know when stacks in groupoids inherit the structure of algebraic stacks simply by sitting in a long exact sequence. The following definition and proposition give exactly such a criterion.
\begin{defn}
Let $f: F \to G$ be a 1-morphism of stacks in groupoids on the site $(\mathcal{C}, \tau)$. We say that $f$ is geometric if, for all representable stacks $X \in \mathcal{C}$ and maps $X \to G$, the 2-fibre product $X \times_G F$ is a geometric stack.
\end{defn}
We have come to the key proposition, which is simply a translation of \cite[\href{https://stacks.math.columbia.edu/tag/05XX}{Tag 05XX}]{stacks-project} into the abstract language of geometric contexts.
\begin{prop}[Geometric Morphisms Transfer Geometricity, cf. Proposition 1.3.3.4 \cite{toen2008homotopical}] \label{transfer}

Let $f \colon F \to G$ be a geometric map of stacks in groupoids. Then if $G$ is a geometric stack, so is $F$.
\end{prop}
\begin{proof}
We create an atlas for $F$ as follows. By definition, $G$ has an atlas $U \to G$. Now take the fibre product with $G \to F$; by the geometricity of the map, the fibre product $U \times_G F$ is a geometric stack. Thus it has an atlas $W$, and we take the composition $W \to U \times_G F \to F$ to be the prospective atlas for $F$.

We claim that $U \times_G F \to F$ is representable by algebraic spaces and belongs to $\bold{P}$. It is representable because it is the pullback of a representable morphism, and thus also has property $\bold{P}$ by the above. Now $W \to F$ also has property $\bold{P}$ by composition, and it is again an effective epimorphism. 

We have reduced to the following statement: Suppose $u \colon \mathcal{U} \to \mathcal{X}$ is a 1-morphism of stacks in groupoids on the site $\mathcal{C}$. If $\mathcal{U}$ is a geometric space and $u$ is an effective epimorphism representable by geometric spaces and in $\bold{P}$, then $\mathcal{X}$ is a geometric stack. This is \cite[\href{https://stacks.math.columbia.edu/tag/05UL}{Tag 05UL}]{stacks-project}, which transfers verbatim to a general geometric context.
\end{proof}

But our stacks will rarely come to us as stacks -- rather, they will be the stackification of a pseudofunctor. Fortunately, we have the following useful fact (\cite[\href{https://stacks.math.columbia.edu/tag/04Y1}{Lemma 04Y1}]{stacks-project}):
\begin{fact}
Stackification commutes with 2-fibered products.
\end{fact}
Thus when we check representability in Proposition \ref{transfer}, we are allowed to show that the stackification of the fiber of pseudofunctors is representable by a geometric stack -- this is easier than checking the fiber of the stackification.

\section{Facts on quotient stacks}
We will also need some basics on quotient stacks.
\begin{defn}
Let $X$ be a geometric space for a geometric context $(\mathcal{C}, \tau, \bold{P})$, and $G$ a group geometric space in the same context acting on $X$. (In other words, we have a map $G \times X \to X$ that satisfies the usual commutativity diagrams.)

The quotient of $X$ by $G$ is a stack in groupoids, denoted $\left[ X/G \right]$. It is defined starting with the quotient groupoid $$\left[ X/G \right](T) = X(T)/G(T)$$ for any test object $T \in \mathcal{C}$ and stackifying.
\end{defn}

We also need the following proposition.
\begin{prop} \label{smoothquot}
Suppose that $(\mathcal{C}, \tau, \bold{P})$ is as above, and additionally that $\mathcal{C}$ has a terminal object. Let $X$ be an object of $\mathcal{C}$ and $H$ a group object in $\mathcal{C}$ acting on $X$ in the usual abstract sense. Then if the map from $X$ to the terminal object is in $\bold{P}$, the stack quotient $\left[ X/H \right]$ is a geometric stack.

\end{prop}
\begin{proof}
By definition $\left[ X/H \right]$ is a stack in groupoids on $\mathcal{C}$; to prove geometricity we follow the proof of the scheme-theoretic case, which is \cite[\href{https://stacks.math.columbia.edu/tag/04X0}{Tag 04X0}]{stacks-project}. The crux of the proof is showing that the map $$X \to \left[ X/H \right]$$ is a smooth atlas. Let $T \in \mathcal{C}$ and $T \to \left[ X/H \right]$ be a test morphism. Then the base change $T \times_{\left[ X/H \right]} X$ is locally of the form $H \times T$. This is because, by the second fact of stackification, we may $\tau$-locally on $T$ lift $T$ to a point of $X$, and then $$T \times_{\left[ X/H \right]} X  \simeq T \times_X (X \times_{\left[ X/H \right]} X) \simeq T \times_{X} (X \times H).$$

Now the base change of $X \to \left[ X/H \right]$ by $T$ is just the projection  $H \times T \to T.$ But this projection is the pullback of the map from $H$ to the terminal object, which is in $\bold{P}$ by assumption. Thus the pullback is also in $\bold{P}$, and we conclude that $$X \to \left[ X/H \right]$$ is an atlas.
\end{proof}

\chapter{Rigid and Formal Geometry} \label{geometry}

\section{Definitions and properties}
A wonderful exposition of formal schemes is the lecture notes of Bosch \cite{bosch2014lectures}, which we follow. The origins of the theory, and a more technical treatment, are in \cite[Chapter 1, Section 10]{dieudonne1971elements}, and the Stacks Project follows this approach \cite[\href{https://stacks.math.columbia.edu/tag/0AHY}{Tag 0AHY}]{stacks-project}.
\begin{defn}
We call a topological ring $A$ an adic ring if there exists an ideal $I$ such that the topology on $A$ is the same as the $I$-adic topology, i.e. the topology whose system of neighborhoods of the identity is given by $\{I^n\}_{n \geq 0}$. We call $I$ an ideal of definition for $A$, and write $(A,I)$ when we want to choose an ideal of definition for $A$.
\end{defn}
We assume here that all adic rings are complete and Hausdorff. A formal scheme should have a full structure sheaf's worth of functions, with only the special fiber visible in the underlying topological space. We formalize this idea as follows.
\begin{defn}
\begin{description}
\item[Local structure]
The formal spectrum of $A$, denoted $\Spf(A)$, is a locally ringed space whose set of points is the set of open prime ideals of $A$. Its topology is the one generated by sets of the form $$D(f) \defeq \{\mathfrak{p} \in \Spf(A) \text{ such that } f \not \in \mathfrak{p}\}.$$
The value of the structure sheaf on such an open set is given by $$\mathcal{O}_{\Spf(A)}(D(f)) = A\langle f \rangle \defeq \lim_n (A/I^n[f^{-1}]).$$ Any locally ringed space isomorphic to $\Spf(A)$ for some $A$ as above is called an affine formal scheme. 
\item[Global structure]
A formal scheme is a locally ringed space $X$ that has a covering by open neighborhoods which are isomorphic as locally ringed spaces to affine formal schemes.
\item[Morphisms]
A morphism of formal schemes is a morphism as locally ringed spaces.
\end{description}
\end{defn}
\begin{rem}
If $X$ is quasi-compact and quasi-separated (cf. \cite[\href{https://stacks.math.columbia.edu/tag/0AJA}{Tag 0AJA}]{stacks-project}, \cite[\href{https://stacks.math.columbia.edu/tag/0AJ7}{Tag 0AJ7}]{stacks-project}), the ideals of definition in each affine piece glue to a global ideal sheaf of definition $\mathcal{I}$ \cite[Proposition 3.32]{yasuda2009non}. This means that for each open $U \subseteq X$, $\mathcal{I}(U)$ forms an ideal of definition of $\mathcal{O}_X(U)$. Any formal schemes dealt with in this work have such a global ideal of definition.
\end{rem}

This definition accords with the idea that only the special fiber should be visible in the underlying set of points of a formal scheme: Since $A$ is adic, a prime ideal is open if and only if it contains a power of $I$, for $I$ an ideal of definition. But such ideals are in bijection with $\Spec(A/I)$. The open-ideal formalism thus allows us to dispense with a choice of ideal of definition.

Formal schemes have equivalent definitions as colimits of schemes over the reductions $A/I^k$ for all $k$. We freely use this description below.

Finally, we discuss some properties of morphisms of formal schemes.
\begin{defn}
A morphism $f \colon X \to Y$ of formal schemes is called adic if, under the canonical morphism $f^* \mathcal{O}_Y \to \mathcal{O}_X$, each ideal sheaf of definition of $f^* \mathcal{O}_Y$ generates an ideal sheaf of definition of $\mathcal{O}_X$.
\end{defn}

\section{The formal geometric context}
We now move to define the properties that will underpin formal stacks, namely smooth morphisms. We will not need the \'etale site in the body on this work, but we nevertheless describe it here because it is likely important for future work on Selmer stacks, and it is a natural topic in the study of stacks.
\begin{defn}
Let $f \colon X \to Y$ be an adic morphism of quasi-compact and quasi-separated formal schemes, and write $f_n \colon X_n \to Y_n$ for the reduction of $f$ modulo $\mathcal{I}^n$, where $\mathcal{I}$ is an ideal of definition for $Y$;
these are maps of schemes. We say that $f$ is an \'etale map if $f_n$ is etale for every $n$. 
\end{defn}

\begin{rem}
Any morphism we use will occur in the simpler case where $X$ and $Y$ above are each adic over an adic ring $(A,I)$, and $X_n$ and $Y_n$ are given by the reductions $X/I^n$ and $Y/I^n$.
\end{rem}

We also pause to note that the notion of \'etale can be extended past adic morphisms using lifting properties. For the more general definition along with a proof that it agrees with the definition given above for adic morphisms, see \cite{tarrio2009local}.
\begin{defn}
We say that an adic morphism $f \colon X \to Y$ of formal schemes is smooth if $f$ every reduction $f_n \colon X_n \to Y_n$ as above is smooth.
\end{defn}
Again, this definition can be extended past the adic case, and the differences between the two definitions are discussed in \cite{tarrio2009local}. For the proof of the next statement, see \cite[\href{https://stacks.math.columbia.edu/tag/0DE9}{Tag 0DE9}]{stacks-project}.
\begin{fact}
Fix a formal scheme $X$. The category of all formal schemes adic over $X$ with adic morphisms between them and \'etale covering families as covers forms a Grothendieck site.
\end{fact}

This site satisfies all of the properties we know and love when it comes to the big \'etale topology for schemes.
\begin{prop} \label{subcanonical}
The \'etale topology is subcanonical, i.e. the Yoneda embedding sends formal schemes over $X$ to sheaves on the formal \'etale site.
\end{prop}
\begin{proof}
Let $Y = \varinjlim Y_i$ be a formal scheme over $X$. Now the \'etale topology for schemes is subcanonical (\cite[\href{https://stacks.math.columbia.edu/tag/03O3}{Tag 03O3}]{stacks-project}) and the colimit of sheaves is a sheaf, and thus $Y$ is a sheaf on the big \'etale site for schemes. We are now reduced to the following fact: suppose that $F$ is a presheaf on a site $\mathcal{C}$ and that every element of $\mathcal{C}$ can be written as the filtered colimit of elements in a subsite $\mathcal{D} \subseteq \mathcal{C}$. Then if $F$ is a sheaf on $\mathcal{C}$, it is a sheaf on $\mathcal{D}$. We leave this statement as an exercise for the reader.
\end{proof}

The key point is the following \cite{yu2018gromov}. If $A$ is an adic ring, we denote by $FSch_A$ the site of formal schemes locally finitely presented over $A$ equipped with the \'etale topology.
\begin{prop}
Fix an adic ring $A$. Then the triple \[(\mathcal{C}, \tau, \bold{P}) = (FSch_A, \text{\'etale coverings}, \text{smooth morphisms})\] is a geometric context. The same is true if we replace \'etale morphisms with coverings for the formal Zariski topology.
\end{prop}
\begin{proof}
\begin{description}
\item[(1)] This is just Proposition \ref{subcanonical}. 
\item[(2)] Closure under isomorphism is obvious. Closure under composition follows because if $X \to Y$, $Y \to Z$ factor through $\mathbb{A}_Y^n$ and  $\mathbb{A}_Z^m$, then we have the factorization $$X \to \mathbb{A}_Y^n \to Y \to \mathbb{A}_Z^m \to Z$$ for the composition. Regarding pullbacks, if $S \to Y$ is any map, then $S \times_Y X \to S$ factors via pullbacks of the factoring maps; pullbacks commute with factorization of maps.
\item[(3)] Each reduction $f_n \colon X_n \to Y_n$ is \'etale as a scheme map, and is thus smooth.
\item[(4)] The property of being \'etale is \'etale-local on the source for schemes, and this statement extends trivially to formal schemes. 
\end{description}
\end{proof}

\section{Rigid analytic geometry and the generic fiber construction: from $\mathbb{Z}_p$ to $\mathbb{Q}_p$ and back again} \label{compat}

We have defined cohomology and Selmer stacks as functors on a certain site of rigid varieties. However, the theorems that we will need on moduli of representations of profinite groups are written in the language of formal stacks. We thus show in this section how to transfer results from formal stacks to rigid analytic stacks. We refer the reader to \cite[Chapter 8]{fresnel2012rigid} for the theory of rigid geometry.

For the remainder of this section $R$ will be a discrete valuation ring of characteristic zero and $K$ will be its field of fractions; $Rig_K$ denotes the category of rigid analytic varieties over $K$. When we use the word affinoid, we mean strictly affinoid, i.e. a quotient of the Tate algebra $k\langle X_1, \ldots, X_r \rangle$. A rigid space is a locally $G$-ringed space $X$ that is locally isomorphic to spaces of the form $\Spm(A)$, where $A$ is a quotient of a Tate algebra and $\Spm(A)$ has the Tate topology. We will often use the following construction.
\begin{thm}{\cite[Proposition A.3.1]{conrad2010universal}} \label{univan}

There is an analytification functor 
\begin{align}
Sch_K \to Rig_K  \\ 
X \mapsto X^{an}
\end{align}
That satisfies the following property: Suppose that $X$ is separated. If $Z = \Spm(A)$ is affinoid then $X^{an}(\Spf(A)) = X(\Spec(A))$.
\end{thm}
Fix a rigid variety $Y.$ One consequence of the above theorem is that if $F$ is a contravariant functor on rigid spaces affinoid over $Y$ and $\{U_i = \Spm(A_i)\}$ is an admissible affinoid covering of $Y$ such that $F|_{U_i}$ is representable by a scheme $S_i$, then $F$ is representable as a rigid space and is given over $U_i$ by $S_i^{an}$.

\begin{defn}
A map of rigid analytic spaces $f \colon Y \to X$ is \'etale at $y \in Y$ if the associated map of local rings $$\mathcal{O}_{X,f(y)} \to \mathcal{O}_{Y,y}$$ is flat and unramified. We call $f$ \'etale if it is \'etale at all $y \in Y$. A map of rigid spaces is smooth if it satisfies the classical Jacobian criterion \cite[\href{https://stacks.math.columbia.edu/tag/01V9}{Tag 01V9}]{stacks-project}, locally in the Tate topology.
\end{defn}
\begin{defn}{\cite[Section 8.2]{fresnel2012rigid}}
The big \'etale site on $Rig_K$ is defined as follows. An object therein is a rigid space over $K$, morphisms are usual morphisms of rigid spaces, and a family of maps $g_i \colon Z_i \to Y$  is deemed a covering family if $g_i$ are \'etale and for any (equivalently, some) choice of admissible affinoid coverings $Z_i = \cup_{j}Z_{i,j}$, we have that $$Y  = \cup_{i,j} g_i(Z_{i,j})$$ and this union is an admissible covering for the $G$-topology on $Y$.
\end{defn}

Then we have the following fact, which follows identically as in the case of schemes.
\begin{fact}
The big \'etale site of rigid analytic varieties over $K$ with smooth morphisms as atlases forms a geometric context. If we replace \'etale covers with covers in the Tate Grothendieck topology, then the resulting category is a site.
\end{fact}
\begin{rem}
Whenever we mention a rigid stack in this work we are referring to stacks in the Tate topology with smooth morphisms as atlases, but as we mention the possibility of using the \'etale topology in the introduction, we have discussed it here.
\end{rem}

We now recall the rigid generic fiber construction, introduced by Berthelot, which constructs a rigid analytic variety from a formal scheme. On affine pieces, this takes an admissible $R$-algebra $A$ to $A \otimes_R K$, and one extends this construction to all formal schemes by gluing. The key point is that two formal schemes which are related by a so-called formal blowing-up have the same associated rigid space. We remind the reader of the general theorem.
\begin{thm}{\cite[Theorem 3]{bosch2014lectures}}

Let $R$ be a complete valuation ring of height 1 with field of fraction $K$. Then the functor $$rig \colon FSch_R \to Rig_K$$ induces an equivalence between the category of all admissible formal $R$-schemes which are quasi-compact, localized by the class $S$ of admissible formal blowing-ups, and the category of all quasi-separated rigid $K$-spaces that are quasi-compact.
\end{thm}
Since the rigid generic fiber respects \'etale and smooth maps, we have by left Kan extension an analytification functor from the 2-category of formal stacks over $\Spf(R)$ to the 2-category of rigid stacks over $K$. By definition of Kan extension, this functor is computed using the following 2-colimit formula: for $\mathfrak{X}$ a formal stack and $T$ a rigid affinoid, then the rigid generic fiber of $\mathfrak{X}$ is $$\mathfrak{X}^{rig}(T) \defeq \varinjlim_{Y^{rig} \to T} \mathfrak{X}(Y).$$ Here the colimit runs over all morphisms from the generic fibers of formal schemes $Y$ to the test affinoid $T$.

For us, the important claim is that the moduli problem of torsors is respected by the operation of taking rigid generic fibers. We adapt the discussion of \cite[Lemma 3.15(ii)]{chenevier2014p}, with a sprinkle of \cite[Lemma 3.2.4]{ardakovequivariant}. For the next proposition, it is important to note that if a profinite group $G$ acts continuously on the formal scheme $\mathcal{G}$, then functoriality furnishes a continuous action of $G$ on $\mathcal{G}^{rig}$ -- see \cite[Proposition 3.1.12]{ardakovequivariant} for a thorough explanation. Furthermore, every such action arises from a $G$-action on a formal model of $\mathcal{G}^{rig}$ (\cite[Lemma 3.2.4]{ardakovequivariant}.)

If $\mathcal{G}$ is an algebraic group over $R$ and $G$ acts continuously on $\mathcal{G}$, then we have an an obvious notion of continuous cohomology groupoids $H^1(G, \mathcal{G})$ on the site $FSch_R$. 
\begin{prop} \label{compatiblewithrig}
If the stack $H^1(G,\mathcal{G})$ on $FSch_R$ is representable by a formal stack then its rigid analytification $H^1(G,\mathcal{G})^{rig}$ represents the stackification of the pseudofunctor $$\Lambda \mapsto H^1(G,\mathcal{G}_K(\Lambda))$$ on affinoid $K$-algebras. 
\end{prop}
\begin{proof}
Stackification is also a left Kan extension, so it commutes with generic fibers and we may deal only with the unstackified cohomology pseudofunctor. Thus the universal property under question is that for any affinoid $K$-algebra $A$ we have a bijection $$\varinjlim H^1(G,\mathcal{G}(\mathcal{A})) \to H^1(G,\mathcal{G}(A)),$$ where the colimit runs over all $\mathcal{A} \in FSch_R$ with rigid generic fiber $A$.

First of all, there is always a continuous ring map from $\mathcal{A}$ to $A$ which induces the map above. To see surjectivity, consider a continuous cocycle $$c_A \colon G \to \mathcal{G}(A).$$ Now continuity and compactness of $G$ imply that the image of $c_A$ lies in a compact subset of $\mathcal{G}(A)$, i.e. in $\mathcal{G}(\mathcal{A}')$ for some model $\mathcal{A}'$ of $A.$

If $\mathcal{A}$ is torsion free then $H^1(G,\mathcal{G}(\mathcal{A})) \to H^1(G,\mathcal{G}(A))$ is injective, as the map $\mathcal{A} \to A$ is injective. Since every $\mathcal{A}'$ maps to a torsion free ring, then, the map from the colimit is injective.

\end{proof}

Last but not least, we need the following statement on the representability of quotients of rigid-analytic varieties.

\begin{lem}\label{flatquotient}
Let $T$ be a rigid-analytic space over a non-archimedean field $K$, and $H$ a rigid-analytic group that is flat and locally of finite presentation over $T$. Let $X$ be a rigid-analytic 0-stack over $T$ (cf \ref{zerostack}.) Assume that $X$, $T$, and $H$ have formal models over $R$. Then the quotient stack $[X/H]$ is representable by a rigid stack.
\end{lem}
\begin{proof}
We take formal models for $H$, $T$, and $X$, denoted $H^f$, $T^f$, and $X^f$, respectively. By the main theorem of \cite{bosch1993formal}, $H^f$ may be taken to be flat over $T^f$. In particular, letting $I$ be the ideal sheaf of definition of $T^f$ and $H_n, T_n, X_n$ the reductions of our formal schemes modulo $I^n$, $H_n$ is flat over $T_n$ and it acts on the algebraic space $X_n$. By \cite[\href{https://stacks.math.columbia.edu/tag/06FI}{Tag 06FI}]{stacks-project}, $[X_n/H_n]$ is an algebraic stack over $T_n$.

Now we have an equivalence \[[\varinjlim X_n/ \varinjlim H_n] \simeq \varinjlim [X_n/H_n]\] because (2-)colimits commute. Thus $[X^f/H^f]$ is representable by a formal stack over $T^f$. The commuting of colimits also shows that \[[X^f/ \varinjlim H^f]^{rig}  \simeq [X/ H] , \] because rigid generic fibers are given by the colimit along formal models. We conclude that $[X/H]$ is representable by a rigid stack over $T$.
\end{proof}
\begin{rem}
All of the rigid stacks in this work have underlying formal models, owing to the fact that Wang-Erickson's representation stacks are given by him as formal stacks; furthermore, all constructions in Theorem \ref{representability} can be done using constructions on formal models (and often the rigid-analytic theorem we cite was proven by reduction to formal model.)

One approach which does not work is naively assuming that all rigid stacks with diagonal representable by rigid varieties have formal models, as we see here: Suppose, indeed, that the diagonal of $X$ representable by rigid spaces -- not by 0-stacks. We might think that we obtain a formal model of a rigid stack $X$ with presentation $U \to X$ by taking formal models of $U$ and $U \times_X U$, which yields a stack in groupoids over $R$ which is easily seen to have rigid generic fiber $X$. Unfortunately, smooth rigid spaces do not necessarily have smooth formal models \cite[p. 225]{bosch2014lectures}, so the resulting groupoid may not be a formal stack.
\end{rem}

\section{Strat-representability} \label{stratrep}
We now introduce the idea of strat-representability. This concept does not seem to have been named in the past, but it is encountered in work of To\"{e}n \cite[Corollary 3.3.4]{toen2010simplicial}. We use Rydh's definition of stratifications of stacks \cite{rydh2016approximation}; see \cite[Definition 5.2.1.1]{gaitsgory2019weil} for more discussions on stratifications of stacks.

For a rigid stack $X$, we let $|X|$ denote the set of morphisms $p \colon \Spm(k) \to X$, where $k$ is a finite extension of the base field of $X$, with the equivalence relation given by equating two morphisms $p,p'$ from $k,k'$ if there exists another finite extension $\Omega$ over the base field and a 2-commutative diagram
\[
\begin{tikzcd}
\Spm \Omega \arrow{r} \arrow{d} & \Spm k' \arrow{d}{p'} \\
\Spm k \arrow{r}{p}  & X
\end{tikzcd}
\]
See \cite[\href{https://stacks.math.columbia.edu/tag/04XE}{Tag 04XE}]{stacks-project} for the algebraic case, which easily generalizes to the formal case.
\begin{defn}
A stratification of an algebraic, formal, or rigid stack $X$ is a sequence of finitely presented closed substacks \[\emptyset = X_0 \hookrightarrow X_1 \hookrightarrow \cdots \hookrightarrow X_n\] such that $|X_n| = |X|$.
\end{defn}
The theorem we will need is \cite[Theorem 8.3]{rydh2016approximation}:
\begin{thm}
Let $X$ be a quasi-compact and quasi-separated algebraic stack and let $W \to X$ be a morphism of finite presentation. Then there is a stratification \[\emptyset = X_0 \hookrightarrow X_1 \hookrightarrow \cdots \hookrightarrow X_n\] of $X$ such that $W$ is flat over $X_i \setminus X_{i-1}$ for all $1 \leq i \leq n$.
\end{thm}
Since the rigid generic fiber respects closed immersions and is point-preserving, a stratification of a formal stack induces a stratification of its generic fiber.

In the below definition, we allow any Grothendieck topologies on the site of schemes, formal schemes, or rigid varieties.
\begin{defn}
We say that a stack $F$ on the category of schemes (resp. formal schemes, rigid varieties) is $0$-strat-representable if $F$ is representable as an algebraic (resp. formal, rigid analytic) stack.

We say that a tower \[F_n \to F_{n-1} \to \cdots \to F_0\] is $n$-strat-representable if $F_0$ is $0$-strat-representable and for every affine scheme (resp. affine formal scheme, affinoid rigid analytic variety) $U$ and morphism $U \to F_0$ there is a finite increasing sequence \[H_{1} \hookrightarrow \cdots \hookrightarrow H_{l} = U\] of closed immersions such that the towers \[(H_{j} \setminus H_{j-1}) \times_{F_0} F_n \to \cdots \to (H_{j} \setminus H_{j-1}) \times_{F_0} F_{1}, \qquad 2 \leq j \leq l\] are $(n-1)$-strat-representable. If $n$ is obvious, we just say that the tower is strat-representable.

We will say that a single stack $F$ is strat-representable if it fits into a strat-representable tower. Finally, \[ \cdots \to F_n \to F_{n-1} \to \cdots \to F_0 \] is a pro-stack such that any finite truncation is strat-representable, we say that the pro-stack is pro-strat-representable.
\end{defn}
We will informally say that the locally closed subspaces $(H_{j} \setminus H_{j-1})$ are the ``pieces'' of the stratification.

A key example to keep in mind is the following. Let $\mathcal{C}$ be the category of algebraic, formal, or rigid stacks. 
\begin{exmp}
Let $X$ be a Noetherian stack in the category $\mathcal{C}$, and $\mathcal{F}$ a coherent sheaf on $X$. Then $\mathcal{F}$ can be viewed as the ``total space'' $\mathcal{T} = \mathcal{H}om(\mathcal{O}_X, \mathcal{F})$ on $\mathcal{C}$; this stack associates to any $f \colon T \to X$ the groupoid $\mathcal{H}om(\mathcal{O}_T, f^*\mathcal{F})$. We will show that \[\mathcal{T} \to X\] is strat-representable.

Let $p \colon U \to X$ be an atlas. There exists a stratification $S'_1 \hookrightarrow \cdots \hookrightarrow U$ of $U$ such that $p^*\mathcal{F}$ is locally free over $S'_i \setminus S'_{i-1}$ for every $i \geq 2$ \cite[\href{https://stacks.math.columbia.edu/tag/051S}{Tag 051S}]{stacks-project}. Setting $S_i = p(S'_i)$, $\mathcal{F}$ is free over $S_i \setminus S_{i-1}$ for every $i \geq 2$. Write $\mathcal{F}_i$ for this restriction.  Now $\mathcal{T}$ is represented by the sheaf $\Spec \left(Sym(\mathcal{F}_i^\vee) \right)$ over these complements. This follows because, for locally free sheaves, pullback commutes with dual. Such representability holds, in fact, if and only if $\mathcal{F}_i$ is locally free! (See \cite{nitsure2004representability}.) This is why we can't do better than strat-representability in general.
\end{exmp}

In fact, representability of total spaces of locally constant sheaves works perfectly in the rigid context. The following construction is a special case of relative analytification, cf. \cite[A.1]{conrad2010universal}.
\begin{exmp} \label{locallyfree}
Let $X$ be a rigid-analytic space, and $\mathcal{E}$ a locally constant sheaf of $\mathcal{O}_X$-modules. Then the total space functor is representable as a rigid-analytic space over $X$.

Let $\{U_i\}$ be an admissible affinoid covering of $X$ such that $\mathcal{E} \simeq \mathcal{O}_X^r$ on $U_i$. The gluing data for $\mathcal{E}$ extends to gluing data of $Sym(\mathcal{E}^\vee) \simeq Sym(\mathcal{O}_X^r)$, and thus to gluing data of $(\mathbb{A}^{r,an}) \times {U_i}$. This produces a rigid space $\mathcal{T}^{an}$ over $X$; it represents the total space functor because of the classical identity $Hom(Y, \mathbb{A}^{1,an}) = \mathcal{O}(Y)$.
\end{exmp}

\begin{rem}
Kobi Kremnitzer has suggested that the ``correct'' framework for strat-representability is in derived geometry. Indeed, if $X$ is a smooth scheme then any coherent sheaf on it has a finite resolution by locally free sheaves of finite type; in other words, it is dualizable. Then the total space functor is represented by the spectrum of the symmetric algebra of the dual, now interpreted in the derived sense.
\end{rem}

\addcontentsline{toc}{chapter}{Bibliography}
\renewcommand{\bibname}{References}

\bibliographystyle{plain}
\bibliography{RelativeMalcev}{}

\end{document}